\documentclass[10pt,reqno]{amsart}
\pdfoutput=1
\usepackage{mathtools,libertine}
\newcommand\s[1]{$\smash{#1}$}
\catcode`\$=\active
\gdef$#1${\s{#1}}
\let\oldll\ll
\let\oldgg\gg
\usepackage{amssymb,mathabx,mathrsfs,eucal}

\renewcommand{\ll}{\oldll}
\renewcommand{\gg}{\oldgg}
\usepackage{graphicx}
\usepackage[margin=1in]{geometry}
\usepackage{enumerate}
\usepackage{bm}
\usepackage[usenames,dvipsnames]{xcolor}
\usepackage[colorlinks,
	citecolor=teal!70!green!50!darkgray,
	linkcolor=magenta!80!black]{hyperref}
\usepackage[font=footnotesize,justification=centering]{caption}
\usepackage[font=footnotesize,justification=centering]{subcaption}

\newcommand{\beq}{\begin{equation}}
\newcommand{\eeq}{\end{equation}}
\newtheorem{thm}{Theorem}[section]
\newtheorem{lem}[thm]{Lemma}
\newtheorem{cor}[thm]{Corollary}
\newtheorem{ppn}[thm]{Proposition}
\newtheorem{cnd}[thm]{Condition}
\theoremstyle{definition}
\newtheorem{dfn}[thm]{Definition}
\newtheorem{rmk}[thm]{Remark}

\DeclareMathOperator{\Var}{Var}
\DeclareMathOperator{\Cov}{Cov}

\DeclareMathOperator{\diag}{diag}

\DeclareMathOperator{\ch}{ch}
\DeclareMathOperator{\sh}{sh}
\DeclareMathOperator{\tnh}{th}
\DeclareMathOperator{\atnh}{ath}
\DeclareMathOperator{\ash}{arsh}
\DeclareMathOperator{\sgn}{sgn}

\DeclareMathOperator{\Real}{Re}
\DeclareMathOperator{\Imag}{Im}
\DeclareMathOperator{\spn}{span}

\newcommand{\f}{\frac}
\renewcommand{\emptyset}{\varnothing}

\newcommand{\set}[1]{\{#1\}}

\newcommand{\Ind}[1]{\mathbf{1}\{#1\}}
\newcommand{\R}{\mathbb{R}}
\renewcommand{\log}{\ln}

\newcommand{\ent}{\mathcal{H}}
\newcommand{\pd}{\partial}
\newcommand{\sfont}[1]{\textsf{\itshape #1}}
\newcommand{\usf}[1]{\textup{\textsf{#1}}}
\newcommand{\st}{\usf{t}} 
\renewcommand{\sqrt}[1]{(#1)^{1/2}}
\newcommand{\cE}{\mathcal{E}} 
\newcommand{\bPhi}{\Psi} 
\newcommand{\Id}{\textsf{I}} 
\newcommand{\ltwo}[1]{\|#1\|_2}
\newcommand{\linf}[1]{\|#1\|_\infty}
\newcommand{\iii}{\mathfrak{i}} 

\newcommand{\eigmin}{\textsf{e}_{\min}}
\newcommand{\eigmax}{\textsf{e}_{\max}}
\newcommand{\giv}{\hspace{1pt}|\hspace{1pt}}
\newcommand{\givBig}{\hspace{1pt}\Big|\hspace{1pt}}
\newcommand{\givbigg}{\hspace{1pt}\bigg|\hspace{1pt}}

\newcommand{\E}{\mathbb{E}}
\newcommand{\oP}{\mathbb{P}}
\newcommand{\tP}{\check{\mathbb{P}}}
\newcommand{\oQ}{\mathbf{P}}
\newcommand{\tQ}{\check{\mathbf{P}}}

\newcommand{\tQE}{\check{\mathbf{E}}}
\newcommand{\tp}{\check{p}}
\newcommand{\op}{p}
\newcommand{\LDQ}{\bar{\mathbf{P}}}
\newcommand{\LDE}{\bar{\mathbf{E}}}

\newcommand{\bA}{\bm{A}}
\newcommand{\bB}{\bm{B}}
\newcommand{\bC}{\bm{C}}
\newcommand{\bD}{\bm{D}}
\newcommand{\bE}{\bm{E}}
\newcommand{\bG}{\bm{G}}
\newcommand{\bI}{\bm{I}}
\newcommand{\bL}{\bm{L}}
\newcommand{\bM}{\bm{M}}
\newcommand{\bN}{\bm{N}}
\newcommand{\bR}{\bm{R}}
\newcommand{\bT}{\bm{T}}
\newcommand{\bV}{\bm{V}}

\newcommand{\bW}{\bm{W}}
\newcommand{\bX}{\bm{X}}
\newcommand{\bY}{\bm{Y}}
\newcommand{\bGa}{\bm{\Gamma}}
\newcommand{\bXi}{\bm{\Xi}}
\newcommand{\bSi}{\bm{\Sigma}}
\newcommand{\bZ}{\bm{Z}}

\newcommand{\vv}{\sfont{V}}
\newcommand{\row}{\sfont{R}}
\newcommand{\precol}{\sfont{C}}
\newcommand{\col}{\sfont{D}}
\newcommand{\colone}{\sfont{D}_\usf{1}}
\newcommand{\all}{\sfont{RCP}}
\newcommand{\allone}{\sfont{RCP}_\usf{1}}
\newcommand{\con}{\sfont{RC}}
\newcommand{\adm}{\sfont{A}}
\newcommand{\prf}{\sfont{P}}
\newcommand{\qrf}{\sfont{Q}}
\newcommand{\admone}{\adm_\usf{1}}
\newcommand{\admtwo}{\adm_\usf{2}}
\newcommand{\prfone}{\prf_\usf{1}}
\newcommand{\qrfone}{\qrf_\usf{1}}
\newcommand{\prftwo}{\prf_\usf{2}}
\newcommand{\qrftwo}{\qrf_\usf{2}}

\newcommand{\sat}{\sfont{S}}
\newcommand{\bsat}{\sfont{B}}
\newcommand{\bsatone}{\sfont{B}_{\usf{1}}}
\newcommand{\psatone}{\sfont{Q}_{\usf{1}}}

\colorlet{DB}{blue!50!black} 
\colorlet{Re}{red!70!gray} 
\colorlet{Bl}{blue!60!cyan!70!black} 
\colorlet{Pu}{purple!50!blue} 
\colorlet{Gr}{green!50!black} 
\colorlet{Or}{orange!50!yellow!80!red} 

\newcommand{\MVEC}[1]{\mathbf{#1}}
\newcommand{\bc}{\MVEC{c}}
\newcommand{\bh}{\MVEC{h}}
\newcommand{\bn}{\MVEC{n}}

\newcommand{\bx}{\MVEC{x}}

\newcommand{\zroM}{\MVEC{0}}
\newcommand{\oneM}{\MVEC{1}}

\newcommand{\hbe}{\hat{\MVEC{e}}}

\newcommand{\bz}{\MVEC{z}}
\newcommand{\XX}{\MVEC{X}}
\newcommand{\bss}{\MVEC{s}}
\newcommand{\dkappa}{\MVEC{\bm{\kappa}}^\Delta}
\newcommand{\bxi}{{\bm{\xi}}}
\newcommand{\bze}{{\bm{\zeta}}}
\newcommand{\bnu}{{\bm{\nu}}}
\newcommand{\btau}{\bm{\tau}}
\newcommand{\bom}{\bm{\omega}}
\newcommand{\cJ}{{\mathcal{J}}}
\newcommand{\cK}{{\mathcal{K}}}

\newcommand{\NVEC}[1]{\mathbf{#1}}
\newcommand{\bmag}{\NVEC{m}}
\newcommand{\br}{\NVEC{r}}
\newcommand{\dbe}{\dot{\NVEC{e}}}
\newcommand{\by}{\NVEC{y}}
\newcommand{\bj}{\NVEC{j}}
\newcommand{\bk}{\NVEC{k}}
\newcommand{\bv}{\NVEC{v}}
\newcommand{\tbv}{\tilde{\NVEC{v}}}
\newcommand{\bw}{\NVEC{w}}
\newcommand{\bH}{\NVEC{H}}
\newcommand{\oneN}{\NVEC{1}}

\newcommand{\AAA}{\mathscr{A}}
\newcommand{\BB}{\mathscr{B}}
\newcommand{\Filt}{\mathscr{F}}
\newcommand{\bgFilt}{\Filt^\textup{bg}}
\newcommand{\HH}{\mathscr{H}}
\newcommand{\II}{\mathscr{I}}
\newcommand{\KK}{\mathscr{K}}
\newcommand{\LL}{\mathscr{L}}
\newcommand{\MM}{\mathscr{M}}
\newcommand{\GG}{\mathscr{G}}
\newcommand{\PP}{\mathscr{P}}
\newcommand{\QQ}{\mathscr{Q}}
\newcommand{\SE}{\mathscr{S}}
\newcommand{\WW}{\mathscr{W}}

\newcommand{\err}{\usf{err}}
\newcommand{\cst}{\usf{cst}}
\newcommand{\barcst}{\overline{\textsf{cs}}\textsf{t}}
\newcommand{\ta}{\textup{a}}
\newcommand{\tb}{\textup{b}}
\newcommand{\aaa}{\mathfrak{a}}
\newcommand{\bbb}{\mathfrak{b}}
\newcommand{\pp}{\mathfrak{p}}
\newcommand{\sss}{\mathfrak{s}}
\newcommand{\uu}{\mathfrak{u}}
\newcommand{\xx}{\mathfrak{x}}
\newcommand{\zz}{\mathfrak{z}}
\newcommand{\frH}{\mathfrak{H}}

\newcommand{\yy}{\mathfrak{y}}

\newcommand{\bbH}{\mathbb{H}} 
\newcommand{\bq}{q_{\text{\textbullet}}}
\newcommand{\bpsi}{\psi_{\text{\textbullet}}}
\newcommand{\sq}{q_\star}
\newcommand{\spsi}{\psi_\star}
\newcommand{\TAPq}{q_1}
\newcommand{\TAPpsi}{\psi_0}
\newcommand{\Nall}{N_\textup{all}}
\newcommand{\aall}{\alpha_\textup{all}}
\newcommand{\kall}{\kappa_\textup{all}}
\newcommand{\DELTA}{\delta}
\newcommand{\EPS}{\epsilon}
\newcommand{\ETA}{\eta}
\newcommand{\bhJ}{\hat{J}_{\text{\textbullet}}}

\newcommand{\lm}{\lambda}
\newcommand{\lmin}{\lambda_{\min}}
\newcommand{\lmsmall}{l_0}
\newcommand{\lmbig}{l_1}

\newcommand{\alval}{0.833078599}
\newcommand{\auval}{0.833078600}
\newcommand{\qlval}{0.56394907949}
\newcommand{\qluval}{0.56394907950}
\newcommand{\qulval}{0.56394908029}
\newcommand{\quval}{0.56394908030}
\newcommand{\psilval}{2.5763513100}
\newcommand{\psiluval}{2.5763513103}
\newcommand{\psiulval}{2.5763513221}
\newcommand{\psiuval}{2.5763513224}
\newcommand{\lbd}{{\textup{lb}}}
\newcommand{\ubd}{{\textup{ub}}}
\newcommand{\lu}{{\textup{l},\textup{u}}}
\newcommand{\ul}{{\textup{u},\textup{l}}}
\newcommand{\asat}{\alpha_\star}
\newcommand{\albd}{\alpha_\lbd}
\newcommand{\aubd}{\alpha_\ubd}
\newcommand{\qlbd}{q_\lbd}
\newcommand{\qubd}{q_\ubd}
\newcommand{\psilbd}{\psi_\lbd}
\newcommand{\psiubd}{\psi_\ubd}
\newcommand{\gamlbd}{\gamma_\lbd}
\newcommand{\gamubd}{\gamma_\ubd}
\newcommand{\gamlu}{\gamma_\lu}
\newcommand{\gamul}{\gamma_\ul}
\newcommand{\xilbd}{\xi_{\lbd,z}}
\newcommand{\xiubd}{\xi_{\ubd,z}}
\newcommand{\zelbd}{\zeta_{\lbd,z}}
\newcommand{\zeubd}{\zeta_{\ubd,z}}

\newcommand{\HHlbd}{\HH_{\star,\lbd}}

\newcommand{\qul}{q_\ul}
\newcommand{\qlu}{q_\lu}
\newcommand{\psiul}{\psi_\ul}
\newcommand{\psilu}{\psi_\lu}
\newcommand{\Out}{\textup{out}}
\newcommand{\Inn}{\textup{in}}
\newcommand{\llm}{\lm_\lbd}
\newcommand{\ulm}{\lm_\ubd}
\newcommand{\olm}{\lm_\Out}
\newcommand{\ilm}{\lm_\Inn}
\newcommand{\ltau}{\tau_\lbd}
\newcommand{\utau}{\tau_\ubd}
\newcommand{\otau}{\tau_\Out}
\newcommand{\itau}{\tau_\Inn}
\newcommand{\PGG}{\mathscr{S}_{\textit{\texttt{P}}}}
\newcommand{\QGG}{\mathscr{S}_{\textit{\texttt{Q}}}}
\newcommand{\bPGG}{\mathcal{S}_{\textit{\texttt{P}}}}
\newcommand{\bPGGubd}{\mathcal{S}_{\textit{\texttt{P}},\ubd}}
\newcommand{\bQGGubd}{\mathcal{S}_{\textit{\texttt{Q}},\ubd}}

\usepackage{stmaryrd}

\title{Capacity lower bound for the Ising perceptron}
\author[J.\ Ding]{Jian Ding$^\star$}
\author[N.\ Sun]{Nike Sun$^\circ$}
\thanks{$^\star$Wharton Statistics Department, University of Pennsylvania. \newline\indent $^\circ$Mathematics Department, Massachusetts Institute of Technology; Statistics Department, University of California at Berkeley.}

\begin{document}

{\raggedright\begin{abstract} We consider the Ising perceptron with gaussian disorder, which is equivalent to the discrete cube $\{-1,+1\}^N$ intersected by $M$ random half-spaces. The perceptron's \emph{capacity} is $\alpha_N \equiv M_N/N$ for the largest integer $M_N$ such that the intersection in nonempty. It is conjectured by Krauth and M\'ezard (1989) that the (random) ratio $\alpha_N$ converges in probability to an explicit constant $\alpha_\star \doteq 0.83$. Kim and Roche (1998) proved the existence of a positive constant $\gamma$ such that $\gamma\le \alpha_N \le 1-\gamma$ with high probability; see also Talagrand (1999). In this paper we show that the Krauth--M\'ezard conjecture $\alpha_\star$ is a lower bound with positive probability, under the condition that an explicit univariate function $\SE_\star(\lambda)$ is maximized at $\lambda=0$. Our proof is an application of the second moment method to a certain slice of perceptron configurations, as selected by the so-called TAP (Thouless, Anderson, and Palmer, 1977) or AMP (approximate message passing) iteration, whose scaling limit has been characterized by Bayati and Montanari (2011) and Bolthausen (2012). For verifying the condition on $\SE_\star(\lambda)$ we outline one approach, which is implemented in the current version using (nonrigorous) numerical integration packages. In a future version of this paper we intend to complete the verification by implementing a rigorous numerical method.\end{abstract}}

\maketitle

\section{Introduction}\label{s:intro}

\noindent We consider the \emph{Ising perceptron under gaussian disorder}, defined as follows. Let $(g_{\mu,i})_{\mu\ge1,i\ge1}$ be an array of i.i.d.\ 
standard gaussian random variables (zero mean, unit variance). Fix any real number $\kappa$ (our main result is for $\kappa=0$). For any integer $N\ge1$,
we define $M_N\equiv M_N(\kappa)$ to be the largest integer $M$ such that 
	\beq\label{e:perceptron.zerotemp}
	\bigcap_{\mu=1}^M
	\bigg\{J\in\set{-1,+1}^N : 
	\sum_{i=1}^N
	\f{g_{\mu,i} J_i}{N^{1/2}}
	\ge\kappa\bigg\} \ne \emptyset\,.\eeq
Krauth and M\'ezard \cite{krauthmezard1989} conjectured that as $N\uparrow\infty$ the ratio $M_N(\kappa)/N$ converges to an explicit constant $\asat(\kappa)$, which for $\kappa=0$ is roughly $0.83$. This was one of several works in the statistical physics literature analyzing various perceptron models via the ``replica'' or ``cavity'' heuristics \cite{gardner1987maximum, gardner1988space, gd1988optimal, mezard1989space}. In particular, in the variant where $J$ ranges not over $\set{-1,+1}^N$ but over the entire sphere of radius $N^{1/2}$, the analogous threshold was computed by Gardner and Derrida \cite{gd1988optimal}. Another common variation is to take $g_{\mu,i}\in\set{-1,+1}$ to be i.i.d.\ symmetric random signs (\emph{Bernoulli disorder}). The conjectured thresholds differ between the Ising \cite{krauthmezard1989} and spherical \cite{gd1988optimal} models, but do not depend on whether the disorder is Bernoulli or gaussian. While the classical problem is under Bernoulli disorder, we have chosen to work under gaussian disorder to remove some technical difficulties.

For the \emph{spherical} perceptron, very sharp rigorous results have been obtained, including a proof of the predicted threshold for all nonnegative $\kappa$ under Bernoulli disorder \cite{shcherbina2003rigorous}. For the \emph{Ising} perceptron, much less has been proved. One can introduce a parameter $\beta\ge0$ and define the associated positive-temperature partition function $Z_{\kappa,\beta}$ (see \eqref{e:perceptron.Z.beta} below). The replica calculation extends to a prediction \cite{gd1988optimal,krauthmezard1989} for the limit of $N^{-1}\log Z_{\kappa,\beta}$ as $N\uparrow\infty$ with $M/N\to\alpha$. This formula has been proved to be correct at sufficiently high temperature (small $\beta$) under Bernoulli disorder \cite{talagrand2000intersecting}. For the original model \eqref{e:perceptron.zerotemp}, corresponding to zero temperature or $\beta=\infty$, the best rigorous result to date is that for $\kappa=0$ there exists positive $\gamma$ such that $\gamma\le M_N/N \le 1-\gamma$ with high probability \cite{kim1998covering, talagrand1999intersecting} as $N\uparrow\infty$ (see also \cite{stojnic2013discrete} for some work on general $\kappa$). Our main result is the following:

\begin{thm}\label{t:main}
Consider the Ising perceptron at $\kappa=0$ under gaussian disorder. Under Condition~\ref{c:G}, we have
	\[\liminf_{N\uparrow\infty}
	\oP\bigg( \f{M_N}{N} \ge \alpha\bigg)
	>0\]
for any $\alpha<\asat$, where $\asat$ is the prediction of Krauth and M\'ezard \cite{krauthmezard1989}, defined formally by Proposition~\ref{p:gardner}.
\end{thm}

\begin{cnd} \label{c:G} The function $\SE_\star$ defined by \eqref{e:GG.condition}  satisfies $\SE_\star(\lm)<0$ for all $\lm\notin\set{0,1}$.
\end{cnd}

\noindent The proof of Theorem~\ref{t:main} occupies Sections~\ref{s:tap}--\ref{s:conclusion}. An approximate upper bound on $\SE_\star$ is shown by the red filled circles in Figure~\ref{f:HPA}. In Sections~\ref{s:quant.bounds}--\ref{s:specialfns} we outline a verification of Condition~\ref{c:G} using (nonrigorous) numerical integration packages; we intend to replace this with a rigorous verification in a future version. As noted above, the choice of gaussian disorder is a simplification of the Bernoulli disorder case which does not affect the predicted threshold. We expect that the arguments of \cite{kim1998covering, talagrand1999intersecting} can be easily modified to give $\gamma\le M_N/N \le 1-\gamma$ with high probability in this setting, but will not recover the value $\asat$.

We next give the formal definition of the Krauth--M\'ezard threshold $\asat$ for the Ising perceptron. Let us write $\varphi$ for the standard gaussian density, and $\bPhi$ for the complementary gaussian distribution function:
	\[\varphi(x)
	\equiv \f{\exp(-x^2/2)}{\sqrt{2\pi}}\,,\quad
	\bPhi(x)\equiv \int_x^\infty \varphi(u)\,du\,.\]
We give the following expressions for general $\kappa\in\R$ (although we focus on $\kappa=0$). Let
	\beq\label{e:def.F}
	\cE(x)\equiv \f{\varphi(x)}{\bPhi(x)}\,,\quad
	\varphi_\xi(x)\equiv 
	\f{\varphi(x) \Ind{x\ge\xi}}{\bPhi(\xi)}
		\,,\quad
	F_q(x)\equiv\f{\cE}{\sqrt{1-q}}
		\bigg(\f{\kappa-x}{\sqrt{1-q}}\bigg)\,.\eeq
For $q\in[0,1)$ and $\psi\ge0$ define
	\beq\label{e:q.psi.recursion}
	P(\psi)\equiv
	\int\tnh(\psi^{1/2}z)^2\,\varphi(z)\,dz\,,\quad
	R_\kappa(q,\alpha)\equiv
	\alpha\int F_q(q^{1/2}z)^2 \varphi(z)\,dz\,.\eeq
For $q\in[0,1)$ and $\psi\ge0$ let
	\beq\label{e:GG}\GG_\kappa(\alpha,q,\psi)
	\equiv
	-\f{\psi(1-q)}{2}
	+ \int\log(2\ch(\psi^{1/2}z))\varphi(z)\,dz
	+ \alpha\int\log\bPhi\bigg(
	\f{\kappa-q^{1/2}z}{\sqrt{1-q}}
	\bigg)\,\varphi(z)\,dz\,.\eeq
For $\kappa=0$ we abbreviate 
$R(q,\alpha)\equiv R_0(q,\alpha)$ and $\GG(\alpha,q,\psi)\equiv\GG_0(\alpha,q,\psi)$. The following (proved in Section~\ref{s:conclusion}) gives our formal characterization of the threshold $\asat$:

\begin{ppn}\label{p:gardner}
For $\kappa=0$, let $\albd,\aubd,\qlbd,\qubd$ be as defined by \eqref{e:numbers}. For any $\alpha\in(\albd,\aubd)$, it holds that
	\beq\label{e:at}
	\sup_{q\in (\qlbd,\qubd)}
	\bigg\{
	 \f{dP(R(q,\alpha))}{dq}
	 \bigg\}
	 <1\,,\eeq
and there is a unique $\sq\equiv\sq(\alpha)$ in the interval $(\qlbd,\qubd)$ that satisfies the fixed-point equation $\sq = P(R(\sq,\alpha))$. Denote $\spsi(\alpha)\equiv R(\sq(\alpha),\alpha)$. On the interval $(\albd,\aubd)$ the function $\GG_\star(\alpha) \equiv \GG(\alpha,\sq(\alpha),\spsi(\alpha))$ is well-defined and strictly decreasing, with a unique root $\alpha_\star\in(\albd,\aubd)$.
\end{ppn}

\noindent The map $q\mapsto P(R(q,\alpha))$ is shown in Figure~\ref{f:qrecursion} for $\alpha\doteq0.833$. We comment that condition~\eqref{e:at} is essentially the \emph{Almeida--Thouless} (AT) condition \cite{AT1978stability} for this model. We will prove Theorem~\ref{t:main} via the ``second moment method'' (or Paley--Zygmund inequality): for any nonnegative random variable $X$,
	\beq\label{e:paley.zygmund}
	\f{(\E X)^2}{\E(X^2)}
	=\f{ \E(X\Ind{X>0})^2}{\E(X^2)}
	\le \f{\E(X^2) \oP(X>0)}{\E(X^2)}
	=\oP(X>0)\,,\eeq
where the intermediate bound follows from the Cauchy--Schwarz inequality. Now, for $(g_{\mu,i})_{\mu\ge1,i\ge1}$ an array of i.i.d.\ standard gaussians, let $\bG_{M\times N}$ be the submatrix indexed by $1\le \mu\le M$ and $1\le i\le N$.
For the Ising perceptron under gaussian disorder, the partition function at inverse temperature $\beta$ is given by 
	\beq\label{e:perceptron.Z.beta}
	Z_{\kappa,\beta} 
	\equiv Z_{\kappa,\beta}(\bG_{M\times N})
	\equiv\sum_{J\in\set{-1,+1}^N}
	\exp\bigg\{
	\sum_{\mu=1}^M
	-\beta\mathbf{1}\bigg\{
	\sum_{i=1}^N
	\f{g_{\mu,i} J_i}{N^{1/2}}
	<\kappa\bigg\}\bigg\}\,.\eeq
We are interested in the zero-temperature limit, $Z_\kappa\equiv Z_{\kappa,\infty}$. Thus $M_N(\kappa)$ is the largest $M$ for which $Z_\kappa(\bG_{M\times N})$ is nonzero. For the case $\kappa=0$ we abbreviate $Z\equiv Z_0$. In the proof we will introduce parameters $\DELTA,t,\EPS$ where $\DELTA,\EPS$ are positive constants and $t$ is a positive integer. We write $\ETA$ for an error term depending on $M,N,\DELTA,t,\EPS$ such that
	\beq\label{e:ETA}
	\adjustlimits 
	\limsup_{\DELTA\downarrow0}
	\limsup_{t\uparrow\infty}
	\limsup_{\EPS\downarrow 0} 
	\limsup_{\substack{N\uparrow\infty,\\
		M/N\to\alpha}}
		|\ETA_{M,N,\DELTA,t,\EPS}|=0\,.\eeq
If a bound has multiple distinct error terms we indicate this by writing $\ETA$, $\acute{\ETA}$, etc.

\begin{thm}\label{t:second}Consider the Ising perceptron at $\kappa=0$ under gaussian disorder, with partition function $Z\equiv Z(\bG_{M\times N})$. We can define a collection of (integer-valued) random variables $\bZ_{\DELTA,t,\EPS}\equiv \bZ_{\DELTA,t,\EPS}(\bG_{M\times N})$,
and a $\sigma$-field $\Filt\equiv\Filt_{M,N,\DELTA,t}$, such that the following hold: $0\le\bZ_{\DELTA,t,\EPS}\le Z$ with probability one; we have
	\beq\label{e:first.mmt.gg}\adjustlimits
	\liminf_{\DELTA\downarrow0}
	\liminf_{t\uparrow\infty}
	\liminf_{\substack{N\uparrow\infty,\\
		M/N\to\alpha}}
	\oP\bigg(
	\f1N\log\E\Big(
	\bZ_{\DELTA,t,\EPS}
	\givBig\Filt\Big)
	\ge\GG_\star(\alpha)- \ETA
	\bigg)>0\eeq
for all $\alpha \in (\albd,\aubd)$;
and under Condition~\ref{c:G} there exists
$\alpha'\in(\albd,\asat)$ such that for all $\alpha\in(\alpha',\asat)$ we have
	\beq\label{e:second.mmt.bound}
	\E\Big( (\bZ_{\DELTA,t,\EPS})^2\givBig\Filt\Big)
	\le \cst
	\bigg\{\E\Big(\bZ_{\DELTA,t,\EPS}\givBig\Filt\Big)
	\bigg\}^2 + \bigg(
	\exp\bigg\{
	2 N \GG_\star(\alpha)
	-\f{N}{\cst}
	\bigg\}
	+\exp\bigg\{
	N \GG_\star(\alpha) 
	\bigg\}\bigg)
	\exp(N\acute{\ETA})\eeq
where $\cst \equiv\cst_{M,N,\DELTA,t,\EPS}$ are $\Filt_{M,N,\DELTA,t}$-measurable and stochastically bounded in the limit $N\uparrow\infty$ with $M/N\to\alpha$.
\end{thm}

\noindent The chief innovation of this paper is the design of the conditioning $\sigma$-field $\Filt$ and the random variables $\bZ_{\DELTA,t,\EPS}$, given in  Section~\ref{s:tap} (with $\bZ_{\DELTA,t,\EPS}$ defined explicitly by \eqref{e:def.restricted.part.fn}). 
The main technical work is in the conditional moment analysis for $\bZ_{\DELTA,t,\EPS}$, which occupies Sections~\ref{s:tilt}--\ref{s:conclusion}. We will see that $\bZ_{\DELTA,t,\EPS}(\bG_{M\times N})$ in fact depends on $\bG_{M\times N}$ together with an extra small random perturbation $\dkappa$ (appearing in \eqref{e:fixed.pt}). The role of $\dkappa$ is purely technical: it smooths a certain distribution and can only decrease the partition function, so it has no effect on the main result. We include $\dkappa$ in the definition of $\Filt$, but otherwise will often suppress it from the notation.

\begin{proof}[Proof of Theorem~\ref{t:main} assuming Theorem~\ref{t:second}] We have $M_N \ge N\alpha$ if and only if $Z(\bG_{\lfloor N\alpha \rfloor \times N})$ is positive, so
	\[\oP\bigg(\f{M_N}{N} \ge \alpha \,\bigg|\,
		\Filt\bigg)
	=\oP\Big(
	Z(\bG_{\lfloor N\alpha \rfloor \times N})
	>0
	\,\Big|\,\Filt
	\Big)
	\ge\oP\Big(
		\bZ_{\DELTA,t,\EPS}
		(\bG_{\lfloor N\alpha \rfloor \times N})
		>0
		\,\Big|\,\Filt\Big)\,.\]
Abbreviate $\bZ_{\DELTA,t,\EPS}\equiv \bZ_{\DELTA,t,\EPS}(\bG_{\lfloor N\alpha \rfloor \times N})$. For $\alpha'<\alpha<\asat$ we have $\GG_\star(\alpha)>0$, so we can choose $\DELTA,t,\EPS$ such that
	\beq\label{e:first.mmt.large.evt}
	\liminf_{\substack{
		N\uparrow\infty,\\
		M/N\to\alpha
		}}\bigg\{
	\min\bigg\{
	\GG_\star(\alpha),\f1{\cst}
	\bigg\}-2 \ETA_{M,N,\DELTA,t,\EPS}
	-\acute{\ETA}_{M,N,\DELTA,t,\EPS}
	\bigg\}>0\eeq
for $\ETA$ as in \eqref{e:first.mmt.gg} and $\acute{\ETA}$ as in \eqref{e:second.mmt.bound}. On the event $N^{-1}\log\E(\bZ_{\DELTA,t,\EPS}\givBig\Filt) \ge \GG_\star(\alpha)-\ETA$, combining with \eqref{e:paley.zygmund} and \eqref{e:second.mmt.bound} gives
	\[\oP\Big(
	\bZ_{\DELTA,t,\EPS}>0
	\,\Big|\,\Filt\Big)
	\ge \f1{\cst +
	\exp(N(2\ETA+\acute{\ETA}))
	[ \exp(-N/\cst)+\exp(-N\GG_\star(\alpha)) ]
	} = \f1{\cst+o_N(1)}\]
where the term $o_N(1)$ tends to zero in the limit $N\to\infty$, with $M/N\to\alpha$, for any choice of $\DELTA,t,\EPS$ such that \eqref{e:first.mmt.large.evt} holds.
Combining with \eqref{e:first.mmt.gg} gives 
	\begin{align*}
	&\liminf_{N\uparrow\infty}
	\oP\bigg(\f{M_N}{N}\ge \alpha\bigg)
	\ge\liminf_{N\uparrow\infty}
	\oP(\bZ_{\DELTA,t,\EPS}
		(\bG_{\lfloor N\alpha
			\rfloor \times N})>0)\\
	&\qquad\ge
	\liminf_{N\uparrow\infty} \E\bigg(
		\f1{\cst + o_N(1)}
	\mathbf{1}\bigg\{
	\f1N\log\E\Big(
	\bZ_{\DELTA,t,\EPS}
		(\bG_{\lfloor N\alpha\rfloor\times N})
	\givBig\Filt\Big)
	\ge\GG_\star(\alpha)-\ETA
	\bigg\}\bigg)>0\,,\end{align*}
concluding the proof.\end{proof}

\noindent The proof of Theorem~\ref{t:second} occupies essentially the entirety of this paper. In Section~\ref{s:tap} we define the $\sigma$-fields $\Filt$ and the random variables $\bZ_{\DELTA,t,\EPS}$, and give the proof outline which is then implemented in the remainder of the paper. 

\begin{rmk}\label{r:kappa} Although our main result is for $\kappa=0$, this assumption is used only in a few steps which will be explicitly indicated. Otherwise we write most steps of the proof for general $\kappa\in\R$, assuming only that we have a fixed point $\sq = P(R_\kappa(\sq,\alpha))$ such that condition~\eqref{e:at} is satisfied for a range $(\qlbd,\qubd)\ni \sq$. We indicate the dependence on $\kappa,\alpha$ by writing $\sq\equiv q_{\star\kappa}(\alpha)$, $\spsi\equiv \psi_{\star\kappa}(\alpha) \equiv R_\kappa(\sq,\alpha)$, and $\GG_{\star\kappa}(\alpha)\equiv \GG_\kappa(\alpha,q_{\star\kappa}(\alpha),\psi_{\star\kappa}(\alpha))$.
\end{rmk}

\begin{rmk}\label{r:sk} A closely related conditional second moment approach is implemented in an independent work \cite{bolthausensk} to compute the free energy of the Sherrington--Kirkpatrick (SK) spin-glass model \cite{sk} (allowing an external field) at high temperature. For SK and some related spin-glass models, a very powerful framework has been extensively developed (see \cite{MR1930572,MR1894111,MR2195134,MR2731561,MR3052333} and refs.\ therein) that extends to the more difficult low-temperature regime; but the approach of \cite{bolthausensk} offers an appealing alternative at high temperature. We point out that \cite{bolthausensk} computes conditional moments of the \emph{unrestricted} SK partition function, yielding tight lower \emph{and upper} bounds, but again only at very high temperature. The main challenge of the current paper is to prove an analogous lower bound as \cite{bolthausensk}, but at zero temperature where the second moment method tolerates much less error, and furthermore for a model with a more complicated (nonlinear) Hamiltonian.
\end{rmk}

\noindent A matching upper bound to Theorem~\ref{t:main} remains for us the most natural and interesting open question. Beyond this, we refer to intriguing experimental investigations \cite{baldassi2016local, baldassi2016unreasonable} which suggest further avenues for investigation in the Ising perceptron model.

\subsection*{Acknowledgements} We are extremely grateful to Andrea Montanari who generously discussed this problem with us on many occasions, and shared ideas that became essential to the proof. We also wish to thank Erwin Bolthausen for sharing with us the manuscripts of his related work. Riccardo Zecchina, Carlo Baldassi, and Lenka Zdeborov\'a introduced us to this problem, and we are grateful for their encouragement to work on it. The perceptron model was also brought up at an American Institute of Mathematics workshop in June 2017, and we thank the other participants in the discussions there: Dimitris Achlioptas, Nick Cook, Reza Gheissari, Aukosh Jagannath, Florent Krzakala, Will Perkins, Eliran Subag, and Yumeng Zhang. Finally, it is a pleasure to acknowledge the hospitality of our colleagues at the National University of Singapore and at the Centre de Recherches Math\'ematiques, where parts of this work were completed. J.D.\ is supported by NSF grant DMS-1757479 and an Alfred Sloan fellowship. N.S.\ is supported by NSF grant DMS-1752728.

\section{TAP iteration and conditioning scheme}
\label{s:tap}

\noindent In the proof of Theorem~\ref{t:second}, a small fraction of the columns of $\bG$ play a special role, and will be fully revealed in the preliminary setup. The subsequent moment calculation (occupying most of the paper) is based on the randomness in the remaining majority of columns, which are only partially revealed in the preliminary phase. For this reason it is convenient to slightly adjust the notation: we recast $N$ as $\Nall$, and instead use $N$ for the portion of $\bG$ that plays the main role in the second moment. To be precise, for the remainder of the paper we let $M=\lfloor \Nall \aall \rfloor$ and 
	\beq\label{e:g.kr}
	\bG= \bG_{M\times \Nall}
	= \begin{pmatrix} 
	g_{1,1} & \cdots & g_{1,N}
		& g_{1,N+1}
		& \cdots & g_{1,\Nall}\\
	\vdots& & \vdots
		& \vdots &&\vdots \\
	g_{M,1} & \cdots & g_{M,N}
		& g_{M,N+1}
		& \cdots & g_{M,\Nall}\\
	\end{pmatrix}
	= \begin{pmatrix} 
	\bE & \hat{\bE}
	\end{pmatrix}\eeq
where $\bE$ is $M\times N$, $\hat{\bE}$ is $M\times \hat{N}$, and $\Nall = N+\hat{N}$ with $\hat{N}/N = \hat{\alpha}(\DELTA)\to0$ as $\DELTA\downarrow0$. We write 
	\[J_\textup{all}
	\equiv \begin{pmatrix} J \\ \hat{J}\end{pmatrix}
	\in\set{-1,+1}^{\Nall}\]
where $J\in\set{-1,+1}^N$ and $\hat{J}\in\set{-1,+1}^{\hat{N}}$. We then also relabel $\kappa$ as $\kall$, so that the perceptron partition function $Z_{\kall}(\bG)$ counts elements $J_\textup{all} \in \set{-1,+1}^{\Nall}$ satisfying 
	\beq\label{e:perceptron.sat.all}	
	\f{\bG J_\textup{all}}
		{\sqrt{\Nall}}
	= \f{\bE J + \hat{\bE}\hat{J}}
		{\sqrt{\Nall}}
	\ge\kall\,.\eeq
For most of the proof it will be more convenient to normalize by $N^{1/2}$ rather than $\sqrt{\Nall}$. We therefore let
	\[\alpha \equiv \f{M}{N}
	= \f{\aall-o_N(1)}
		{1-\hat{\alpha}(\DELTA)}\,,\quad
	\kappa \equiv 
	\kall \bigg(
	\f{\Nall}{N}\bigg)^{1/2}
	= \f{\kall}
	{\sqrt{1- \hat{\alpha}(\DELTA)}}\]
(so for our main result $\kappa=\kall=0$). In this section we formally define the $\sigma$-field $\Filt$ of Theorem~\ref{t:second}, and prove some results in preparation for the second moment analysis.

\subsection{TAP iteration and state evolution}\label{ss:tap} The conditioning $\sigma$-field $\Filt$ of Theorem~\ref{t:second} is based on the so-called TAP (Thouless--Anderson--Palmer \cite{tap1977solution}) or AMP (approximate message passing) iteration, which we now review. 

\begin{rmk}\label{r:fixed.pt.adjusted} 
To restate Remark~\ref{r:kappa} in our new notation,
we assume that we have 
$\sq \equiv q_{\star\kall}(\aall)$ satisfying the fixed-point equation $\sq = P(R_{\kall}(\sq,\aall))$, such that
for some range $(\qlbd,\qubd)\ni\sq$ we have
	\beq\label{e:at.all}
	\sup_{q\in (\qlbd,\qubd)}
	\bigg\{
	 \f{dP(R_{\kall}
	 	(q,\aall))}{dq}
	 \bigg\}
	 <1\,.\eeq
Write $\spsi\equiv\psi_{\star\kall}(\aall)\equiv R_{\kall}(\sq,\aall)$. We will arrange (in Proposition~\ref{p:kr}) to have $\hat{N}/N=\hat{\alpha}(\DELTA)=o_{1/\DELTA}(1)$, where we use $o_{1/\DELTA}(1)$ to indicate a quantity that tends to zero in the limit $\DELTA\downarrow0$. As a result, by continuity considerations, for all sufficiently small $\DELTA$ there will be a value $\TAPq \equiv q_{\star\kappa}(\alpha)= \sq + o_{1/\DELTA}(1)$
satisfying $\TAPq=P(R_\kappa(\TAPq,\alpha))$
and 
	\beq\label{e:at.N}
	\sup_{q\in (\qlbd,\qubd)}
	\bigg\{
	\f{dP(R_\kappa(q,\alpha))}{dq}
	\bigg\}
	<1\eeq
(cf.~\eqref{e:at}~and~\eqref{e:at.all}).
This assumption holds for the rest of the paper, even when not explicitly stated. Let $\TAPpsi\equiv R_\kappa(\TAPq,\alpha)$.
\end{rmk}

\noindent For $\ell\ge1$, $\bx\in\R^\ell$ and $f:\R\to\R$ we write $f(\bx)$ for the vector obtained by coordinatewise application of $f$, that is, $f(\bx)\equiv (f(x_i))_{i\le \ell}\in\R^\ell$. Let $F\equiv F_{\TAPq}$ as defined by \eqref{e:def.F}. Initialize $\bn^{(0)}\equiv\zroM\in\R^M$ and $\bmag^{(1)}\equiv \sqrt{\TAPq}\oneN\in\R^N$. Then use the TAP equations
	\begin{alignat}{4}
	\label{e:tap.n}
	\bn^{(s)}
	\equiv F(\bh^{(s)})
	&\equiv F\bigg(
		\f{\bE \bmag^{(s)}}{N^{1/2}}
		-b^{(s)} \bn^{(s-1)}
		\bigg)\,,\quad
	&& b^{(s)} &&= \f{(\oneN,\tnh'(\bH^{(s)}))}{N}
		= 1-\f{\ltwo{\bmag^{(s)}}^2}{N}\,,\\
	\bmag^{(s+1)}
	\equiv \tnh(\bH^{(s+1)})
	&\equiv \tnh\bigg(
	\f{\bE^\st \bn^{(s)}}{N^{1/2}}- d^{(s)} \bmag^{(s)}
	\bigg)\,,\quad
	&& d^{(s)} &&= \f{(\oneM,F'(\bh^{(s)}))}{N}\,,
	\label{e:tap.m}
	\end{alignat}
to define the sequence $\bn^{(1)},\bmag^{(2)},\bn^{(2)},\ldots,\bmag^{(t)},\bn^{(t)}$. For all $s\ge1$ denote $q_s\equiv \ltwo{\bmag^{(s)}}^2/N$ and $\psi_s\equiv\ltwo{\bn^{(s)}}^2/N$, and note that $b^{(s)}=1-q_s$. We hereafter abbreviate $\bmag\equiv\bmag^{(t)}$, $\bH\equiv\bH^{(t)}$, and $\bq\equiv q_t$.
 
\begin{rmk}\label{r:tap} For a bounded number $t$ of iterations, the $N\uparrow\infty$ distributional limit of TAP has been rigorously characterized in terms of a ``state evolution'' recursion \cite{bayati2011dynamics,bolthausencmp}. For a special case of AMP, finite-sample results were obtained more recently, allowing even for $t$ growing slowly with $N$ \cite{rush2016finite}. In this work we only require some results from the earlier works \cite{bayati2011dynamics, bolthausencmp}, which we informally summarize as follows:
\begin{enumerate}[a.]
\item\label{r:tap.a} For large $N$ and large $t$ the vectors $\bn^{(t)}\equiv F(\bh^{(t)})$ and $\bn^{(t-1)}\equiv F(\bh^{(t-1)})$ are close in $\ell^2$, and likewise the vectors $\bmag^{(t)}\equiv\tnh(\bH^{(t)})$ and $\bmag^{(t-1)}\equiv\tnh(\bH^{(t-1)})$ are close in $\ell^2$.
\item\label{r:tap.b} For any $s$, the empirical profile of $\bh^{(s)}$ resembles a gaussian distribution with variance $\TAPq$, while the empirical profile of $\bH^{(s)}$ resembles a gaussian distribution with variance $\TAPpsi$, where $\TAPq,\TAPpsi$ are as defined by Remark~\ref{r:fixed.pt.adjusted}.
\item\label{r:tap.c} For any fixed $t$, the matrix of inner products among $\bh^{(1)},\ldots,\bh^{(t)}$ converges to a nondegenerate limit as $N\uparrow\infty$, as does the matrix of inner products among $\bH^{(1)},\ldots,\bH^{(t)}$.
\end{enumerate}
The formal statements will be reviewed below as required.\end{rmk}

\noindent With $\bG_{M\times \Nall}$ decomposed as in \eqref{e:g.kr}, we first prove the following:

\begin{ppn}\label{p:kr} For any small positive $\DELTA$ there is a decomposition \eqref{e:g.kr} with $\hat{N}/N = \hat{\alpha}(\DELTA) \le \DELTA/\exp(1/\DELTA^{2/3})$, and a large enough constant $t(\DELTA)$, such that the following holds: if $\bmag\equiv\bmag^{(t)}$ is defined by $t$ iterations of the TAP equations on $\bE$
for $t\ge t(\DELTA)$, then with probability at least $1/10$ there exists $\hat{J} \in\set{-1,+1}^{\hat{N}}$ satisfying
	\beq\label{e:hJ.greater.than.kappa}
	\f{1}{N^{1/2}}
	\bG_{M\times \Nall}
	\begin{pmatrix}
	\bmag\\\hat{J}
	\end{pmatrix}
	=\f{\bE\bmag+\hat{\bE}\hat{J}}
		{N^{1/2}}
	\ge \Big(
	\kappa + 2\DELTA^{1/2}\Big)
	\oneM\eeq
coordinatewise.
\end{ppn}

\noindent The proof of Proposition~\ref{p:kr} is an adaptation of the argument of \cite{kim1998covering}, and is deferred to Section~\ref{ss:kr}. Its purpose is explained by the following lemma:

\begin{lem}\label{l:kr}
For any $y>\kappa$ there is a unique solution $h\in\R$ to the equation
$h+(1-\bq)F_{\bq}(h) = y$.

\begin{proof} We rely on some basic properties of the function $\cE(x)$ given in Lemma~\ref{l:EE}. 
Consider the function
	\[L_{\bq}(h)
	\equiv h+(1-\bq)F_{\bq}(h)
	= h+\sqrt{1-\bq}\cE\bigg(
		\f{\kappa-h}{\sqrt{1-\bq}}
		\bigg) \ge\kappa\]
where the last inequality holds because $\cE(x)\ge x$ for all $x\in\R$. Since $\cE(x)$ is always positive, we have $L_{\bq}(h)> h$, therefore $L_{\bq}(h)\to\infty$ as $h\to\infty$. On the other hand, as $x\to\infty$ we have $\cE(x)-x \asymp 1/x$, which implies that as $h\to-\infty$ we have $L_{\bq}(h)-\kappa \asymp 1/|h|$. Finally, since $\cE'(x)\in(0,1)$ for all $x\in\R$, we have $dL_{\bq}/dh \in(0,1)$ for all $h\in\R$. It follows that $h\mapsto L_{\bq}(h)$ is a strictly increasing map from $\R$ \emph{onto} $(\kappa,\infty)$, so a unique solution to the equation $L_{\bq}(h) = y$ exists provided $y>\kappa$.
\end{proof}
\end{lem}

\noindent Define $\bhJ$ to be the lexicographically minimal element of $\set{-1,+1}^{\hat{N}}$ satisfying
	\beq\label{e:def.bhJ}
	\f1{N^{1/2}}\bG_{M\times \Nall}
	\begin{pmatrix}\bmag \\ \bhJ\end{pmatrix}
	=
	\f{\bE\bmag +\hat{\bE}\bhJ}{N^{1/2}}
	\ge\Big(
	\kappa + 2\DELTA^{1/2}\Big)\oneM\eeq
coordinatewise. Let $\dkappa$ be an independent random vector sampled uniformly from the cube $[0,\DELTA/\exp(1/\DELTA^2)]^M$, and solve for $\bh\in\R^N$ such that
	\beq\label{e:fixed.pt}
	\f{\bE\bmag + \hat{\bE}\bhJ}{N^{1/2}}
	-\dkappa
	= \bh + (1-\bq)F_{\bq}(\bh) \,.\eeq
Note that \eqref{e:fixed.pt} is equivalent to
	\beq\label{e:fixed.pt.n}
	\bn \equiv F_{\bq}(\bh)
	= F_{\bq}\bigg( \f{\bE\bmag
	+ \hat{\bE}\bhJ}{N^{1/2}}
		-\dkappa
		-(1-\bq) \bn \bigg)\,.\eeq
We recognize \eqref{e:fixed.pt.n} as the TAP equation \eqref{e:tap.n} with a perturbation that will have an essential role in the proof. Let
	\beq\label{eq-def-H-circ}
	\bmag^{(t+1)}
	\equiv \tnh(\bH^{(t+1)})
	\equiv\tnh\bigg( \f{\bE^\st\bn}{N^{1/2}}
		- d^{(t)} \bmag\bigg)\,.\eeq
We emphasize that the construction of $\bmag^{(t+1)}$ differs from that of the previous $\bmag^{(s)}$, since we have passed through the perturbed equation
\eqref{e:fixed.pt.n}. We continue to denote $q_s\equiv\ltwo{\bmag^{(s)}}^2/N$ and $\bq\equiv q_t$, and we also let $\bpsi \equiv \ltwo{\bn}^2/N$. The conditioning $\sigma$-field in Theorem~\ref{t:second} is given by $\Filt\equiv\Filt_{M,\Nall,\DELTA,t,\EPS}\equiv
\sigma(\textsf{DATA}_{M,\Nall,\DELTA,t})$ where
	\beq\label{e:filt.t}
	\textsf{DATA}\equiv
	\textsf{DATA}_{M,\Nall,\DELTA,t}
	\equiv \Big(\bn^{(1)},\bmag^{(2)},\bn^{(2)},
	\ldots,\bmag^{(t)}\equiv\bmag,\bn^{(t)},
	\bmag^{(t+1)}, \hat{\bE},
	\dkappa\Big)\,.\eeq
Note that $\Filt$ is contained in $\sigma(\bG_{M\times \Nall},\dkappa)$.

\begin{rmk}\label{r:cst}
Throughout this paper we write $\cst\equiv\cst_{M,\Nall,\DELTA,t,\EPS}$ to denote a collection of random variables that is measurable with respect to
$\Filt_{M,\Nall,\DELTA,t}$, and remains
stochastically bounded as $\Nall\uparrow\infty$
with $M/\Nall\to\alpha$ for fixed $\DELTA,t,\EPS$: that is to say,
	\[\limsup_{\substack{\Nall\uparrow\infty,\\
		M/\Nall \to\alpha}}
	\oP\Big( \cst_{M,\Nall,\DELTA,t,\EPS} \ge C\Big) 
	\downarrow0\]
as $C\uparrow\infty$, for any $\alpha\in(\albd,\aubd)$. In our usage, the value of $\cst$ may change from one occurrence to the next, as long as the stochastic boundedness is maintained. If a result depends on multiple choices of $\cst$ simultaneously we will indicate this by writing $\cst$, $\cst'$, and so on (as in the proof of Corollary~\ref{c:adm.density.largerevt}). To indicate dependence on any other parameter, say $\gamma$, we shall write $\cst_\gamma\equiv \cst_{M,\Nall,\DELTA,t,\EPS,\gamma}$.
\end{rmk}

\subsection{Restricted partition function}
\label{sec-restricted-partition-function}

We next make a convenient change of coordinates. We let $(\br^{(s)})_{1\le s\le t}$ be an orthonormal basis for $\spn((\bmag^{(s)})_{1\le s\le t-1},\bmag)$ where $\bmag\equiv\bmag^{(t)}$. Likewise let $(\bc^{(s)})_{1\le s\le t}$ be an orthonormal basis for $\spn((\bn^{(s)})_{1\le s\le t-1},\bn)$ where $\bn$ is the solution of \eqref{e:fixed.pt}. From now on we specify
	\beq\label{e:basis.fixed}
	\br^{(t)}\equiv \f{\bmag}{\sqrt{N\bq}}\,,\quad
	\bc^{(t)}\equiv\f{\bn}{\sqrt{N\bpsi}}\,.\eeq
Explicitly, define the $N\times t$ matrix
	\beq\label{e:bM.before.change.basis}
	\bM\equiv\begin{pmatrix}
	\displaystyle
	\f{\bmag}{\sqrt{N\bq}}
	&\displaystyle\f{\bmag^{(t-1)}}
		{\sqrt{Nq_{t-1}}}
	&\cdots
	&\displaystyle
		\f{\bmag^{(1)}}{\sqrt{N\TAPq}}
	\end{pmatrix}\,,\eeq
and apply the Gram--Schmidt procedure to obtain $\bM=\bR\bGa$ where $\bR$ is $N\times t$ with $\bR^\st\bR=\Id_{t\times t}$ while $\bGa$ is the $t\times t$ matrix containing the change-of-basis coefficients. (In the usual notation of QR factorization, $\bR$ corresponds to ``Q'' while $\bGa$ corresponds to ``R.'') It follows from \cite[Lemma 1(a)]{bayati2011dynamics} that as $N\uparrow\infty$ with $t$ fixed, we have $\bGa$ converging to a constant matrix; and by \cite[Lemma 1(g)]{bayati2011dynamics} the limiting matrix is invertible (cf.\ Remark~\ref{r:tap}\ref{r:tap.c}). The columns of $\bR=\bM\bGa^{-1}$ give the desired orthonormal basis $(\br^{(t)},\ldots,\br^{(1)})$. To define $\bc^{(s)}$ we instead consider the $M\times t$ matrix
	\beq\label{e:bN.before.change.basis}
	\bN
	\equiv
	\begin{pmatrix}
	\displaystyle\f{\bn}{\sqrt{N\bpsi}} &
	\displaystyle\f{\bn^{(t-1)}}
		{\sqrt{N\psi_{t-1}}}
	& \cdots &
	\displaystyle 
		\f{\bn^{(1)}}
		{\sqrt{N\psi_1}}
	\end{pmatrix}\,,\eeq
and obtain the QR factorization $\bN=\bC\bXi$ where $\bC$ is $M\times t$ with $\bC^\st\bC=\Id_{t\times t}$ while $\bXi$ is the $t\times t$ change-of-basis matrix. Since $\bn$ is obtained by \eqref{e:fixed.pt.n} rather than by the TAP iteration, the result of \cite{bayati2011dynamics} does not give convergence of $\bXi$ in the limit $N\uparrow\infty$. Nevertheless, it follows from the additional randomness introduced by $\dkappa$ that $\bXi$ is bounded and nondegenerate, in the sense that all its eigenvalues are bounded away from zero and infinity in the limit $N\uparrow\infty$. We will assume this fact for now, deferring the formal proof to Corollary~\ref{c:Xi.nondegenerate}. Note that both sets of basis vectors, $(\br^{(s)})_{1\le s\le t}$ and $(\bc^{(s)})_{1\le s\le t}$, are measurable with respect to the $\sigma$-field $\Filt
\equiv\sigma(\textsf{DATA})$ as defined by \eqref{e:filt.t}.

We now define the restricted partition function $\bZ_{\DELTA,t,\EPS}(\bG_{M\times \Nall}) \le Z_{\kall}(\bG_{M\times \Nall})$. The general idea is to first restrict to a certain (affine) slice of the discrete cube selected by the TAP iteration, then to restrict the perceptron satisfiability condition \eqref{e:perceptron.sat.all} by imposing additional constraints on the vector $\bG J_\textup{all}$. The details are as follows. Let $\DELTA>0$, and decompose $\bG_{M\times \Nall}$ as in \eqref{e:g.kr} with $\hat{N}/N=\hat{\alpha}(\DELTA)$. Let $t\ge t(\DELTA)$ as specified by Proposition~\ref{p:kr}. On the $M\times N$ matrix $\bE$, run the TAP equations \eqref{e:tap.n}~and~\eqref{e:tap.m} for $t$ iterations; then define $\bhJ$ to be the lexicographically minimal element of $\set{-1,+1}^{\hat{N}}$ satisfying \eqref{e:def.bhJ}. From now on we only consider spin configurations of the form $J_\textup{all} = (J,\bhJ)$. If there exists no $\bhJ$ satisfying \eqref{e:def.bhJ} then we simply set $\bZ_{\DELTA,t,\EPS}\equiv0$. We also restrict to elements $J\in\set{-1,+1}^N$ which ``resemble samples from $\bmag$'' in the following manner:

\begin{dfn}[restrictions in discrete cube]
\label{d:cube} Conditional on $\Filt_{M,\Nall,\DELTA,t}$, we define $\bbH_{M,\Nall,\DELTA,t}$ to be the set of spin configurations $J\in\set{-1,+1}^N$ such that the orthogonal projection of $J$ onto the span of the vectors $\bmag^{(1)},\ldots,\bmag^{(t)}$ and $\bH^{(2)},\ldots,\bH^{(t)},\bH^{(t+1)}$ is very close to $\bmag^{(t)}\equiv\bmag$. Formally, we fix a positive absolute constant $C$ and say that $J\in\bbH_{M,\Nall,\DELTA,t}$ if $\bj\equiv J/N^{1/2}$ can be decomposed as
	\beq\label{e:restrict.j} 
	\bj = \sqrt{q_{J,t}}\br^{(t)}
		+\sum_{s=1}^{t-1}
		\gamma_{J,s}\br^{(s)}
	 + \sqrt{1-q_J} \bv\eeq
such that $|\gamma_{J,s}|\le C/N$ for all $1\le s\le t-1$, $|q_J-\bq|\le C/N$, and $\bv$ is a unit vector which is exactly orthogonal to $\spn(\bmag^{(1)},\ldots,\bmag^{(t)})$, and also is nearly orthogonal to $\spn(\bH^{(2)},\ldots,\bH^{(t+1)})$ in the sense that
	\beq\label{e:restrict.j.appx}
	\max\bigg\{ |(\bv,\bH^{(1)})|\,,
	\ldots\,,|(\bv,\bH^{(t+1)})|\bigg\}
	\le \f{C}{N^{1/2}}\,.\eeq
Assuming that $1/\EPS$ is a positive integer, we let $\bbH_\EPS\equiv\bbH_{M,\Nall,\DELTA,t,\EPS}$ be the subset of $J\in\bbH_{M,\Nall,\DELTA,t}$ which additionally satisfy that for each integer $1\le\ell\le1/\EPS$ we have
	\[\bigg|\f{\#\set{i : (\ell-1)\EPS 
	\le m_i \le \ell\EPS, J_i=1}}
	{\#\set{i : (\ell-1)\EPS \le m_i \le \ell\EPS}}
	-\ell\EPS\bigg|
	\le2\EPS\,,\]
where $m_i$ is the $i$-th entry of the vector $\bmag$ from \eqref{e:fixed.pt}. We shall abbreviate $\bbH\equiv\bbH_{M,\Nall,\DELTA,t}$ and $\bbH_\EPS\equiv\bbH_{M,\Nall,\DELTA,t,\EPS}$.
\end{dfn}

\noindent We then restrict the satisfiability event as follows:

\begin{dfn}[profile truncation]\label{d:profile}
For $J\in\set{-1,+1}^N$ we define the basic satisfiability event as
	\beq\label{e:basic.sat.evt}
	\sat_{J,\DELTA,t}\equiv \bigg\{
	\f{\bE J + \hat{\bE}\bhJ}{N^{1/2}}
	\ge\kappa\oneM
	+\dkappa\bigg\}\,.\eeq
(Note $\sat_{J,\DELTA,t}$ implies $\bG J_\textup{all}/\sqrt{\Nall}\ge\kappa\oneM$.) We then define a more restricted event $\sat_{J,\DELTA,t,\EPS}\subseteq\sat_{J,\DELTA,t}$ as follows. Let 
	\beq\label{e:def.zeta.J}
	\bxi\equiv
	\f{\kappa\oneM-\bh}{\sqrt{1-\bq}}\,,\quad
	\bze^J\equiv
	\f{\sqrt{q_{J,t}}(1-\bq)\bn}
	{\sqrt{\bq}\sqrt{1-q_J}}\,,\quad
	\bnu\equiv \bE\bv+\bze^J\,.\eeq
Assuming that $1/\EPS$ is a positive integer, let $\sat_{J,\DELTA,t,\EPS}$ be the event that
\begin{enumerate}[(i)]
\item\label{d:profile.i} $\sat_{J,\DELTA,t}$ occurs;
\item\label{d:profile.ii} The vector $\bnu$ satisfies the empirical moment bound
	\[\f1N\sum_{\mu=1}^M(\nu_\mu)^{20}
	\le \f{2}{N}\sum_{\mu=1}^M
		\int \nu^{20}\,
		\varphi_{\xi_\mu}(\nu)\,d\nu\,;\]
\item\label{d:profile.iii} For each pair of integers
$1\le\ell,\ell'\le1/\EPS$ we have
	\[\bigg|
	\f{
	\#\set{\mu : (\ell-1)\EPS
	\le \bPhi(h_\mu/\sqrt{\bq})
	\le \ell\EPS,
	(\ell'-1)\EPS\le\bPhi(\nu_\mu)/\bPhi(\xi_\mu)
	\le \ell'\EPS
	}}
	{
	\#\set{\mu : (\ell-1)\EPS
	\le \bPhi(h_\mu/\sqrt{\bq}) \le \ell\EPS
	}} - \EPS
	\bigg|\le \EPS^{10}\,.\]
\end{enumerate}
(Informally speaking, we wish to restrict to the event that the empirical distribution of $(\bh,\bnu)$ is close to the measure on $\R^2$ specified by the density function
	\[\f1{\sqrt{\bq}}
	\varphi\bigg(\f{h}{\sqrt{\bq}}\bigg)
	\varphi_{\xi_{\bq,z}}(\nu)
	\,.\]
This is formalized by conditions \eqref{d:profile.ii}~and~\eqref{d:profile.iii}.)
\end{dfn}

\noindent We will prove Theorem~\ref{t:second} for the restricted perceptron partition function
(with $M=\lfloor\Nall\aall\rfloor$)
	\beq\label{e:def.restricted.part.fn}
	\bZ_{\DELTA,t,\EPS}
	\equiv
	\bZ_{\DELTA,t,\EPS}
	(\bG_{M\times \Nall})
	\equiv \sum_{J\in\bbH_{M,\Nall,\DELTA,t,\EPS}}
	\mathbf{1}\Big\{
	\sat_{J,\DELTA,t,\EPS}
		\Big\}\,,\eeq
which is integer-valued with $0\le \bZ_{\DELTA,t,\EPS}(\bG_{M\times \Nall}) \le Z_{\kall}(\bG_{M\times \Nall})$. We shall always take $\EPS\downarrow0$ followed by $t\uparrow\infty$ followed by $\DELTA\downarrow0$ while keeping $\kall$ and $\aall$ fixed. For this reason we often suppress dependence on $\kall$ and $\aall$ in order to simplify the notation. We will also often abbreviate $\sat_J\equiv \sat_{J,\DELTA,t}$ and $\sat_{J,\EPS}\equiv \sat_{J,\DELTA,t,\EPS}$.

\subsection{Proof strategy} As above, fix $\kall$, $\aall$, $\sq\equiv q_{\star\kall}(\aall)$, $\spsi\equiv \psi_{\star\kall}(\aall)$. To compute the (conditional) second moment of \eqref{e:def.restricted.part.fn}, we take a second $K\in\bbH_\EPS$, so that $\bk\equiv K/N^{1/2}$ has an analogous decomposition as \eqref{e:restrict.j},
	\beq\label{e:restrict.k}
	\bk = \sqrt{q_{K,t}} \br^{(t)}
	+\sum_{s=1}^{t-1} \gamma_{K,s} \br^{(s)}
	+ \sqrt{1-q_K} \tbv\,.\eeq
Define the \textbf{overlap} between $J$ and $K$ as $\lm\equiv\lm_{J,K}\equiv(\bv,\tbv)$. Let
	\beq\label{e:c.lambda}
	\bw\equiv \f{\tbv-\lm\bv}{\sqrt{1-\lm^2}}\,,\quad
	c_\lm\equiv
	\f{1-\lm}{\sqrt{1-\lm^2}}
	=\f{\sqrt{1-\lm}}{\sqrt{1+\lm}}\,.\eeq
We will find below (in Proposition~\ref{p:cube.ld}) that there is a constant $\lmin \in (-1,0)$, explicitly defined by \eqref{e:lmin}, such that
	\[\min_{J,K\in\bbH_\EPS}
	\bigg\{
	\lm_{J,K}\bigg\}\ge \lmin-\ETA\]
where $\ETA$ is some error tending to zero in the manner of \eqref{e:ETA}. For any $\lmin-\ETA\le\lm\le1$ we let
	\begin{align*}
	\bbH(\lm)\equiv \bbH_{M,\Nall,\DELTA,t}(\lm)
	&\equiv \bigg\{
	(J,K) \in (\bbH_{M,\Nall,\DELTA,t})^2
	: \lm-\f1N\le\lm_{J,K} \le\lm+\f1N
	\bigg\}\,,\\
	\bbH_\EPS(\lm)
	\equiv \bbH_{M,\Nall,\DELTA,t,\EPS}(\lm)
	&\equiv \bigg\{
	(J,K) \in (\bbH_{M,\Nall,\DELTA,t,\EPS})^2
	: \lm-\f1N\le\lm_{J,K} \le\lm+\f1N
	\bigg\}\,,
	\end{align*}
so $\set{-1,+1}^{2N}\supseteq\bbH^2\supseteq\bbH(\lm)\supseteq \bbH_\EPS(\lm)$. In order to prove Theorem~\ref{t:second}, we first consider
the events $\sat_{J,\EPS}\equiv\sat_{J,\DELTA,t,\EPS}$ and $\sat_K\equiv\sat_{K,\DELTA,t}$
 for a fixed pair $(J,K)\in\bbH(\lm)$ 
(for this part of the calculation, the further restriction from $\bbH$ to $\bbH_\EPS$ is not needed). We condition on a background $\sigma$-field $\bgFilt$ as discussed in Remark~\ref{r:cond}, and hereafter suppress it from the notation, so that $\oP(\cdot)$ means $\oP(\cdot\,|\,\bgFilt)$. We will compute $\oP(\sat_{J,\EPS})$ and prove an upper bound on $\oP(\sat_K\giv\sat_{J,\EPS})$; this occupies Sections~\ref{s:tilt}~to~\ref{s:adm}. In Section~\ref{s:cube} we compute $\#\bbH_\EPS$ and prove an upper bound on $\#\bbH_\EPS(\lm)$ (at this point the restriction from $\bbH$ to $\bbH_\EPS$ becomes important). We shall see the formula \eqref{e:GG} arise from the limits
	\begin{align}\label{e:HH}
	\f{\log\#\bbH_\EPS}{N}
	\longrightarrow
	\HH_\star\equiv\HH_{\star\kall}(\aall)
	&\equiv
	-\psi(1-q)+
	\int\log(2\ch(\psi^{1/2}z))\varphi(z)\,dz
	\,\bigg|_{q=\sq,\psi=\spsi}\,,\\
	\label{e:PP}
	\f{\log \oP(\sat_{J,\EPS})}{N}
	\longrightarrow
	\PP_\star
	\equiv
	\PP_{\star\kall}(\aall)
	&\equiv 
	\f{\psi(1-q)}{2}
	+
	\alpha\int\log\bPhi\bigg(
	\f{{\kall}
	-q^{1/2}z}{\sqrt{1-q}}
	\bigg)\,\varphi(z)\,dz
	\,\bigg|_{q=\sq,\psi=\spsi}\,,
	\end{align}
where convergence holds in the limit $\Nall\uparrow\infty$ followed by $t\uparrow\infty$ followed by $\DELTA\downarrow0$, for any fixed $\EPS>0$. Combining gives the first moment estimate \eqref{e:first.mmt.gg}, since $\GG_\star=\HH_\star+\PP_\star$. We will see below that the (restricted) first moment \eqref{e:first.mmt.gg} is relatively straightforward, whereas the second moment bound \eqref{e:second.mmt.bound} requires significantly more involved calculations. We will introduce here the function $\SE_\star(\lm)$, and leave most of its interpretation and discussion to later sections. Let $P_{H,D}$ be the probability distribution on $\set{-1,+1}^2$ given by
	\beq\label{e:P.H.D.intro}
	P_{H,D}=\f14
	\begin{pmatrix}
	(1+m)^2+D&1-m^2-D\\
	1-m^2-D&(1-m)^2+D
	\end{pmatrix}
	\,\bigg|_{m=\tnh H}
	\,.\eeq
(We assume $D$ is such that $P_{H,D}$ is a nonnegative measure.) Let $\Gamma(H,D)$ be the (Shannon) entropy of $P_{H,D}$, so $0\le \Gamma(H,D) \le \log 4$.
For $H\in\R$ and $A\in(0,\infty)$ let
	\beq\label{e:D.H.intro}
	D_H(A)
	\equiv \f{(A^2-1)(1-m^2)^2}
		{ (
		\sqrt{A^2 + m^2- (Am)^2}
		+1)^2}
		\,\bigg|_{m=\tnh H}\,.\eeq
We integrate over the distribution of $H$ to define
	\beq\label{e:def.ell.A}
	\ell(A)
	\equiv \int \f{D_{\psi^{1/2} z}(A)}{1-q}
	\varphi(z)\,dz
	\,\bigg|_{q=\sq,\psi=\spsi}\,.\eeq
We show in Section~\ref{s:cube} that the function $\ell$ is strictly increasing on $A\in(0,\infty)$, sandwiched by boundary values
	\begin{align}\label{e:lmin}
	\ell(0)
	&=-\int \f{(1-|\tnh(\psi^{1/2}z)|)^2}{1-q}
	\,\varphi(z)\,dz
	\equiv \lmin \in (-1,0)\,,\\
	\ell(\infty)
	&= \int\f{1-\tnh(\psi^{1/2}z)^2}
		{1-q}\,\varphi(z)\,dz=1\,,
		\nonumber\end{align}
where the last equality is by \eqref{e:q.psi.recursion}.
See Figure~\ref{f:ell.of.tau} (where we have chosen what turns out to be a nice parametrization, $A=A(\tau)\equiv \exp(2\atnh(\tau))$).
The inverse $\lm\mapsto \ell^{-1}(\lm) \equiv A(\lm)$ is therefore well-defined for $\lmin\le\lm\le1$. Let
	\beq\label{e:HH.pair}
	\HH(\lm)
	\equiv
	- 2\HH_\star +
	\int \Gamma\Big(\psi^{1/2}z,
		D_{\psi^{1/2}z}(A(\lm))
		\Big)\varphi(z)\,dz
	\,\bigg|_{q=\sq,\psi=\spsi}\eeq
We will find in Section~\ref{s:cube} that
$\HH(0)=0$, $\HH(1)=-\HH_\star$, and
	\beq\label{e:pairs.ld.intro}
	\adjustlimits
	\limsup_{\DELTA\downarrow0}
	\limsup_{t\uparrow\infty}
	\limsup_{\EPS\downarrow0}
	\limsup_{\Nall\uparrow\infty} 
	\f{\log\#\bbH_\EPS(\lm)}
		{\Nall}
	\le 2\HH_\star+\HH(\lm)\,.\eeq
(see Proposition~\ref{p:cube.ld}).
See Figure~\ref{f:HH}. Next let
	\beq\label{e:def.xi}
	\xi_{q,z} \equiv \f{\kall
		-q^{1/2}z}{\sqrt{1-q}}\,,\eeq
and note that $\xi_{q,z}$ appears in the above definition \eqref{e:PP} of $\PP_\star$. Recalling the definition of $\varphi_\xi$ from \eqref{e:def.F}, let
	\beq\label{e:II}\II_s(\lm)
	\equiv\aall
	\int\int\log
	\bPhi\bigg(
	\f{\xi_{q,z}-\lm\nu}{\sqrt{1-\lm^2}}
	- \f{\cE(\xi_{q,z}) \cdot s}
		{\psi^{1/2}\sqrt{1-q}}
	\bigg)\,
	\varphi_{\xi_{q,z}}(\nu)\,d\nu
	\,\varphi(z)\,dz
	\,\bigg|_{q=\sq,\psi=\spsi}\,.\eeq
Abbreviate $\II(\lm)\equiv\II_0(\lm)$ and let
	\beq\label{e:PP.pair}
	\PP(\lm)
	\equiv-\PP_\star+
	\f{\psi(1-q)(1-\lm)}{2(1+\lm)}
	+\II(\lm)
	\,\bigg|_{q=\sq,\psi=\spsi}\,.\eeq
Then $\PP(0)=0$ and $\PP(1)=-\PP_\star$
(with precise estimates near $\lm\in\set{0,1}$ in Section~\ref{s:quant.bounds}); see Figure~\ref{f:PP}. Let
	\beq\label{e:BB}
	\BB(\lm,s)
	\equiv \f{s^2}{2}
	-\psi^{1/2}\sqrt{1-q} s
	\bigg( \f{1-\lm}{1+\lm}\bigg)^{1/2}
	+ \II_s(\lm) - \II(\lm)
	\,\bigg|_{q=\sq,\psi=\spsi}\,.\eeq
and define $\AAA(\lm) \equiv \inf_s\BB(\lm,s)$;
see Figure~\ref{f:AA}. We remark that $\AAA(\lm) \le \BB(\lm,0)=0$, and
	\[
	\f{\pd\BB(0,s)}{\pd s}\bigg|_{s=0}
	=-\psi^{1/2}\sqrt{1-q}
	+\aall \int
	\f{\cE(\xi_{q,z})^2}{\psi^{1/2}\sqrt{1-q}}
	\,\varphi(z)\,dz
	\,\bigg|_{q=\sq,\psi=\spsi}
	=0\,,
	\]
where the last equality is by \eqref{e:q.psi.recursion}. We will see in Lemma~\ref{l:adm.mgf} below that $\BB(\lm,s)$ arises as a limit of cumulant-generating functions, and is therefore convex in $s$. It follows that $s=0$ is a minimizer of $\BB(0,s)$, and so $\AAA(0)=0$. Since $\AAA\le0$ everywhere it follows that $\AAA'(0)=0$ and $\AAA''(0)\le0$. We will find in Sections~\ref{s:tilt}~through~\ref{s:adm} that
	\beq\label{e:PP.AA.ld.intro}
	\adjustlimits
	\limsup_{\DELTA\downarrow0}
	\limsup_{t\uparrow\infty}
	\limsup_{\EPS\downarrow0}
	\limsup_{\Nall\uparrow\infty} 
	\f{\log\oP(\sat_{J,\EPS},\sat_K)}{\Nall}
	\le 2\PP_\star +\PP(\lm) + \AAA(\lm)\,.\eeq
Let $\SE(\lm)\equiv \HH(\lm)+\PP(\lm)+\AAA(\lm)$;
see Figure~\ref{f:HPA}. Although we have suppressed it from the above notation, we now recall that the functions $\SE(\lm)$, $\HH(\lm)$, $\PP(\lm)$, $\AAA(\lm)$ all depend on the parameters
$\kall$ and $\aall$. We make the dependence explicit by writing $\SE(\lm)\equiv\SE_{\kall,\aall}(\lm)$ and so on. Then, from \eqref{e:pairs.ld.intro} and \eqref{e:PP.AA.ld.intro} it is easy to derive that
	\beq\label{e:informal.statement.SE}
	\limsup_{\DELTA\downarrow0}
	\limsup_{t\uparrow\infty}
	\limsup_{\EPS\downarrow0}
	\limsup_{\Nall\uparrow\infty} 
	\f{\log \E(\bZ_{\DELTA,t,\EPS}
		(\bG_{\lfloor N\alpha
		\rfloor \times N})^2)}{N_{\textup{all}}}
	\le 2 \GG_{\star\kall}(\aall)
	+\SE_{\kall,\aall}(\lm)\eeq
with $\GG_\star=\HH_\star+\PP_\star$ as above. 
(The detailed derivation of \eqref{e:informal.statement.SE} is given in the proof of Theorem~\ref{t:second}, in Section~\ref{s:conclusion}.)
Condition~\ref{c:G} refers to the function
	\begin{align}\nonumber
	\SE_\star(\lm)
	\equiv \SE_{0,\asat}(\lm)
	&=\bigg\{ -2\HH_\star
	+\int \Gamma\Big(\psi^{1/2}z,
		D_{\psi^{1/2}z}(A(\lm))
		\Big)\varphi(z)\,dz
	-\PP_\star+\f{\psi(1-q)(1-\lm)}{2(1+\lm)}\\
	&\qquad+\inf_s\bigg\{
	\f{s^2}{2}
	-\psi^{1/2}\sqrt{1-q} s
	\bigg( \f{1-\lm}{1+\lm}\bigg)^{1/2}
	+ \II_s(\lm)
	\bigg\}
	\bigg\}
	\,\bigg|_{\substack{
	\kall=0,
	\aall=\asat,\\
	q=\sq(\asat),\psi=\spsi(\asat)}}
	\label{e:GG.condition}
	\end{align}
for $\asat$ as given by Proposition~\ref{p:gardner}. It is not difficult to see from \eqref{e:first.mmt.gg} and \eqref{e:informal.statement.SE} that Condition~\ref{c:G} is certainly a necessary condition for the result of Theorem~\ref{t:second}.

\begin{figure}[!ht]\centering
\begin{subfigure}[b]{0.48\textwidth}\centering
\includegraphics[height=1.8in]{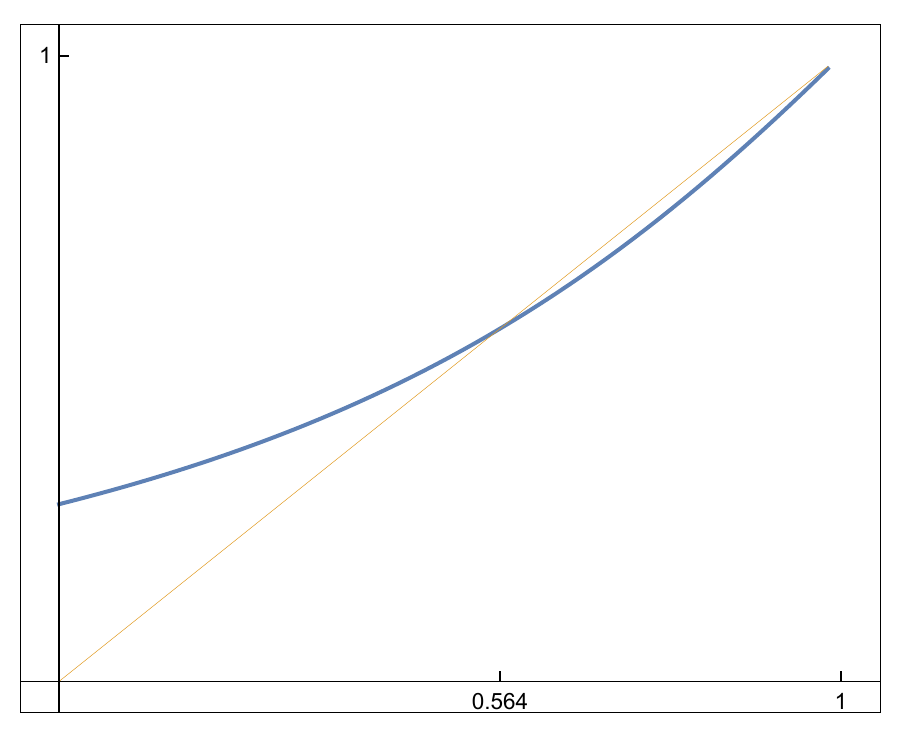}
\caption{For $\alpha\doteq0.833$ the function $q\mapsto P(R(q,\alpha))$ of \eqref{e:q.psi.recursion}, with fixed point $\sq(\alpha)\doteq 0.564$. The figure also indicates that condition~\eqref{e:at} (the Almeida--Thouless condition) is satisfied at $\alpha\doteq0.833$, since the slope of the recursion at $\sq$ is strictly less than one.}
\label{f:qrecursion}
\end{subfigure}\quad
\begin{subfigure}[b]{0.48\textwidth}\centering
\includegraphics[height=1.8in]{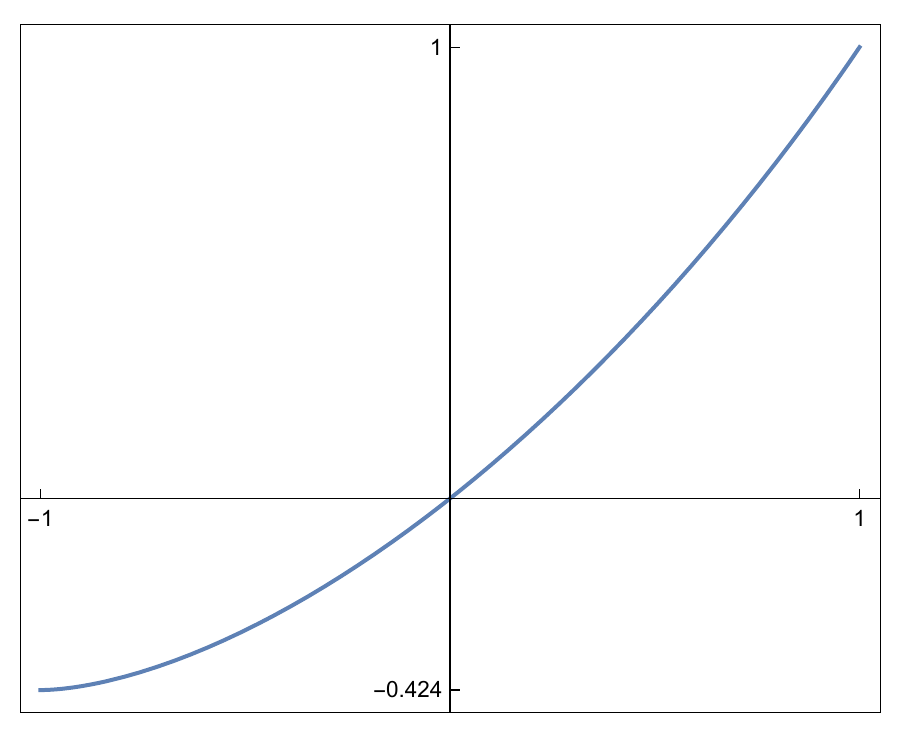}
\caption{For $\alpha\doteq0.833$ the map $\tau \mapsto \ell(A(\tau))$ for $A(\tau)\equiv \exp(2\atnh(\tau))$ and $\ell$ as defined by \eqref{e:def.ell.A}. The map is strictly increasing, sandwiched between boundary values
$\lmin\doteq -0.424$ (whose exact formula is given by \eqref{e:lmin}) and $\ell(\infty)=1$.}
\label{f:ell.of.tau}
\end{subfigure}
\end{figure}

\begin{figure}[!ht]\centering
\begin{subfigure}[b]{0.48\textwidth}\centering
\includegraphics[width=\textwidth]{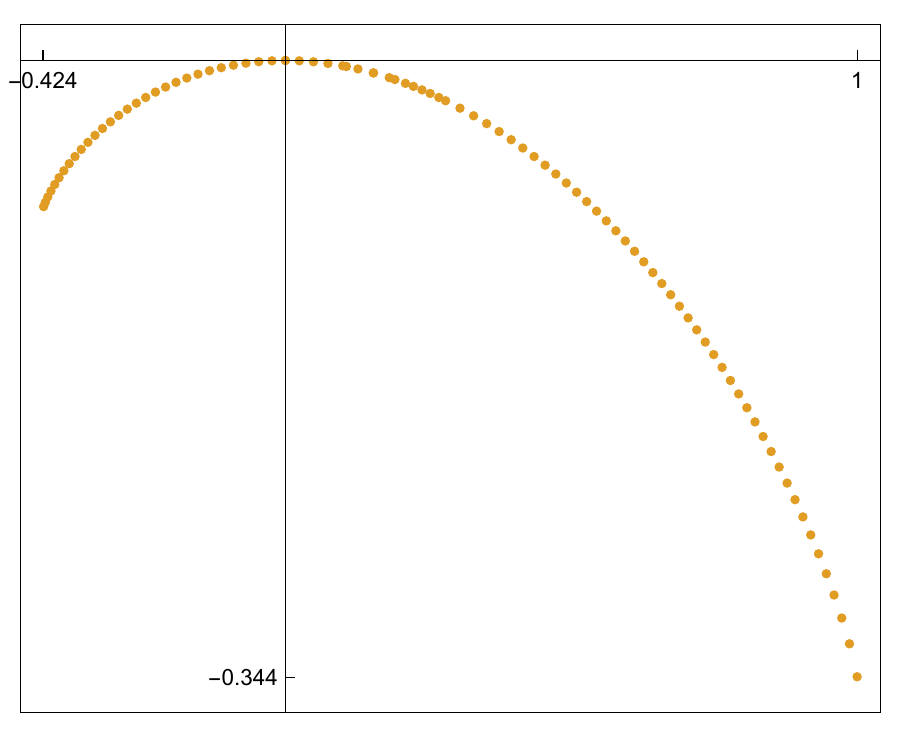}
\caption{$\HH(\lm)$ as in \eqref{e:HH.pair} and \eqref{e:pairs.ld}.\\
$\HH(0)$ and $\HH(1)=-\HH_\star \doteq -0.344$.}
\label{f:HH}
\end{subfigure}\quad
\begin{subfigure}[b]{0.48\textwidth}\centering
\includegraphics[width=\textwidth]{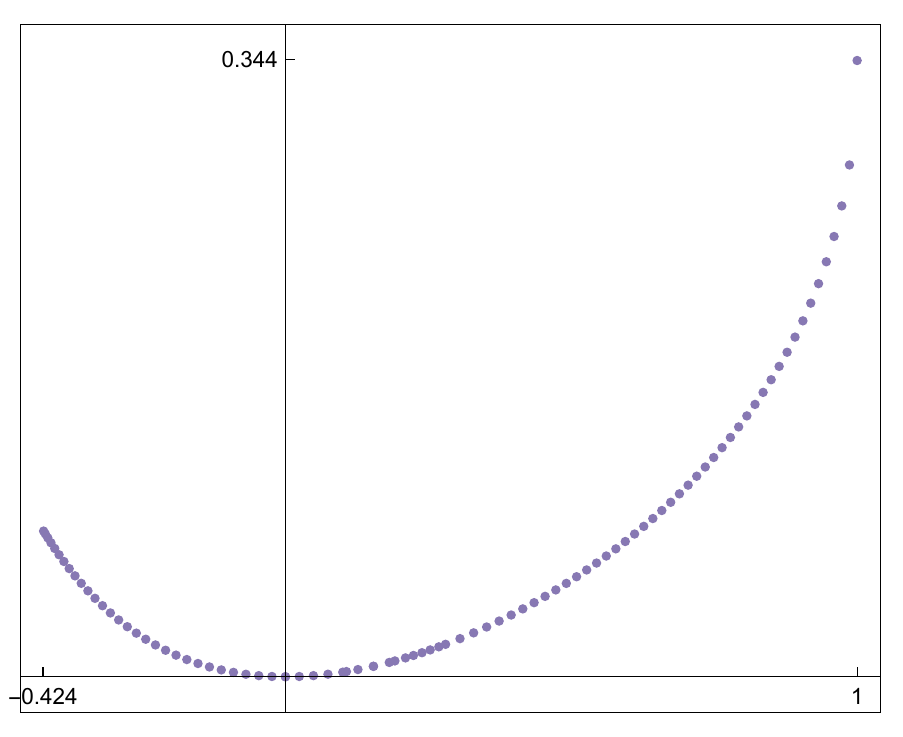}
\caption{$\PP(\lm)$ as in \eqref{e:PP.pair}
and \eqref{e:conv.to.PP.pair}.\\
$\PP(0)=0$ and $\PP(1)=-\PP_\star\doteq 0.344$}
\label{f:PP}
\end{subfigure}
\bigskip\bigskip

\begin{subfigure}[b]{0.48\textwidth}\centering
\includegraphics[width=\textwidth]{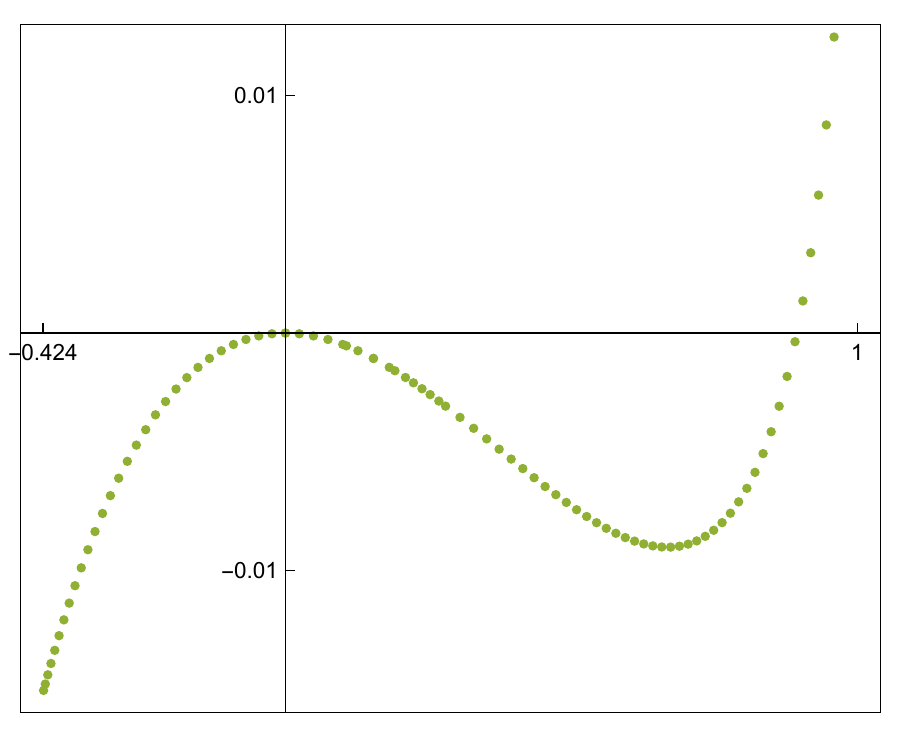}
\caption{$\BB(\lm,s)$ from \eqref{e:BB} with $s=-0.3\lm$\\
(an upper bound $\AAA(\lm)$ as appears in \eqref{e:adm.givensr.ld}).}
\label{f:AA}
\end{subfigure}
\quad
\begin{subfigure}[b]{0.48\textwidth}\centering
\includegraphics[width=\textwidth]{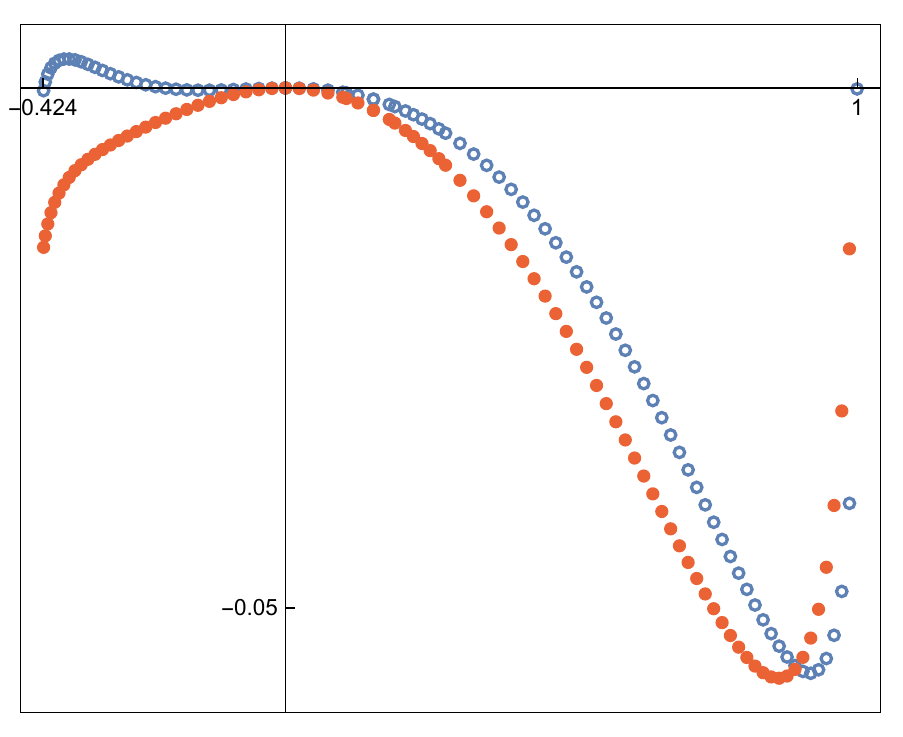}
\caption{
$\HH(\lm)+\PP(\lm)$ marked by blue open circles;\\
$\HH(\lm)+\PP(\lm)+\BB(\lm,-0.3\lm)$ marked by red filled points.}
\label{f:HPA}
\end{subfigure}
\caption{For $\alpha\doteq0.833$ the functions that arise in the second moment calculation.}
\end{figure}

\begin{rmk}\label{r:cond}
We now make some comments on our conditioning scheme. We decompose $\bG_{M\times \Nall}$ as in \eqref{e:g.kr}, and run the TAP equations
\eqref{e:tap.n}~and~\eqref{e:tap.m}
 on $\bE$ with deterministic initial starting vectors $\bn^{(0)}\in\R^M$ and $\bmag^{(1)}\in\R^N$. For $s\ge1$ we obtain $\bn^{(s)}$ as a function of $\bE\bmag^{(s)}$ and $\bn^{(s-1)}$; then $\bmag^{(s+1)}$ as a function of $\bE^\st\bn^{(s)}$ and $\bmag^{(s)})$; and so on, up to $\bmag^{(t)}\equiv\bmag$ and $\bn^{(t)}$. We define $\bhJ$ in \eqref{e:def.bhJ} as a measurable function of $\bE\bmag,\hat{\bE}$. We then sample an independent vector $\dkappa$, and obtain $\bn$ in \eqref{e:fixed.pt.n} as a function of $\bE\bmag$, $\hat{\bE}\bhJ$, and $\dkappa$. Lastly we obtain $\bmag^{(t+1)}$ in \eqref{eq-def-H-circ} as a function of $\bE^\st\bn$ and $\bmag$. This defines for us $\Filt=\sigma(\textsf{DATA})$ with $\textsf{DATA}$ as in \eqref{e:filt.t}. Conditional on $\Filt$, the matrix $\bE$ has the law of a standard gaussian in $\R^{M\times N}$ subject to the equations
	\beq\label{e:tap.as.lin.constraints.row}
	\f{\bE\bmag^{(s)}}{N^{1/2}}
	= \bh^{(s)}+b^{(s)}\bn^{(s-1)}\,,\quad
	\f{\bE^\st\bn^{(s)}}{N^{1/2}}
	= \bH^{(s+1)}+d^{(s)}\bmag^{(s)}\eeq
for $1\le s\le t-1$, together with
(see \eqref{e:fixed.pt.n} and \eqref{eq-def-H-circ})
	\beq\label{e:tap.as.lin.constraints.col}
	\f{\bE\bmag+\hat{\bE}\bhJ}{N^{1/2}}
		-\dkappa
	= \bh+(1-\bq)\bn\,,\quad
	\f{\bE^\st\bn}{N^{1/2}}
	= \bH^{(t+1)} + d^{(t)}\bmag\,.\eeq
Let $\mathbb{Q}$ denote the marginal law of the sequence $\textsf{DATA}$. For the remainder of this paper, we \emph{first} sample $\textsf{DATA}\sim\mathbb{Q}$, and make all calculations conditional on a background $\sigma$-field $\bgFilt\equiv\sigma(\textsf{DATA})$. We then let $\bE$ be a standard gaussian in $\R^{M\times N}$, \emph{independent of} $\bgFilt$ until further notice. We shall subsequently reintroduce constraints on $\bE$ in a way that is equivalent to conditioning on $\Filt$.
\end{rmk}

\subsection{Fixed point existence}\label{ss:kr} We
conclude this section with the proof of Proposition~\ref{p:kr}.

\begin{lem}\label{l:kr.base.case}
For any $\DELTA>0$ there exists $t = t(\DELTA)$ such that $\bmag\equiv\bmag^{(t)}$ satisfies
	\beq\label{e:kr.base.case}
	\bigg\|\bigg(
	\DELTA^{1/2}+\kappa
	- \f{\bE\bmag}{N^{1/2}}
	\bigg)_+
	\bigg\|_2
	\le \f{M^{1/2}}{\exp(c/\DELTA)}\eeq
with high probability for large enough $N$, where $c$ is a positive constant depending only on $\kappa$ for $\alpha\in(\albd,\aubd)$.

\begin{proof} From the TAP equations \eqref{e:tap.as.lin.constraints.row} we have
	\beq\label{e:tap.at.t.rewrite}
	\f{\bE\bmag}{N^{1/2}}
	=\bh^{(t)} + b^{(t)}\bn^{(t-1)}
	=L_{\bq}(\bh^{(t)})
	+(1-\bq) \Big(\bn^{(t-1)}-\bn^{(t)}\Big)\eeq
for $L_{\bq}$ as in the proof of Lemma~\ref{l:kr}. It follows from the results of \cite{bayati2011dynamics,bolthausencmp} that $\bn^{(t)}$ converges in $\ell^2$ to a fixed point (Remark~\ref{r:tap}\ref{r:tap.a}), in the sense that
	\[\adjustlimits\lim_{t\uparrow\infty}
	\limsup_{N\uparrow\infty}
	\f{\ltwo{\bn^{(t)}-\bn^{(t-1)}}}{N^{1/2}}=0\]
in probability. Consequently, for any fixed $\DELTA$, we can choose $t=t(\DELTA)$ large enough such that
	\beq\label{eq-difference-two-Fs}
	\f{\ltwo{\bn^{(t)}-\bn^{(t-1)}}}{N^{1/2}}
	\le \f1{\exp(1/\DELTA^2)}\eeq
for all $t\ge t(\DELTA)$. As explained in the proof of Lemma~\ref{l:kr}, the function $L_{\bq}$ defines a strictly increasing map from $\R$ onto $(\kappa,\infty)$, with $L_{\bq}(h)-\kappa \asymp 1/|h|$ in the limit $h\to-\infty$. This implies that $L_{\bq}(h) \le \kappa +\DELTA^{1/2}$ only if $h \le -c'/\DELTA^{1/2}$ for some positive constant $c'$ (depending only on $\kappa$). If $Z$ is a standard gaussian, then the chance of $\sqrt{\TAPq} Z \le -c'/\DELTA^{1/2}$ is upper bounded by $\exp(-\Omega(1/\DELTA))$. Now recall that the empirical profile of $\bh^{(t)}$ is close to the distribution of $\sqrt{\TAPq} Z$: this was informally stated by Remark~\ref{r:tap}\ref{r:tap.b}, and a formal statement is given by \cite[Lemma 1(c)]{bayati2011dynamics}. Their result yields (possibly enlarging $t(\DELTA)$ as needed) that
	\beq\label{e:ht.far.tail}
	\f1M \sum_{\mu=1}^M
	\mathbf{1}\Big\{L_{\bq}(h^{(t)}_\mu)
		\le \kappa + \DELTA^{1/2}\Big\}
	\le
	\f1M \sum_{\mu=1}^M (h^{(t)}_\mu)^2 
	\mathbf{1}\bigg\{
	h^{(t)}_\mu
	\le -\f{c'}{\DELTA^{1/2}}
	\bigg\}
	\le \f1{\exp(2c/\DELTA)}\eeq
with high probability, for a positive constant $c$ depending only on $\kappa$ 
for $\alpha\in(\albd,\aubd)$. Consequently
	\[\bigg\|\Big(
	\DELTA^{1/2} + \kappa
	- L_{\bq}(\bh^{(t)})
	\Big)_+
	\bigg\|_2
	\le \f{M^{1/2} \DELTA^{1/2}}
	{\exp(c/\DELTA)}\,,\]
and combining with \eqref{eq-difference-two-Fs} yields the claimed bound.
\end{proof}
\end{lem}

\noindent We next prove Proposition~\ref{p:kr} by adapting an argument of \cite{kim1998covering}.

\newcommand{\bdef}{\MVEC{f}}

\begin{proof}[Proof of Proposition~\ref{p:kr}]
Recall \eqref{e:g.kr} that we decompose $\bG\equiv\bG_{M\times \Nall}$ into submatrices $\bE$ and $\hat{\bE}$. Denote $\bE_{\mu,i}\equiv g_{\mu,i}$ and $\hat{\bE}_{\mu,i}\equiv \hat{g}_{\mu,i}$, and recall these are i.i.d.\ standard gaussian random variables. We construct $\hat{J}$ in stages $1\le s\le s_\textup{end}$ where $s_\textup{end} =O(\log N)$. We decompose the coordinates $1\le i\le \hat{N}$ into consecutive blocks of cardinality
	\beq\label{e:kr.Ns}
	N_s\equiv
	\f{M\DELTA/4^s}
	{4s+2\,\exp(1/\DELTA^{3/4}) + 16}\,.\eeq
At stage $s$, we choose each of the next $N_s$ spins of $\hat{J}$ according to a \emph{weighted} majority of the corresponding column in $\hat{\bE}$, where the weights are determined by what was revealed in the previous stages. More precisely, the coordinates involved in stages $r$ through $s$ (inclusive) are denoted
	\[r:s \equiv \bigg\{
	\sum_{u=1}^{r-1}N_u\,,
		\ldots,
	 \sum_{u=1}^s N_u\bigg\}\,.\]
We write $\hat{\bE}(r:s)$ for the submatrix of $\hat{\bE}$ with column indices in $r:s$, and $\hat{J}(r:s) \equiv (\hat{J}_i)_{i\in r:s}$. With this notation, the portion of $\bE\bmag^{(t)} + \hat{\bE}\hat{J} $ that is fixed upon completion of stage $s$ is
	\[\bz^{(s)}
	\equiv \bE\bmag^{(t)}
	+ \hat{\bE}(1:s) \hat{J}(1:s)
	\equiv \bE\bmag^{(t)}
		+\bigg( \sum_{i\in 1:s}
			\hat{g}_{\mu,i}
			\hat{J}_i\bigg)_{\mu\le M}
	\in \R^M\,.\]
For $s\ge0$ let us define
	\[T_s\equiv N^{1/2} \kappa 
	+ (M\DELTA)^{1/2}
	\bigg(1 + \f1{2^s}\bigg)\,,\quad
	\bdef^{(s)}
	\equiv \Big(T_s\oneM-\bz^{(s)}\Big)_+
	\in\R^M\,.\]
We use $\bdef^{(s-1)}$ to measure the ``deficit'' after stage $s-1$, and determine the weights in stage $s$:
	\[\hat{J}(s:s)
	\equiv \sgn\bigg\{
	\Big(\hat{\bE}(s:s) \Big)^\st \bdef^{(s-1)}\bigg\}
	=\bigg(
	\sgn\bigg\{
	\sum_{\mu=1}^M 
	\hat{g}_{\mu,i}(\xi^{(s-1)})_\mu
	\bigg\}
	\bigg)_{i\in s:s}\,.\]
Abbreviate $\hat{\Filt}_s=\sigma(\bE\bmag^{(t)},\hat{\bE}(1:s))$. We will prove inductively that
	\beq\label{e:kr.induction}
	\oP\bigg(\ltwo{\bdef^{(s)}} \le
		\f{N_{s+1}}{5}
	\,\bigg|\, \hat{\Filt}_{s-1} \bigg)
	\ge 1- \f{(2/3)^s}{2}\eeq
for $N_{s+1}$ as defined by \eqref{e:kr.Ns}. Note that the base case
	\[\ltwo{\bdef^{(0)}}
	\le\f{N_1}{5}
	=\f{M\DELTA/40}
		{10 + \exp(1/\DELTA^{3/4})}\]
follows from Lemma~\ref{l:kr.base.case}. Now take $s\ge1$, abbreviate $\bdef\equiv\bdef^{(s-1)}$, and suppose inductively that $\ltwo{\bdef}\le N_s/5$ holds. All subsequent calculations are conditional on $\hat{\Filt}_{s-1}$, which we suppress from the notation. To bound the output of stage $s$, let us first fix any 
 $\mu\le M$ and $i\in s:s$. Write $\mathbf{P} \equiv \oP(\cdot\,|\,\hat{\Filt}_{s-1},\hat{g}_{\mu,i})$ and write $\mathbf{E}$ for expectation under $\mathbf{P}$. Then, writing $Z$ for an independent standard gaussian random variable, we have
	\begin{align*}
	m_{\mu,i}
	&\equiv \mathbf{E}
		\bigg(\hat{J}_i
		\cdot \sgn \hat{g}_{\mu,i}
		\bigg)
	=2\mathbf{P}_{\mu,i}\bigg(
	\sgn\bigg\{\sum_{\mu'\le M}
		\xi_{\mu'} \hat{g}_{\mu',i}
	\bigg\}
	= \sgn \hat{g}_{\mu,i}
	\bigg)-1\\
	&= 2\mathbf{P}_{\mu,i}\bigg(0\le Z
	\le \f{ |\hat{g}_{\mu,i}|
		(\xi_\mu / \ltwo{\bdef})}
		{	(1
		- (\xi_\mu/\ltwo{\bdef})^2)^{1/2} }
	\bigg)\end{align*}
for all $i\in s:s$. In particular,
$m_{\mu,i}$ always lies between zero and one, and can be lower bounded by
	\[m_{\mu,i} 
	\ge 2\,
	\mathbf{P}
	\bigg( 0 \le Z \le
		\f{|\hat{g}_{\mu,i}|\xi_\mu}
			{\ltwo{\bdef}}
		\bigg)
	\ge\f{2\,\xi_\mu |\hat{g}_{\mu,i}|}{\ltwo{\bdef}}
		\varphi(\hat{g}_{\mu,i})\,.\]
Averaging this bound over the law of $\hat{g}_{\mu,i}$ gives
	\beq\label{e:kr.bias.lbd}
	\E\Big( \hat{g}_{\mu,i} \hat{J}_i \Big)
	=\E\Big( |\hat{g}_{\mu,i}| m_{\mu,i}
	 \Big)
	\ge \f{2\xi_\mu }
			{\ltwo{\bdef}}
	\E\bigg(
	|\hat{g}_{\mu,i}|^2
	\varphi(\hat{g}_{\mu,i})
	\bigg)
	\ge \f{\xi_\mu}{5\ltwo{\bdef}}\eeq
for all $i\in s:s$. In addition,
with $Z$ denoting an independent standard gaussian as above, we have
	\beq\label{e:kr.exp.bound}
	\f{\E \exp(\gamma \hat{g}_{\mu,i} \hat{J}_i)}
	{\exp \{\gamma \E ( \hat{g}_{\mu,i} \hat{J}_i)\}}
	=\f1{\exp(\gamma m_{\mu,i})}
	\bigg(
	\E \exp(\gamma Z) + m_{\mu,i}
	\f{\E \exp(\gamma|Z|)-\E \exp(-\gamma|Z|)
	}{2}\bigg)
	\le \exp(\gamma^2/2)\,,\eeq
where the last bound holds for all $0\le \gamma \le \gamma_{\max}$ with $\gamma_{\max}$ a small absolute constant --- this follows by noting that for 
sufficiently small $\gamma$ we have
	\begin{align*}
	&1+ m_{\mu,i}
	\f{\E \exp(\gamma|g|)-\E \exp(-\gamma|g|)}{2}
	=1+m_{\mu,i}\bigg(\gamma \E(|g|) + O(\gamma^2)
		\bigg)\\
	&\qquad=1+m_{\mu,i}
	\bigg(\f{2^{1/2}\gamma}{\pi^{1/2}} + O(\gamma^2)\bigg)
	\le \exp(\gamma m_{\mu,i})\,.\end{align*}
Denote $\tau_s\equiv T_{s-1}-T_s= (M\DELTA)^{1/2}/2^s$, and define the error vector
	\beq\label{e:kr.err}
	\err \equiv \err^{(s)}
	\equiv \bz^{(s-1)}
	-\bz^{(s)}
	+ \f{N_s \bdef}{5\ltwo{\bdef}}
	- \tau_s\oneM
	\in\R^M\,.\eeq
Each individual entry $\err_\mu$ is the sum of $N_s$ i.i.d.\ random variables (although there is some dependency among the different entries of $\err$). It follows from \eqref{e:kr.bias.lbd} that
	\[-\E\err 
	= \bigg(\sum_{i\in s:s}\E \Big(
		\hat{g}_{\mu,i} \hat{J}_i\Big)
		\bigg)_{\mu\le M}
		-\f{N_s \bdef}{5\ltwo{\bdef}}
		+\tau_s\oneM
		\ge\tau_s\oneM\,.\]
It follows from \eqref{e:kr.exp.bound} that
for all $0\le \gamma\le\gamma_{\max}$, 
	\beq\label{e:kr.gaussian.mgf}
	\E \exp\bigg\{
	\gamma \Big( \err_\mu-\E \err_\mu\Big)
	\bigg\} 
	= \prod_{i\in s:s}
	\E \exp \bigg\{\gamma \Big(
		\hat{g}_{\mu,i} \hat{J}
		-\E(\hat{g}_{\mu,i} \hat{J})
		\Big)\bigg\}
	\le \exp\bigg\{ \f{N_s\gamma^2}{2}
	\bigg\}\,.\eeq
We then use the definition \eqref{e:kr.err} of $\err$ to rewrite
	\[T_s\oneM-\bz^{(s)}
	=\bigg(T_{s-1}-\bz^{(s-1)}
		-\f{N_s\bdef}{5\ltwo{\bdef}}
		\bigg)
		+\err
	\le \bdef\bigg( 1 - \f{N_s}{5\ltwo{\bdef}} \bigg)
		+\err
	\le \err\,,\]
where the last bound holds since we have 
$N_s \ge 5\ltwo{\bdef}$ by the inductive hypothesis. Taking the positive part of each entry gives $\bdef^{(s)}\le (\err)_+$, therefore
	\[\E\Big(
	[(\xi^{(s)})_\mu]^2\Big)
	= 
		\int_0^\infty \oP(\err_\mu \ge r^{1/2})\,dr
	\le \inf_{\gamma\ge0}
	\int_0^\infty \f{\E
		\exp( \gamma 
		(\err_\mu- \E\err_\mu)
		)}
	{ \exp(\gamma
	( r^{1/2} - \E\err_\mu))}
	\,dr\]
By \eqref{e:kr.gaussian.mgf} together with the earlier observation $-\E\err_\mu \ge\tau_s$, the above is
	\[\le
	\inf_{0\le \gamma\le\gamma_{\max}}
	\f{\exp(N_s\gamma^2/2)}
		{ \exp(-\gamma\E\err_\mu) }
	\int_0^\infty \f{dr}{\exp(\gamma r^{1/2})}
	= \inf_{0\le \gamma\le\gamma_{\max}}
	\f{2 \exp(N_s\gamma^2/2 )
	}{\gamma^2 \exp(\gamma \tau_s)}\,.\]
Setting $\gamma\equiv\tau_s/N_s$ (we shall check below that $\gamma\le\gamma_{\max}$ for $s\le s_\textup{end}$) and
 summing over $\mu$ gives
	\[\oP\bigg(
	\ltwo{\bdef^{(s)}} \ge \f{N_{s+1}}{5} \bigg)
	\le
	\f{\E( \ltwo{\bdef^{(s)}}^2)}
		{(N_{s+1}/5)^2 }
	\le \f{50 M (N_s/N_{s+1})^2 / 
		(\tau_s)^2}{
		\exp( (\tau_s)^2 / (2 N_s))}\,.\]
Recall from \eqref{e:kr.Ns} that 
$N_s = (M\DELTA/4^s)/(4s+2\,\exp(1/\DELTA^{3/4})+16)$,
while $\tau_s=(M\DELTA)^{1/2}/2^s$.
Thus
	\begin{align*}
	\oP\bigg(
	\ltwo{\bdef^{(s)}} \ge \f{N_{s+1}}{5} \bigg)
	&\le \bigg(\f{6(s+1)
	+2\,\exp(1/\DELTA^{3/4})+16}
		{4s
		+2\,\exp(1/\DELTA^{3/4})
		+16}\bigg)^2
	\f{50 \cdot 16 \cdot 4^s / \DELTA}
		{\exp(2s+
		\exp(1/\DELTA^{3/4})
		+8 )\}}\\
	&= \bigg(1 + \f{O(1)}{\exp(1/\DELTA^{3/4})}\bigg)
	\f{50 \cdot 16}{\exp(8)}
	\f{4^s}{\exp(2s)}
	\le \f{(2/3)^s}2\,.\end{align*}
To conclude, we note that the required condition
	\[\gamma
	= \f{\tau_s}{N_s}
	= \f{4s + 2\,\exp(1/\DELTA^{3/4})
	 + 16}{\tau_s}
	\le \gamma_{\max}\]
holds as long as $s$ is small relative to $\tau_s = (M\DELTA)^{1/2}/2^s$. For instance, we can certainly continue the induction up to the largest integer $s_\textup{end}$ such that 
$2^{s_\textup{end}} \le (M\DELTA)^{1/4}$, which gives $s_\textup{end}\asymp \log M\ll (M\DELTA)^{1/4} \le \tau_{s_\textup{end}}$ as desired. The induction will end (with positive probability) at
	\[\ltwo{\bdef^{(s_\textup{end})}}^2
	\le \f{N_{s_\textup{end}+1}}{5}
	\le \f{M\DELTA/4^{s_\textup{end}+1}}
		{4s_\textup{end} + 
		2\,\exp(1/\DELTA^{3/4}) + 16}
	\le \f{(M\DELTA)^{1/4}}{4}\,.\]
On the other hand, any entry of $\bz^{(s_\textup{end})}$ which is smaller than $(M\DELTA)^{1/2}/2$ must contribute at least $M\DELTA/4$ to $\ltwo{\bdef^{(s_\textup{end})}}^2$. Since $M\DELTA$ is large, we obtain a contradiction unless 
$\bz^{(s_\textup{end})}$ has no such entry, that is to say,
$\bz^{(s_\textup{end})} \ge (M\DELTA)^{1/4}/4$. The probability of this event is lower bounded by
	\[\prod_{j=0}^\infty
	\bigg(1 - \f{(2/3)^s}{2}\bigg)
	\ge \f1{10}\,.\]
The result follows by taking $1:s_\textup{end} = \set{1,\ldots,\hat{N}}$.
\end{proof}

\begin{cor}\label{c:kr}
In the setting of Proposition~\ref{p:kr},
it holds for all $t\ge t(\DELTA)$ that
	\[\f{\ltwo{\bn-\bn^{(t)}}
	+\ltwo{\bh-\bh^{(t)}}}{N^{1/2}}
	\le \f{1}{\exp(1/\DELTA^{1/3})}\,.\]

\begin{proof}
For any fixed $\hat{J}\in\set{-1,+1}^{\hat{N}}$, we have
	\[\ltwo{\hat{\bE}\hat{J}}^2
	= \sum_{\mu=1}^M\bigg(
	\sum_{i\le\hat{N}}
	\hat{g}_{\mu,i} \hat{J}_i\bigg)^2
	\stackrel{\textup{d}}{=}
	\hat{N}\sum_{\mu=1}^M
	(g_\mu)^2\]
where the $g_\mu$ are i.i.d.\ standard gaussians, and $\stackrel{\textup{d}}{=}$ indicates equality in law. Therefore
	\[ 
	\oP\bigg( \f{\ltwo{\hat{\bE}\hat{J}}}{N}
	\ge \f{\DELTA}{\exp(1/\DELTA^{1/2})}\bigg)
	\le \exp\bigg\{
		-N\exp\bigg( \f{\Omega(1)}{\DELTA^{2/3}}
		\bigg)
	\bigg\}\,,\]
using the bound on $\hat{N}/N$ from Proposition~\ref{p:kr}. It follows by a union bound that with high probability we have
	\beq\label{e:union.bound.hat.J}
	\f{\ltwo{\hat{\bE}\hat{J}}}{N}
	\le \f{\DELTA}{\exp(1/\DELTA^{1/2})}\eeq
simultaneously for \emph{all} $\hat{J}\in\set{-1,+1}^{\hat{N}}$, in particular for $\hat{J}=\bhJ$. Recall from~\eqref{e:fixed.pt}~and~\eqref{e:tap.at.t.rewrite} that	
	\[ L_{\bq}(\bh^{(t)})
		+(1-\bq)\Big(
		\bn^{(t-1)}-\bn^{(t)}
		\Big)
	=\f{\bE\bmag}{N^{1/2}}
	=L_{\bq}(\bh)
	+\dkappa - \f{\hat{\bE}\bhJ}{N^{1/2}}\,.\]
Rearranging the above gives
	\beq\label{e:diff.tildeF.error.bound}
	\f{\ltwo{L_{\bq}(\bh)-L_{\bq}(\bh^{(t)})}}
		{N^{1/2}}
	=\f1{N^{1/2}}
	\bigg\| (1-\bq)\Big(
		\bn^{(t-1)}-\bn^{(t)}
		\Big)
		+\f{\hat{\bE}\bhJ}{N^{1/2}}
		-\dkappa
		\bigg\|_2
	\le \f{1}{\exp(1/\DELTA^{1/2})}
	\,,\eeq
where the last bound follows by combining \eqref{eq-difference-two-Fs}, \eqref{e:union.bound.hat.J}, and the definition of $\dkappa$ (as given immediately prior to~\eqref{e:fixed.pt}). To deduce a bound on $\ltwo{\bh-\bh^{(t)}}$, we will use a well-known expansion (see \cite{MR2390935} and refs.\ therein)
	\[\cE(x) 
	= x+\f1{x}-\f{2}{x^3}+\f{O(1)}{x^5}\]
which holds in the limit $x\to+\infty$.
It follows that $\cE'(x) = 1 - 1/x^2 + O(1)/x^4$, so
as $h\to-\infty$ we have
	\beq\label{e:tildeF.deriv.lbd}
	\f{dL_{\bq}}{dh}
	= 1 -\cE\bigg(
		\f{\kappa-h}{\sqrt{1-\bq}}
		\bigg)
	\asymp \f1{h^2}
	\,.\eeq
The definition \eqref{e:fixed.pt} implies
$L_{\bq}(\bh)\ge(\kappa+\DELTA^{1/2})\oneM$ coordinatewise, therefore $\bh \ge -(c/\DELTA^{1/2})\oneM$ coordinatewise for some constant $c$. Let $\bh' \equiv \max\set{\bh^{(t)}, -c/\DELTA^{1/2}}$. Then
	\[\f{\ltwo{\bh-\bh'}}{N^{1/2}}
	\le \f{\ltwo{L_{\bq}(\bh)-L_{\bq}(\bh')}}
		{N^{1/2} \min\set{ |dL_{\bq}/dh|
			: h \ge -c/\DELTA^{1/2}
		} }
	\le \f{O(1/\DELTA)}
		{\exp(1/\DELTA^{1/2})}\]
by combining \eqref{e:diff.tildeF.error.bound} with \eqref{e:tildeF.deriv.lbd}. Combining with
\eqref{e:ht.far.tail} gives the final claim.
\end{proof}
\end{cor}

\section{Change of measure}\label{s:tilt}

\noindent Let $\hbe_\mu$ be the $\mu$-th coordinate vector in $\R^M$, and $\dbe_i$ the $i$-th coordinate vector in $\R^N$. For any $\bx\in\R^M$ let $\bX_i$ denote the $M\times N$ matrix with $\bx$ in column $i$ and all other columns zero, $\bX_i\equiv \bx(\dbe_i)^\st$. For any $\by\in\R^N$ let $\bY_\mu$ denote the $M\times N$ matrix with $\by$ in row $\mu$ and all other rows zero, $\bY_\mu\equiv \hbe_\mu \by^\st$. Note
	\[(\bE,\bX_i)_{i\le N} 
	= \bE^\st\bx\,,\quad
	(\bE,\bY_\mu)_{\mu\le M}=\bE\by\,.\]
Recalling \eqref{e:restrict.j} and \eqref{e:restrict.k}, let $\bV_\mu\equiv\hbe_\mu\bv^\st$, and define similarly $\tilde{\bV}_\mu$ and $\bW_\mu$.

\subsection{Orthogonalized TAP equations} 

Recall that $(\br^{(s)})_{1\le s\le t}$ is an orthonormal basis for $\spn((\bmag^{(s)})_{1\le s\le t-1},\bmag)$, obtained by the change of coordinates (QR factorization) $\bM=\bR\bGa$ for $\bM$ as in \eqref{e:bM.before.change.basis}. Similarly, $(\bc^{(s)})_{1\le s\le t}$ is an orthonormal basis for $\spn((\bn^{(s)})_{1\le s\le t-1},\bn)$, obtained by the change of coordinates
$\bN=\bC\bXi$ for $\bN$ as in \eqref{e:bN.before.change.basis}. The constraints
\eqref{e:tap.as.lin.constraints.row} and \eqref{e:tap.as.lin.constraints.col} 
can be rewritten as
	{\setlength{\jot}{0pt}
	\begin{align*}
	\label{e:all.row.constraints}
	\tag{{\footnotesize$\color{DB}\usf{ROW}$}}
	\bE \br^{(s)}
	= \bx^{(s)} \
	&\textup{ for } 1\le s\le t\,,\\
	\label{e:all.col.constraints}
	\tag{{\footnotesize$\color{DB}\usf{COL}$}}
	\bE^\st\bc^{(s)} 
	=\by^{(s)} \
	&\textup{ for } 1\le s\le t\,,
	\end{align*}}%
where $\bx^{(s)}$ and $\by^{(s)}$ are defined by the same change of coordinates:
	\begin{align}\label{e:explicit.x}
	\begin{pmatrix}
	\bx^{(t)} & \cdots & \bx^{(1)}
	\end{pmatrix}
	&= 
	\begin{pmatrix}
	\displaystyle
	\f{\bh+(1-\bq)\bn}{\sqrt{\bq}}
	& \displaystyle\f{\bh^{(t-1)}
	+b^{(t-1)}\bn^{(t-2)}}{\sqrt{q_{t-1}}}
	& \cdots
	& \displaystyle
	\f{\bh^{(1)}+b^{(1)}\bn^{(0)}}{\sqrt{\TAPq}}
	\end{pmatrix}
	\bGa^{-1}\,,\\
	\label{e:explicit.y}
	\begin{pmatrix}
	\by^{(t)} & \cdots & \by^{(1)}
	\end{pmatrix}
	&=\begin{pmatrix}
	\displaystyle
	\f{\bH^{(t+1)}+d^{(t)}\bmag^{(t)}}
		{\sqrt{\bpsi}}
	&\displaystyle
	\f{\bH^{(t)}+d^{(t-1)}\bmag^{(t-1)}}
		{\sqrt{\psi_{t-1}}}
	& \cdots
	& \displaystyle
	\f{\bH^{(2)}+d^{(1)}\bmag^{(1)}}{\sqrt{
	\psi_1}}
	\end{pmatrix}
	\bXi^{-1}\,.
	\end{align}
Let $\bR^{(s)}_\mu\equiv\hbe_\mu(\br^{(s)})^\st$
and $\bC^{(s)}_i\equiv \bc^{(s)}(\dbe_i)^\st$, and define the subspaces
	{\setlength{\jot}{0pt}
	\begin{align*}
	\vv_\row &\equiv \spn(
	(\bR^{(s)}_\mu)_{1\le s\le t,\mu\le M})\,,\\
	\vv_\precol
	&\equiv\spn( (
	\bC^{(s)}_i
	)_{1\le s\le t,i\le N})\,.\end{align*}
	}%
For $\bv$ as given by \eqref{e:restrict.j}, define
$\bV^{(s)}\equiv \bc^{(s)} \bv^\st$,
and similarly $\tilde{\bV}^{(s)}$ and $\bW^{(s)}$. Observe that $\bV^{(s)},\tilde{\bV}^{(s)},\bW^{(s)}$ are all orthogonal to $\vv_\row$ and lie in $\vv_\precol$. For any $\bE$ satisfying the column constraints \eqref{e:all.col.constraints}, we must have
	{\setlength{\jot}{0pt}\begin{alignat*}{2}
	\tag{{\footnotesize$\color{DB}
	\usf{ADM}_\textsf{1}$}}
	\label{e:def.admiss.one}
	(\bE, \bV^{(s)})
	=(\bc^{(s)})^\st\bE\bv
	=(\by^{(s)},\bv) \
	&\textup{ for } 1\le s\le t\,,\\
	\tag{{\footnotesize$\color{DB}
	\usf{ADM}_\textsf{2}$}}
	(\bE,\tilde{\bV}^{(s)})
	=(\bc^{(s)})^\st\bE\tbv
	=(\by^{(s)},\tbv) \
	&\textup{ for } 1\le s\le t\,.
	\label{e:def.admiss.two}
	\end{alignat*}}%
We refer to \eqref{e:def.admiss.one} and \eqref{e:def.admiss.two}
 as the \emph{admissibility} constraints. Define the following subspaces of $\R^{M\times N}$:
	{\setlength{\jot}{0pt}\begin{alignat}{2}
	\label{e:v.prfone}
	\vv_{\admone} &=
	\spn( (\bV^{(s) })_{1\le s\le t} )\,,\quad
	& \vv_{\prfone} &=\spn((\bV_\mu)_{\mu\le M})
	= \vv_{\admone} \obot \vv_{\qrfone}\,,\\
	\nonumber
	\vv_{\admtwo} &=
	\spn( (\tilde{\bV}^{(s) })_{1\le s\le t} )\,,\quad
	& \vv_{\prftwo} 
	&=\spn((\tilde{\bV}_\mu)_{\mu\le M})
	= \vv_{\admtwo} \obot \vv_{\qrftwo}\,,
	\end{alignat}}%
We use $\obot$ to refer to the orthogonal sum of two subspaces, so $\vv_{\qrfone}$ is by definition the orthogonal complement of $\vv_{\admone}$ inside $\vv_{\prfone}$. 
For any subspaces $\vv_\sfont{X}$, $\vv_\sfont{Y}$
we abbreviate their direct (not necessarily orthogonal) sum as
	\[\vv_\sfont{XY}\equiv\vv_\sfont{X}\oplus\vv_\sfont{Y}\,.\]
Let $\vv_\adm\equiv \vv_{\admone\admtwo}$, $\vv_\qrf\equiv \vv_{\qrfone\qrftwo}$, and $\vv_\prf\equiv \vv_{\prfone\prftwo}= \vv_{\adm\qrf}$. 

\begin{lem}[subspace decompositions]
\label{l:extra.orth} 
Note $\vv_{\admone}\subseteq \vv_\adm\subseteq \vv_\con$. Let $\vv_{\colone}$ and $\vv_\col$ be defined by the relations
	\[\vv_\row\obot \vv_{\admone}\obot \vv_{\colone}
	=\vv_\con =
	\vv_\row\obot \vv_\adm \obot \vv_\col\,.\]
We then have the additional relations
	\[\vv_\row\obot \vv_{\prfone}\obot\vv_{\colone}
	=\vv_{\allone}\,,\quad
	\vv_\row\obot \vv_\prf\obot\vv_\col
	= \vv_\all\,.\]

\begin{proof}
We prove the claim for $\vv_\col$ only, as the claim for $\vv_{\colone}$ follows by a similar argument. Let $\bA^{(s)}_i$ be the projection of $\bC^{(s)}_i$ onto the orthogonal complement of $\vv_\row\obot \vv_\adm$,
	\[\bA^{(s)}_i
	=\bC^{(s)}_i
	-\sum_{r=1}^t
	\bigg\{\sum_{\mu=1}^M
		\Big(\bC^{(s)}_i,\bR^{(r)}_\mu\Big)\bR^{(r)}_\mu
		+
		\Big(\bC^{(s)}_i,\bV^{(r)}\Big)\bV^{(r)}
		+\Big(\bC^{(s)}_i,\bW^{(r)}\Big)\bW^{(r)}
		\bigg\}\,.\]
One can check that $(\bC^{(s)}_i,\bV^{(r)})=0$ for $r\ne s$, while
	\[\Big(\bC^{(s)}_i,\bV^{(s)}\Big)\bV^{(s)}
	=\sum_{j\le N} v_j
		\Big(\bC^{(s)}_i,\bC^{(s)}_j\Big) \bV^{(s)}
	= v_i \bV^{(s)}
	= \sum_{\mu=1}^M
		\Big(\bC^{(s)}_i,\bV_\mu\Big) \bV_\mu\]
(and similarly with $\bW$ in place of $\bV$). It follows that
	\[\bA^{(s)}_i
	=\bC^{(s)}_i
	-\sum_{\mu=1}^M\bigg\{ \sum_{r=1}^t
	\bigg[
	\Big(\bC^{(s)}_i,\bR^{(r)}_\mu\Big)\bR^{(r)}_\mu
	+
	\Big(\bC^{(s)}_i,\bV_\mu\Big)\bV_\mu
	+
	\Big(\bC^{(s)}_i,\bW_\mu\Big)\bW_\mu\bigg]
	\bigg\}\,,\]
from which we immediately see that each $\bA^{(s)}_i$ is in fact orthogonal to $\vv_\row\obot \vv_\prf$. The space $\vv_\col$ is spanned by the $\bA^{(s)}_i$, and it follows that
$\vv_\row\obot \vv_\prf\obot \vv_\col=\vv_\all$.
\end{proof}
\end{lem}

\subsection{Change of measure}

Recall from Remark~\ref{r:cond} that $\bE$ is a standard gaussian random variable in $\R^{M\times N}$ which is \emph{independent of} the background $\sigma$-field $\bgFilt$. Let $\oP$ denote the law of $\bE$ conditional on $\bgFilt$ (which, by independence, is the same as the unconditional law). We will make calculations under a tilted measure $\tP$, defined by
	\beq\label{e:tilt}
	\f{d\tP}{d\oP}
	=\exp\bigg\{(\bE,\bT)-\f{\|\bT\|^2}{2}\bigg\}\eeq
for $\bT$ to be specified below. Let $\tp$ denote the density of $\bE$ under $\tP$. 
 We indicate projections of $\bE$ onto subspaces by subscripts, e.g.\ the projection of $\bE$ onto $\vv_\row$ is denoted $\bE_\row$. We shall prove the following:

\begin{ppn}[change of measure formula]
\label{p:tilt} For any $\bE_{\row\prf}$-measurable event $\bsat$, and for any $\bT\in \vv_\con$,
	\begin{equation*}\label{e:changeofmsr}
	\tag{{\footnotesize$\color{magenta}
	\usf{S}|\usf{CR}$}}
	\oP(\bsat\giv E_\con)
	=\f{\tP(\bsat\giv E_\row)
		\tp_{\adm|\row\bsat}
		(E_\adm\giv E_\row)}
	{\tp_{\adm}(E_\adm)}\end{equation*}
where $\tp_{\adm|\row\bsat}(E_\adm\giv E_\row)$ is the density of $\bE_\adm$ under $\tP(\cdot\giv\bE_\row=E_\row,\bE_{\row\prf}\in\bsat)$. Likewise, for any $\bE_{\row\prfone}$-measurable event $\bsatone$,
	\begin{equation*}\label{e:changeofmsr.first}
	\tag{{\footnotesize$\color{magenta}
	\usf{S}_\textsf{1}|\usf{CR}$}}
	\oP(\bsatone\giv E_\con)
	=\f{\tP(\bsatone\giv E_\row)
	\tp_{\admone|\row\bsatone}
	(E_{\admone}\giv E_\row)}
	{\tp_{\admone}(E_{\admone})}
	\end{equation*}
where $\tp_{\admone|\row\bsatone}(E_{\admone}\giv E_\row)$ is the density of $\bE_{\admone}$ under $\tP(\cdot\giv\bE_\row=E_\row,\bE_{\row\prfone}\in\bsatone)$.\end{ppn}

\begin{proof} We prove \eqref{e:changeofmsr} only, as \eqref{e:changeofmsr.first} follows similarly. Write $\op$ and $\tp$ for densities under $\oP$ and $\tP$ respectively.
From Lemma~\ref{l:extra.orth} we have the orthogonal decompositions
	\[\vv_\row\obot \vv_\adm \obot \vv_\col=\vv_\con\,,\quad
	\vv_\row\obot \vv_\prf \obot \vv_\col=\vv_\all\,.\]
Recall also that $V_\prf=V_{\adm\qrf}$.
Therefore the left-hand side of \eqref{e:changeofmsr} can be written as
	\[\oP(\bsat\giv E_\con)
	= \f1{\op_{\row\adm\col}(E_{\row\adm\col})}
	\int \Ind{E_{\row\adm\qrf}\in\bsat}
	\op_{\row\adm\qrf\col}(E_{\row\adm\qrf\col})
		\,dE_\qrf\,.\]
The right-hand side above is unaffected if we replace $\op$ with $\tp$, because $\bT\in\vv_\con$ and so the Radon--Nikodym derivative \eqref{e:tilt} takes the same fixed value in both numerator and denominator. After replacing $\op$ with $\tp$, we apply the orthogonality result of Lemma~\ref{l:extra.orth} to simplify
	\beq\label{e:def.adm.cond.density}
	\oP(\bsat\giv E_\con)
	= \f{\tP(\bsat\giv E_\row)}
	{\tp_\adm(E_\adm)}
	\int\f{\Ind{E_{\row\adm\qrf}\in\bsat}
	\tp_{\adm\qrf}(E_{\adm\qrf})}
		{\tP(\bsat\giv E_\row)}
		\,dE_\qrf
	= \f{\tP(\bsat\giv E_\row)
		\tp_{\adm|\row\bsat}(E_\adm\giv E_\row)}
	{\tp_\adm(E_\adm)}\,,\eeq
where the last equality is by definition of $\tp_{\adm|\row\bsat}(E_\adm\giv E_\row)$. This proves \eqref{e:changeofmsr}.
\end{proof}

\noindent For $E\in\R^{M\times N}$ satisfying the equations \eqref{e:all.row.constraints} and \eqref{e:all.col.constraints} (and hence also satisfying \eqref{e:def.admiss.one} and \eqref{e:def.admiss.two}) we will have the projections $E_\row$, $E_\precol$, $E_\adm$, $E_\col$ constrained to fixed values, hereafter denoted $\bar{E}_\row$, $\bar{E}_\precol$, $\bar{E}_\adm$, $\bar{E}_\col$: for instance,
	\beq\label{e:def.Ebar.adm}
	\bar{E}_\adm
	= \sum_{s=1}^t \bigg\{
	(\bv,\by^{(s)})\bV^{(s)}
	+(\bw,\by^{(s)})\bW^{(s)}
	\bigg\}\,.\eeq
Recalling Definition~\ref{d:cube}, we now fix a pair
$(J,K)\in\bbH(\lm)\equiv\bbH_{M,\Nall,\DELTA,t}(\lm)$, and proceed to calculate the terms of Proposition~\ref{p:tilt} with $E_{\row\precol}=\bar{E}_{\row\precol}$ and
	{\setlength{\jot}{0pt}\begin{align}\nonumber
	\bsatone
	&=\sat_{J,\EPS}
	\equiv \sat_{J,\DELTA,t,\EPS}\,,\\
	\bsat
	&=\set{\bsatone,\sat_K}
	\equiv \set{\bsatone,\sat_{K,\DELTA,t}}\,.
	\label{e:events.B}
	\end{align}}%
This calculation will occupy Sections~\ref{s:row}~and~\ref{s:adm}.

\section{Probabilities given row constraints}
\label{s:row}

\noindent As above, fix $\kall$, $\aall$, $\sq\equiv q_{\star\kall}(\aall)$, $\spsi\equiv \psi_{\star\kall}(\aall)$, and $(J,K)\in\bbH(\lm)$. Recall from \eqref{e:PP} and \eqref{e:PP.pair} the definitions of $\PP_\star$ and $\PP(\lm)$. Now recall Proposition~\ref{p:tilt} and \eqref{e:events.B}. The main result of this section is that
	\begin{equation*}\label{e:conv.to.PP}
	\tag{{\footnotesize$\color{Pu}\usf{P}_\usf{1}
		|\usf{R}$}}
	\adjustlimits \limsup_{\DELTA\downarrow0} 
	\limsup_{t\uparrow\infty}
	\limsup_{\Nall\uparrow\infty}
	\bigg\{
	\f1{\Nall}
	\log\f{\tP(\sat_{J,\EPS}
		\giv\bar{E}_\row)}
	{\tp_{\admone}(\bar{E}_{\admone})}
	-\PP_\star\bigg\}=0
	\end{equation*}
for any fixed positive $\EPS$, and
moreover that
	\begin{equation*}
	\label{e:conv.to.PP.pair}
	\tag{{\footnotesize$\color{Pu}\usf{P}|\usf{R}^\usf{LD}$}}
	\adjustlimits
	\limsup_{\DELTA\downarrow0}
	\limsup_{t\uparrow\infty}
	\limsup_{\EPS\downarrow0} 
	\limsup_{\Nall\uparrow\infty}
	\bigg\{
	\f1{\Nall}
		\log\f{\tP(\sat_K\giv
		\bar{E}_\row,\sat_{J,\EPS})
	\,\tp_{\admone}(\bar{E}_{\admone})^2}
	{\tP(\sat_{J,\EPS}\giv\bar{E}_\row)
	\,\tp_{\adm}(\bar{E}_{\adm})}
	-2\PP_\star -
	\PP(\lm)
	\bigg\}
	\le 0\,.\end{equation*}
The combination of \eqref{e:conv.to.PP}~and~\eqref{e:conv.to.PP.pair}, proved in Corollary~\ref{c:PP}, will be sufficient for most $\lm
$. For small $\lm$ we will require a more precise estimates, given as
\eqref{e:PP.cst} and
\eqref{e:PP.clt.regime} in Corollary~\ref{c:PP.small.lm}.

\subsection{Admissibility given row constraints}

We begin by evaluating the probability to satisfy the admissibility constraints, which we recall are orthogonal to the row constraints. Recall $(J,K)\in\bbH(\lm)$, so we have the decompositions \eqref{e:restrict.j} and \eqref{e:restrict.k} and the overlap $\lm\equiv\lm_{J,K}$ from \eqref{e:c.lambda}. From now on, in the change of measure \eqref{e:tilt} we fix
	\beq\label{e:form.of.T}\bT
	=-\Big( N\bpsi(1-\bq)\Big)^{1/2} \bc
	\Big(\vartheta_1 \bv +\vartheta_2 \bw\Big)^\st\,.\eeq
where the $\vartheta_i$ satisfy the equations
	\beq\label{e:def.tilt}
	\vartheta_1
	=\bigg(
	\f{q_{J,t}(1-\bq)}{\bq(1-q_J)}
	\bigg)^{1/2}\,,\quad
	\lm\vartheta_1+\sqrt{1-\lm^2}\vartheta_2
	=\bigg(
	\f{q_{K,t}(1-\bq)}{\bq(1-q_K)}
	\bigg)^{1/2}\,.\eeq
Recalling \eqref{e:def.zeta.J}, we comment that $\bT\bv=-\bze^J$ and $\bT\tbv=-\bze^K$, so
$\bnu=(\bE-\bT)\bv$; these facts will be used later. We first give the computation for the probability density under $\tP$ on the admissibility constraints:

\begin{ppn}\label{p:adm}
Let $(J,K)\in\bbH(\lm)$. If $\bT$ is given by \eqref{e:form.of.T}, then
	\begin{equation*}\label{e:adm.givenr.first}
	\tag{{\footnotesize$\color{Re}
		\usf{A}_\textsf{1}|\usf{R}$}}
	\f1{\tp_{\admone}(\bar{E}_{\admone})}
	= \exp\bigg\{
	\f{N\bpsi(1-\bq)}{2}
	+O(\cst)\bigg\}\,.\end{equation*}
Additionally, with $c_\lm$ as in \eqref{e:c.lambda}, we have
	\begin{align*}
	\f1{\tp_\adm(\bar{E}_\adm)}
	&=
	\f1{\tp_{\admone}(\bar{E}_{\admone})}
	\exp\bigg\{ \f{N\bpsi(1-\bq)(c_\lm)^2}{2}
		+\f{O(\cst)}{\sqrt{1-\lm^2}}
		+\f{O(\cst)}{N(1-\lm^2)}
	\bigg\}\\
	&=\exp\bigg\{
		\f{N\bpsi(1-\bq)}{1+\lm} 
		+\f{O(\cst)}{\sqrt{1-\lm^2}}
		+\f{O(\cst)}{N(1-\lm^2)}
		\bigg\}\,.
	\tag{{\footnotesize$\color{Re}\usf{A}|\usf{R}$}}
	\label{e:adm.givenr}
	\end{align*}
Finally, we have the non-asymptotic bound
	\beq\label{e:adm.givenr.trivial.ubd}
	\f{\tp_{\admone}(\bar{E}_{\admone})}
		{\tp_\adm(\bar{E}_\adm)}
	\le \f1{(2\pi)^{t/2}}\,.\eeq
which holds for any $\lm$.
\end{ppn}

\begin{proof}
By definition of $\tP$ we have
	\begin{align}\nonumber
	\tp_{\admone}(\bar{E}_{\admone})
	&= \prod_{s=1}^t
	\varphi\Big( (\bv,\by^{(s)})-(\bT,\bV^{(s)}) 
		\Big)
	\,,\\
	\f{\tp_\adm(\bar{E}_\adm)}
		{\tp_{\admone}(\bar{E}_{\admone})}
	&=
	\prod_{s=1}^t
	\varphi\Big( (\bw,\by^{(s)})-(\bT,\bW^{(s)})
	\Big)\,,
	\label{e:adm.exact.dens}
	\end{align}
where the last equality immediately implies \eqref{e:adm.givenr.trivial.ubd}.
Recall that we specified $\bc^{(t)}=\bc$, so
$\bV^{(t)}=\bc\bv^\st$ and $\bW^{(t)}=\bc\bw^\st$, and
\eqref{e:form.of.T} can be rewritten as
	\[\bT=
	-\Big( N\bpsi(1-\bq)\Big)^{1/2}
	\Big(\vartheta_1 \bV^{(t)} +\vartheta_2 \bW^{(t)}\Big)^\st
	\,.\]
With $c_\lm$ as in \eqref{e:c.lambda}, assumption \eqref{e:restrict.j} implies
	\[\vartheta_1=1+\f{O(\cst)}{N}\,,\quad
	\vartheta_2
	=c_\lm+\f{O(\cst/N)}{\sqrt{1-\lm^2}}\,.\]
Next, from condition~\eqref{e:restrict.j.appx} we have that the scalar products products $(\bv,\bH^{(s)})$ (for $1\le s\le t-1$) and $(\bv,\bH^{(t+1)})$ are all bounded in absolute value by $C/N^{1/2}$. Recall that by definition $(\bmag^{(s)},\bv)=0$ for all $1\le s\le t$. Since $\bXi$ has all eigenvalues lower bounded by $1/\cst$ (see Corollary~\ref{c:Xi.nondegenerate}) we conclude via \eqref{e:explicit.y} that $|(\bv,\by^{(s)}| \le \cst /N^{1/2}$ for all $1\le s\le t$. Likewise we have $|(\tilde{\bv},\by^{(s)}| \le \cst /N^{1/2}$, so
	\[|(\bw,\by^{(s)})|
	= \bigg|
	\bigg( \f{\tilde{\bv}-\lm\bv
	}{\sqrt{1-\lm^2}},\by^{(s)} \bigg)
	\bigg|
	 \le \f{\cst/N^{1/2}}{\sqrt{1-\lm^2}}\]
for all $1\le s\le t$.
Combining the above estimates gives
	\begin{align*}
	\bigg\{(\bv,\by^{(s)})
		-(\bT,\bV^{(s)})\bigg\}^2
	&= \begin{cases}
	(\bv,\by^{(s)})^2
	= O(\cst/N)
	& \textup{for }1\le s\le t-1\,, \\
	O(\cst) + N\bpsi(1-\bq)
	&\textup{for }s=t\,,
	\end{cases}\\
	\bigg\{(\bw,\by^{(s)})
		-(\bT,\bW^{(s)})\bigg\}^2
	&= \begin{cases}\displaystyle
	(\bw,\by^{(s)})^2
	= \f{O(\cst)}{N(1-\lm^2)}
	& \textup{for }1\le s\le t-1\,, \\
	\displaystyle
	\f{O(\cst)}{N(1-\lm^2)})+
	\f{O(\cst)}{\sqrt{1-\lm^2}} + N\bpsi(1-\bq)
	&\textup{for }s=t\,.
	\end{cases}\end{align*}
Substituting into \eqref{e:adm.exact.dens} gives the claims \eqref{e:adm.givenr.first}
and \eqref{e:adm.givenr}.
\end{proof}

\subsection{Satisfiability given row constraints}
 Let
	\beq\label{e:def.xi.J}
	\bxi^J
	\equiv
	\f1{\sqrt{1-q_J}}\bigg(
	\kappa\oneM
	- \f{\sqrt{q_{J,t}}}{\sqrt{\bq}} \bh
	- \sum_{s=1}^{t-1} \gamma_{J,s} \bx^{(s)}
	-\bigg(1-\f{\sqrt{q_{J,t}}}{\sqrt{\bq}}\bigg)
	 \bigg(\f{\hat{\bE}\bhJ}{N^{1/2}}-\dkappa\bigg)
	\bigg)\eeq
and let $\bxi^K$ be analogously defined. Recall from \eqref{e:v.prfone} that $\bE_{\prfone}$ is the projection of $\bE$ onto the space $V_{\prfone}$ spanned by the elements $\bV_\mu\equiv\hbe_\mu\bv^\st$ for $\mu\le M$:
	\[\bE_{\prfone}
	= \sum_{\mu=1}^M
	(\bE,\bV_\mu)\bV_\mu
	= \sum_{\mu=1}^M
	\Big((\hbe_\mu)^\st\bE\bv\Big)
	\bV_\mu\,.\]
That is to say, the coordinates of $\bE_{\prfone}$ relative to the basis $(\bV_\mu)_{\mu\le M}$ is precisely $\bE\bv$. 

\begin{ppn}\label{p:basic.sat.giv.row}
Fix $(J,K)\in\bbH(\lm)$ and consider the events defined by \eqref{e:events.B}. Suppose in the change of measure \eqref{e:tilt} that we take $\bT$ as defined \eqref{e:form.of.T}. Then,
for $\bxi^J$ defined by \eqref{e:def.xi.J}, we have
	\begin{equation*}\label{e:sgivenr.first}
	\tag{{\footnotesize$\color{Bl}
		\usf{S}_\textsf{1}|\usf{R}$}}
	\f{\tP(\sat_{J,\EPS}\giv\bar{E}_\row)}
		{1-o_N(1)}
	= \tP(\sat_J\giv\bar{E}_\row)
	= \exp\bigg\{
	\Big(\oneM,\log\bPhi(\bxi^J)\Big)
	\bigg\}\,.\end{equation*}
Let $\bnu\equiv \bE\bv+\bze^J$ as in \eqref{e:def.zeta.J} (recall $\bE\bv$ is equivalent to $\bE_{\prfone}$, so $\bnu$ is equivalent to $\bE_{\prfone}$ plus a deterministic shift). Then
	\begin{equation*}\label{e:sat.givenr}
	\tag{{\footnotesize$\color{Bl}
		\usf{S}|\usf{R}$}}
	{\tP(\sat_K\giv\bar{E}_\row,\sat_{J,\EPS})}
	=
	\int
	\tp_{\prfone|\row\bsatone}
	(\bE_{\prfone}\giv\bar{E}_\row)
	\exp\bigg\{
	\bigg(\oneM,\log\bPhi\bigg(
		\f{\bxi^K-\lm\bnu}
	{\sqrt{1-\lm^2}}
	\bigg)\bigg)
	\bigg\}\,d\bE_{\prfone}\end{equation*}
where $\tp_{\prfone|\row\bsatone}(\cdot\giv\bar{E}_\row)$ denotes density under $\tP(\cdot\giv\bar{E}_\row,\bE_{\row\prfone}\in
\sat_{J,\EPS})$.
\end{ppn}

\begin{proof}
Recall from \eqref{e:restrict.j} that we decomposed 
$\bj\equiv J/N^{1/2}$ as
	\[\bj = \sqrt{q_{J,t}}\br^{(t)}
		+\sum_{s=1}^{t-1}
		\gamma_{J,s}\br^{(s)}
	 + \sqrt{1-q_J} \bv\]
where the $\gamma_{J,s}$ are $O(1/N)$ while $q_{J,t}$, $q_J$ lie within $O(1/N)$ of $\bq$. Recall from \eqref{e:basis.fixed} that we fixed $\br^{(t)}= \bmag/\sqrt{N\bq}$, and from \eqref{e:fixed.pt} that we chose $\bn\equiv F_{\bq}(\bh)$ to satisfy
	\[\f{\bE\bmag+\bE\bhJ}{N^{1/2}}
		-\dkappa
	= \bh+(1-\bq)\bn \,.\]
Consequently, the law under $\tP(\cdot\giv\bar{E}_\row)$ of $\bE\bj + \hat{\bE}\bhJ/N^{1/2}-\dkappa$ agrees with the law under $\oP$ of
	\begin{align*}
	\cJ
	&\equiv 
	\f{\sqrt{q_{J,t}}}{\sqrt{\bq}}
	\bigg(
	\bh+(1-\bq)\bn+\dkappa
	-\f{\hat{\bE}\bhJ}{N^{1/2}}
	\bigg)
	+\sum_{s=1}^t \gamma_{J,s} \bx^{(s)}
	+\sqrt{1-q_J} \Big(\bT\bv + \bE\bv\Big)
	+\f{\hat{\bE}\bhJ}{N^{1/2}}
	-\dkappa \\
	&=\f{\sqrt{q_{J,t}}}{\sqrt{\bq}}\bh
	+\sum_{s=1}^t \gamma_{J,s} \bx^{(s)}
	+\bigg(1-\f{\sqrt{q_{J,t}}}{\sqrt{\bq}}\bigg)
	 \bigg(\f{\hat{\bE}\bhJ}{N^{1/2}}-\dkappa\bigg)
	+\sqrt{1-q_J}
	\bigg(\Big(\bze^J + \bT\bv\Big)
	+ \bE\bv\bigg)\,,\end{align*}
for $\bze^J$ as defined by \eqref{e:def.zeta.J}.
The definition of $\bT$ \eqref{e:form.of.T} implies $\bT\bv=-\bze^J$. Rearranging gives
	\[\f{\kappa\oneM-\cJ}{\sqrt{1-q_J}}
	= \bxi^J - \bE\bv\]
for $\bxi^J$ as defined by \eqref{e:def.xi.J}. Recalling from \eqref{e:basic.sat.evt} the definition of $\sat_J\equiv\sat_{J,\DELTA,t}$, we have
	\[\tP\Big(\sat_J\,\Big|\,\bar{E}_\row\Big)
	= \oP\Big(\cJ\ge\kappa\oneM\Big)
	= \oP\Big(\bE\bv\ge\bxi^J\Big)
	=\exp\bigg\{
	\Big(\oneM,\log\bPhi(\bxi^J)\Big)
	\bigg\}\]
which is the second equality in \eqref{e:sgivenr.first}. The first equality in \eqref{e:sgivenr.first} follows by the law of large numbers: under the measure $\tP(\cdot\giv\bar{E}_\row,\sat_J)$, the event $\sat_{J,\EPS}$ holds with probability $1-o_N(1)$ for any $\EPS>0$. Similarly, using
the expansion of $\bk\equiv K/N^{1/2}$ from
 \eqref{e:restrict.k}, the law under $\tP(\cdot\giv\bar{E}_\row,\bE_{\prfone})$ of 
 $\bE\bk+\hat{\bE}\bhJ-\dkappa$ agrees with the law under $\oP$ of
	\begin{align*}
	\cK
	&\equiv 
	\f{\sqrt{q_{K,t}}}{\sqrt{\bq}}\bh
	+ \sum_{s=1}^t \gamma_{J,s} \bx^{(s)}
	+ \bigg(1-\f{\sqrt{q_{K,t}}}{\sqrt{\bq}}\bigg)
	 \bigg(\f{\hat{\bE}\bhJ}{N^{1/2}}-\dkappa\bigg)\\
	&\qquad + \sqrt{1-q_K} \bigg(
	\Big(\bze^K+ \bT\tbv\Big)+
	\lm\bnu + \sqrt{1-\lm^2} \bE\bw
	\bigg)\,,\end{align*}
since $\bT\bv=-\bze^J$ implies $\bnu=(\bE-\bT)\bv$. The definition of $\bT$ also implies $\bT\tilde{\bv}=-\bze^K$, and rearranging gives
	\[\f{\kappa\oneM-\cK}{\sqrt{1-q_K}}
	= \bxi^K
	-\lm\bnu
	-\sqrt{1-\lm^2}\bE\bw\,.\]
Consequently, for any realization of $\bE_{\prfone}$,
	\beq\label{e:bsattwo.calc}
	\tP(\sat_K
	\giv\bar{E}_\row,\bE_{\prfone})
	=\oP\bigg(\bE\bw\ge
	\f{\bxi^K-\lm\bnu}{\sqrt{1-\lm^2}}
	\givbigg \bnu\bigg)
	= \prod_{\mu=1}^M
	\bPhi\bigg(
	\f{\xi^K_\mu-\lm\nu_\mu}{\sqrt{1-\lm^2}}
	\bigg)\,.\eeq
The result \eqref{e:sat.givenr} follows by integrating \eqref{e:bsattwo.calc} over the law of $\bE_{\prfone}$.
\end{proof}

\noindent The main results of this section are summarized by the next two corollaries:

\begin{cor}\label{c:PP}
The limits \eqref{e:conv.to.PP} and \eqref{e:conv.to.PP.pair} hold.
\end{cor}

\begin{cor}\label{c:PP.small.lm} It holds uniformly over $J\in\bbH$ that
	\begin{equation*}\label{e:PP.cst}
	\tag{\footnotesize$\color{Pu}\usf{P}_\usf{1}
		|\usf{R}^\asymp$}
	\log \f{\tP(\sat_{J,\EPS}\giv\bar{E}_\row)}
	{\tp_{\admone}(\bar{E}_{\admone})}
	=
	\Big(\oneM,\log\bPhi(\bxi)\Big)
	+ \f{N\bpsi(1-\bq)}{2} + O(\cst)\,.
	\end{equation*}
There exists $\rho\equiv\rho_{M,N_\textup{all},\DELTA,t,\EPS}$ (depending on $J$ and $K$) 
and $\lmsmall>0$ such that for small $\lm$,
	\begin{equation*}
	\label{e:PP.clt.regime}
	\tag{\footnotesize$\color{Pu}\usf{P}
		|\usf{R}^\usf{CLT}$}
	\log\f{\tP(\sat_K\giv
		\bar{E}_\row,\sat_{J,\EPS})
	\,\tp_{\admone}(\bar{E}_{\admone})^2}
	{\tP(\sat_{J,\EPS}\giv\bar{E}_\row)
	\,\tp_{\adm}(\bar{E}_{\adm})}
	= \f{N \rho \lm^2}{2}
	+O\bigg( \cst\Big( 
		1+N|\lm|^3\Big)\bigg)\,.
	\end{equation*}
In addition, for a small positive $\lmsmall$, it holds uniformly over all $|\lm|\le\lmsmall$ and all $(J,K)\in\bbH(\lm)$ that
	\beq\label{e:conv.to.PP.second.deriv}
	\limsup_{\DELTA\downarrow0}
	\limsup_{t\uparrow\infty}
	\limsup_{\EPS\downarrow0}
	\limsup_{\Nall\uparrow\infty}
	\rho_{M,N_\textup{all},\DELTA,t,\EPS}
	\le \PP''(0)
	= 2\spsi(1-\sq)
	+\f{d^2\II(\lm)}{d\lm^2}\,\bigg|_{\lm=0} \,.\eeq
For a more explicit calculation of
$\PP''(0)$ we refer to 
\eqref{e:PP.second.derivative.ubd} below.
\end{cor}

\noindent We prove the corollaries via a general lemma that follows. Recall from \eqref{e:def.zeta.J} the definitions of $\bxi$, and let
	\beq\label{e:def.bxi}
	\bze
	\equiv \sqrt{1-\bq}\bn
	= \sqrt{1-\bq} F_{\bq}(\bh)
	= \cE\bigg(\f{\kappa\oneM-\bh}{\sqrt{1-\bq}}
	\bigg)= \cE(\bxi)\,.\eeq
For $\bn^{(t)}=F(\bh^{(t)})$ as defined by the TAP iteration (Section~\ref{ss:tap}), we let $\bxi^{(t)},\bze^{(t)}$ be defined by the equations
	\[\bze^{(t)}
	\equiv \sqrt{1-\bq} F_{\bq}(\bh^{(t)})
	= \cE\bigg(
		\f{\kappa\oneM-\bh^{(t)}}{\sqrt{1-\bq}}
	\bigg)
	= \cE(\bxi^{(t)})\,.\]
Recall from \eqref{e:def.xi} the definition of $\xi_{q,z}$, and let
	\[\zeta_{q,z}
	\equiv \sqrt{1-q} F_q(q^{1/2}z)
	= \cE\bigg( \f{\kappa-q^{1/2}z}{\sqrt{1-q}} \bigg)
	= \cE(\xi_{q,z})\,.\]
Recall also from
\eqref{e:def.zeta.J} that
	\[\bze^J\equiv
	\f{\sqrt{q_{J,t}}(1-\bq)\bn}
	{\sqrt{\bq}\sqrt{1-q_J}}\,,\quad
	\bnu\equiv \bE\bv+\bze^J\,,\]
and recall from \eqref{e:def.xi.J} that
	\[\bxi^J
	\equiv
	\f1{\sqrt{1-q_J}}\bigg(
	\kappa\oneM
	- \f{\sqrt{q_{J,t}}}{\sqrt{\bq}} \bh
	- \sum_{s=1}^{t-1} \gamma_{J,s} \bx^{(s)}
	\bigg)\,.\]
With the above notations we define the vectors
	{\setlength{\jot}{0pt}
	\begin{alignat*}{2}
	\bss^{J,K}
		&\equiv(\bxi^J,\bxi^K,\bze^J,\bze^K,\bh)\,,
		\quad
	& \XX^{J,K} &= (\bss^{J,K},
		\bE\bv+\bze^J)\,,\\
	\bss^\textup{apx}
	&\equiv
	(\bxi,\bxi, \bze, \bze, \bh)\,,\quad
	& \XX^\textup{apx}
	&\equiv(\bss^\textup{apx}, \bE\bv+\bze)\,,\\
	\bss^{(t)}
	&\equiv
	(\bxi^{(t)},\bxi^{(t)},
		\bze^{(t)},\bze^{(t)},
		\bh^{(t)})\,,\quad
	&\XX^{(t)}
		&\equiv(\bss^{(t)}, \bE\bv+\bze^{(t)})\,,\\
	s(z)
	&\equiv
	(\xi_{q,z},\xi_{q,z},
	\zeta_{q,z},\zeta_{q,z},z)\,,\quad
	& X(z,\nu)
	&\equiv(s(z),\nu)\,.
	\end{alignat*}}%
We regard $\XX^{J,K}$, $\XX^\textup{apx}$, and $\XX^{(t)}$ as elements of $(\R^6)^M$. For $J,K\in\bbH$ (see Definition~\ref{d:cube})
the vectors $\XX^{J,K}$ and $\XX^\textup{apx}$ are very close to one another (at scale $\asymp 1/N$); they are both moderately close to $\XX^{(t)}$ (at scale $\asymp 1$); and in the limit their empirical profiles are approximated by the distribution of $X(z,\nu)\in\R^6$ for $z$ standard gaussian and $\nu$ sampled from $\varphi_{\xi_{q,z}}(\nu)$. This is formalized by the following:

\begin{lem}\label{lem-CLT-approximation} Let $f_1,f_2: \mathbb R^6 \mapsto \mathbb R$ be twice differentiable functions, satisfying $|f_j(x)| + \ltwo{\nabla f_j(x)} \le c(1+\ltwo{x}^5)$ for some finite constant $c$. Then, for $J,K\in\bbH$ we have
	\beq\label{e:innerproduct.taylor.appx}
	\bigg|\Big(f_1(\XX^{J,K}),
	f_2(\XX^{J,K})\Big)
	-\Big(f_1(\XX^\textup{apx}),
	f_2(\XX^\textup{apx})\Big)
	\bigg| \le \cst\eeq
uniformly over the event $\sat_{J,\EPS}$. In the limit $\Nall\uparrow\infty$ followed by $\EPS\downarrow0$ followed by $t\uparrow\infty$ followed by $\DELTA\downarrow0$, we have
	\beq\label{e:innerproduct.state.limit}
	\f{(f_1(\XX^{J,K}),
	f_2(\XX^{J,K}))}{M}
	\longrightarrow
	\int \int f_1(X(z,\nu))
	f_2(X(z,\nu))
	\,\varphi_{\xi_{q,z}}(\nu)\,d\nu
	\,\varphi(z)\,dz
	\,\bigg|_{q=\sq}\,,\eeq
again uniformly over the event $\sat_{J,\EPS}$.

\begin{proof} 
We shall express $\XX\in(\R^6)^M$ as
$\XX=(\XX_\mu)_{\mu\le M}$ where
$\XX_\mu\equiv(\XX_{\mu,i})_{i\le 6}\in\R^6$.
We define the interpolating vector
$\XX(s)\equiv\XX^{J,K}+s(\XX^\textup{apx}-\XX^{J,K})$ for $0\le s\le1$. Then
	\begin{align}\nonumber
	&\bigg|\Big(f_1(\XX^{J,K}),
	f_2(\XX^{J,K})
		-f_2(\XX^\textup{apx}
	)\Big)\bigg|
	=\bigg| \sum_{\mu=1}^M
	f_1(\XX^{J,K}_\mu)
	\int_0^1
	\f{d}{ds}
	f_2(\XX_\mu(s))
	 \,ds\bigg|\\\nonumber
	&\qquad\le\bigg|
	\sum_{\mu=1}^M f_1(\XX^{J,K}_\mu)
	\bigg\{\max_{0\le s\le 1}
	\ltwo{\nabla f_2(\XX_\mu(s))}
	\bigg\}
	\ltwo{\XX^\textup{apx}_\mu
		-\XX^{J,K}_\mu}
	\bigg|\\
	&\qquad
	\le O(1)
	\bigg(
	\sum_{\mu=1}^M
	\bigg\{ 1
	+ \ltwo{\XX^{J,K}_\mu}^{20}
	+ \ltwo{\XX^\textup{apx}_\mu}^{20}
	\bigg\}
	\bigg)^{1/2}
	\ltwo{\XX^{J,K}-\XX^\textup{apx}}\,,
	\label{e:interpolation.prelim.bound}
	\end{align}
using the Cauchy--Schwarz inequality and the assumed bounds on $|f_1|$ and $\ltwo{\nabla f_2}$. Now,
for $\xi\ge0$, we can use the bound $\cE(\xi)\le\xi$
(see Lemma~\ref{l:EE}) to obtain
	\[\int \nu^{20}\,\varphi_\xi(\nu)\,d\nu
	=
	\int_{u\ge0}
	\f{(\xi+u)^{20}\cE(\xi)}
		{\exp(\xi u+u^2/2)}\,du
	\le O(1)
	\int_{u\ge0}
	\f{(\xi^{20}+u^{20})
	\xi}{\exp(u^2/2)}\,du
	\le O(1 + \xi^{21})\,.\]
For $\xi\le0$ we instead use the bound
$\varphi_\xi(\nu) \le 2\varphi(\nu)$ to obtain
	\[\int \nu^{20}\,\varphi_\xi(\nu)\,d\nu
	\le 2 \int \nu^{20}\varphi(\nu)\,d\nu
	\le O(1)\,.\]
Consequently it holds uniformly over all $\xi\in\R$ that
	\beq\label{e:nu.mmt.poly.bound}
	\int \nu^{20}\,\varphi_\xi(\nu)\,d\nu
	\le O(1 + \xi^{21})\,.\eeq
We then note that
	\begin{align}\nonumber
	&\sum_{\mu=1}^M\bigg\{
	1
	+ \ltwo{\XX^{J,K}_\mu}^{20}
	+ \ltwo{\XX^\textup{apx}_\mu}^{20}\bigg\}
	\le O(1)
	\sum_{\mu=1}^M\bigg\{1+
	\ltwo{\bss^{J,K}_\mu}^{20}
	+\ltwo{\bss^\textup{apx}_\mu}^{20}
	+(\nu_\mu)^{20}
	\bigg\}\\
	&\qquad
	\le O(1)
	\sum_{\mu=1}^M\bigg\{1+
	\ltwo{\bss^{J,K}_\mu}^{20}
	+\ltwo{\bss^\textup{apx}_\mu}^{20}
	+\f2N
	\int \nu^{20}\,\varphi_{\xi_\mu}(\nu)\,d\nu
	\bigg\}
	\le \cst \cdot N
	\label{e:use.nu.mmt.bound}
	\end{align}
where the first bound is by Minkowski's inequality,
the second bound is by Definition~\ref{d:profile}~part~\eqref{d:profile.ii}, and the last bound is by \cite[Lemma 1(c)]{bayati2011dynamics} (whose conditions are satisfied due to \eqref{e:nu.mmt.poly.bound}). 
On the other hand, by the bounds on $|q_J-\bq|$, $\gamma_{J,s}$ assumed for $J\in\bbH$
(Definition~\ref{d:cube}), we have
	\[\ltwo{\XX^{J,K}-\XX^\textup{apx}}
	\le \f{O(1)}{N}
	\bigg(\ltwo{\bh}
	+\ltwo{F_{\bq}(\bh)}
	+\sum_{0\le s\le t-1} \ltwo{\bx^{(s)}}\bigg)
	\le \f{\cst}{N^{1/2}}
	\,,\]
where the last inequality is obtained as follows:
we have
$\ltwo{\bh^{(s)}}+\ltwo{F_{\bq}(\bh^{(s)})}\le\cst \cdot N^{1/2}$ as a straightforward consequence of \cite[Lemma 1(c)]{bayati2011dynamics}. This implies $\ltwo{\bx^{(s)}}\le\cst\cdot N^{1/2}$ via \eqref{e:explicit.x} and our earlier observation that the matrix $\bGa$ is nondegenerate (all eigenvalues bounded away from zero in the limit $N\uparrow\infty$). Lastly, combining with Corollary~\ref{c:kr} gives also $\ltwo{\bh}+ \ltwo{F_{\bq}(\bh)}\le\cst \cdot N^{1/2}$. Combining the last two bounds, we see that \eqref{e:interpolation.prelim.bound}
is upper bounded by $\cst$. A very similar argument gives also
	\[\bigg|\Big(
	f_1(\XX^{J,K})
	-f_1(\XX^\textup{apx})
		,
		f_2(\XX^\textup{apx}
	)\Big)\bigg|\le\cst\,,\]
so we obtain \eqref{e:innerproduct.taylor.appx}. The derivation of \eqref{e:interpolation.prelim.bound} also gives
	\[\bigg|\Big(f_1(\XX^\textup{apx}),
	f_2(\XX^\textup{apx})
		-f_2(\XX^{(t)}
	)\Big)\bigg|
	\le O(1)
	\bigg(
	\sum_{\mu=1}^M
	\bigg\{ 1
	+ \ltwo{\bX^\textup{apx}_\mu}^{4p}
	+ \ltwo{\bX^{(t)}_\mu}^{4p}
	\bigg\}
	\bigg)^{1/2}
	\ltwo{\XX^\textup{apx}-\XX^{(t)}}
	\,.\]
It follows from Corollary~\ref{c:kr}
that $\ltwo{\XX^\textup{apx}-\XX^{(t)}}\le N^{1/2}o_{1/\DELTA,t}(1)$ where we use $o_{1/\DELTA,t}(1)$ to indicate an error term tending to zero in the limit $t\uparrow\infty$ followed by $\DELTA\downarrow0$. Therefore
	\beq\label{e:innerproduct.state.apx.t}
	\f{
	(f_1(\XX^\textup{apx}),f_2(\XX^\textup{apx}))
	}{M} 
	=
	\f{(f_1(\XX^{(t)}),f_2(\XX^{(t)}))}{M}
	+ o_{1/\DELTA,t}(1)\,.\eeq
For any constant $R$ let us define
	\begin{align*}
	(f_1(\XX^{(t)}),f_2(\XX^{(t)}))_{\le R}
	&\equiv
	\sum_{\mu=1}^M
	\mathbf{1}\bigg\{\begin{array}{c}
		\ltwo{\bss^{(t)}_\mu}\le R\\
		\textup{ and }|\nu_\mu|\le R\end{array}\bigg\}
	f_1(\XX^{(t)}_\mu)
	f_2(\XX^{(t)}_\mu)\,,\\
	(f_1(\XX^{(t)}),f_2(\XX^{(t)}))_{\ge R}
		&\equiv\sum_{\mu=1}^M
		\mathbf{1}\bigg\{\begin{array}{c}
		\ltwo{\bss^{(t)}_\mu}\ge R\\
		\textup{ or }|\nu_\mu|\ge R\end{array}\bigg\}
	\Big|f_1(\XX^{(t)}_\mu)
	f_2(\XX^{(t)}_\mu)\Big|\end{align*}
(note the absolute values in the last expression), as well as
	\[I_{q,\le R}
	\equiv
	\int\int
	\mathbf{1}\bigg\{\begin{array}{c}
		\ltwo{s(z)}\le R\\
		\textup{and }|\nu|\le R\end{array}\bigg\}
	f_1(X(z,\nu))f_2(X(z,\nu))
	\,\varphi(\nu)\,d\nu\,\varphi(z)\,dz\,.\]
Write $I_q\equiv I_{q,\le\infty}$.
It follows from
\cite[Lemma 1(c)]{bayati2011dynamics} 
together with
 Definition~\ref{d:profile}~part~\eqref{d:profile.iii} that for any finite $R$,
	\[\limsup_{\EPS\downarrow0}
	\limsup_{\Nall\uparrow\infty}
	\bigg\{
	\f{(f_1(\XX^{(t)}),f_2(\XX^{(t)}))_{\le R}}{M}
	-I_{\bq,\le R}\bigg\}=0\,.\]
On the other hand, arguing similarly as in \eqref{e:use.nu.mmt.bound} we have
	\[(f_1(\XX^{(t)}),f_2(\XX^{(t)}))_{\ge R}
	\le O(1)\sum_{\mu=1}^M
	\mathbf{1}\bigg\{\begin{array}{c}
		\ltwo{\bss^{(t)}_\mu}\ge R\\
		\textup{ or }|\nu_\mu|\ge R\end{array}\bigg\}
	\bigg\{
	1 + \ltwo{\bss^{(t)}}^{10}
		+|\nu_\mu|^{10}
	\bigg\}\,,\]
so by \cite[Lemma 1(c)]{bayati2011dynamics}
together with Definition~\ref{d:profile}~part~\eqref{d:profile.ii} we conclude
	\[\limsup_{R\uparrow\infty}
	\limsup_{\Nall\uparrow\infty}
	\f{(f_1(\XX^{(t)}),f_2(\XX^{(t)}))_{\ge R}}{M}
	=0\,.\]
Combining these bounds and sending $R\uparrow\infty$ gives
	\[\limsup_{\EPS\downarrow0}
	\limsup_{\Nall\uparrow\infty}
	\bigg\{
	\f{(f_1(\XX^{(t)}),f_2(\XX^{(t)}))}{M}
	-I_{\bq}\bigg\}=0\,.\]
In the limit $t\uparrow\infty$ followed by $\DELTA\downarrow0$ we have $I_{\bq}\to I_{\sq}$, therefore
	\[\limsup_{\DELTA\downarrow0}
	\limsup_{t\uparrow\infty}
	\limsup_{\EPS\downarrow0}
	\limsup_{\Nall\uparrow\infty}
	\bigg\{
	\f{(f_1(\XX^{(t)}),f_2(\XX^{(t)}))}{M}
	-I_{\sq}\bigg\}=0\,.\]
Combining with \eqref{e:innerproduct.state.apx.t} yields \eqref{e:innerproduct.state.limit}, concluding the proof.
\end{proof}
\end{lem}

\noindent We now apply Lemma~\ref{lem-CLT-approximation} to deduce the preceding corollaries.

\begin{proof}[Proof of Corollary~\ref{c:PP}]
It follows from \eqref{e:adm.givenr.first} and \eqref{e:sgivenr.first} that
	\[\f{\tP(\sat_{J,\EPS}
		\giv\bar{E}_\row)}
		{\tp_{\admone}(\bar{E}_{\admone})}
	=
	\exp\bigg\{
	\f{N\bpsi(1-\bq)}{2}
	+\Big(\oneM,\log\bPhi(\bxi^J)\Big)
	+O(\cst)
	\bigg\}\,.\]
The limit \eqref{e:conv.to.PP} follows by applying Lemma~\ref{lem-CLT-approximation}. Similarly, for any constant $\lmbig<1$, the limit \eqref{e:conv.to.PP.pair} holds uniformly over $\lmin\le\lm\le\lmbig$ by combining \eqref{e:adm.givenr}~and~\eqref{e:sat.givenr} with Lemma~\ref{lem-CLT-approximation}. For $\lmbig\le\lm\le1$ we simply use \eqref{e:adm.givenr.trivial.ubd} to obtain
	\[\f{\tP(\sat_K\giv
		\bar{E}_\row,\sat_{J,\EPS})}
	{\tp_{\adm}(\bar{E}_{\adm})
		/\tp_{\admone}(\bar{E}_{\admone})}
	\le\f{\tp_{\admone}(\bar{E}_{\admone})}
		{\tp_{\adm}(\bar{E}_{\adm})}
	\le \f1{(2\pi)^{t/2}}\,,\]
which implies for $\lmbig\le\lm\le1$ that
	\[\limsup_{\Nall\uparrow\infty}
	\f1{\Nall}\log
	\f{\tP(\sat_K\giv
		\bar{E}_\row,\sat_{J,\EPS})}
	{\tp_{\adm}(\bar{E}_{\adm})
		/\tp_{\admone}(\bar{E}_{\admone})}\le0\,.\]
By taking $c\downarrow0$ we conclude that \eqref{e:conv.to.PP.pair} holds for all $\lmin\le\lm\le1$.\end{proof}

\noindent We emphasize that the profile restriction (Definition~\ref{d:profile}~part~\eqref{d:profile.iii}) was crucially used in the proof of Lemma~\ref{lem-CLT-approximation}, which in turn is used to prove Corollary~\eqref{c:PP}. Without this restriction we do not expect the limit \eqref{e:conv.to.PP.pair} to be correct for general $\lm\ne0$.

\begin{proof}[Proof of Corollary~\ref{c:PP.small.lm}]
It follows from
\eqref{e:adm.givenr.first} and \eqref{e:adm.givenr} that for small $\lm$ we have
	\beq\label{e:adm.expand}
	\log
	\f{\tp_{\admone}(\bar{E}_{\admone})^2}
	{\tp_{\adm}(\bar{E}_{\adm})}
	= -\f{N\bpsi(1-\bq) \lm}{1+\lm}
		+O(\cst)
	= N\bpsi(1-q) 
	(-\lm + \lm^2 )
	+O\bigg( \cst\Big(
		1+N|\lm|^3\Big)\bigg)
		\,.\eeq
From the explicit expression~\eqref{e:bsattwo.calc} for
the conditional probability $\tP(\sat_K\giv\bar{E}_\row,\bE_{\prfone})$, we derive 
	\begin{align}
	\label{e:bsattwo.calc.first.deriv}
	\f{d\log\tP(\sat_K
	\giv\bar{E}_\row,\bE_{\prfone})
	}{d\lm}
	\,\bigg|_{\lm=0}
	&= \Big(\cE(\bxi^K)\,,\bnu\Big)\,,\\
	\f{d^2\log\tP(\sat_K
	\giv\bar{E}_\row,\bE_{\prfone})}{d\lm^2}
	\,\bigg|_{\lm=0}
	&=-\Big(\cE(\bxi^K)\,,\bxi^K\Big)
	-\Big(\cE'(\bxi^K)\,,\bnu^2
	\Big)\,.
	\label{e:bsattwo.calc.second.deriv}
	\end{align}
Now recall from Proposition~\ref{p:basic.sat.giv.row} that we obtained the result for $\tP(\sat_K\giv\bar{E}_\row,\sat_{J,\EPS})$ in \eqref{e:sat.givenr} simply by integrating \eqref{e:bsattwo.calc} over the law of $\bE_{\prfone}$. Combining with \eqref{e:bsattwo.calc.first.deriv} gives
	\begin{align*}
	\f{d\tP(\sat_K
	\giv\bar{E}_\row,\sat_{J,\EPS})}{d\lm}
	\,\bigg|_{\lm=0}
	&=\int\Big(\cE(\bxi^K),\bnu\Big)
	\tp_{\prfone|\row\bsatone}
	(\bE_{\prfone}\giv\bar{E}_\row)
	\exp\Big\{\Big(\oneM,\log\bPhi(\bxi^K)
		\Big)\Big\}
	\,d\bE_{\prfone}\\
	&=\exp\Big\{\Big(\oneM,\log\bPhi(\bxi^K)
		\Big)\Big\}
	\bigg(\cE(\bxi^K),
	\int \bnu\,
	\tp_{\prfone|\row\bsatone}
	(\bE_{\prfone}\giv\bar{E}_\row)
	\,d\bE_{\prfone}
	\bigg)\,,\end{align*}
while we see from \eqref{e:sat.givenr} that
	\beq\label{e:bsattwo.calc.integral}
	\tP(\sat_K
	\giv\bar{E}_\row,\sat_{J,\EPS})
	\,\bigg|_{\lm=0}
	=\exp\Big\{\Big(\oneM,\log\bPhi(\bxi^K)
		\Big)\Big\}
	\int
	\tp_{\prfone|\row\bsatone}
	(\bE_{\prfone}\giv\bar{E}_\row)\,d\bE_{\prfone}
	= \exp\Big\{\Big(\oneM,\log\bPhi(\bxi^K)
		\Big)\Big\}\,.\eeq
Dividing these two equations gives 
	\beq\label{e:bsattwo.calc.first.deriv.integral}
	\f{d\log\tP(\sat_K
	\giv\bar{E}_\row,\sat_{J,\EPS})}{d\lm}
	\,\bigg|_{\lm=0}
	=
	\bigg(\cE(\bxi^K),
	\int \bnu\,
	\tp_{\prfone|\row\bsatone}
	(\bE_{\prfone}\giv\bar{E}_\row)
	\,d\bE_{\prfone}
	\bigg)\,.\eeq
Likewise we can use \eqref{e:bsattwo.calc.second.deriv} to obtain the second derivative
	\beq\label{e:bsattwo.calc.second.deriv.integral}
	\f{d^2\log\tP(\sat_K
	\giv\bar{E}_\row,\sat_{J\EPS})}{d\lm^2}
	\,\bigg|_{\lm=0}
	=
	-\Big(\cE(\bxi^K)\,,\bxi^K\Big)
	-\bigg(\cE'(\bxi^K)\,,
	\int\bnu^2\,
	\tp_{\prfone|\row\sat_{J\EPS}}
	(\bE_{\prfone}\giv\bar{E}_\row)
	\,d\bE_{\prfone}
	\bigg)\,.\eeq
For any $\EPS>0$, the measure
$\tP(\cdot\giv\bar{E}_\row,\sat_J)$ is exponentially well concentrated on the event $\sat_{J,\EPS}$.
As a result, in the above integrals, the density $\tp_{\prfone|\row\bsatone}$ can be replaced by $\tp_{\prfone|\row\sat_J}$ with negligible error: for instance, \eqref{e:bsattwo.calc.first.deriv.integral} can be expressed as
	\[\bigg(\cE(\bxi^K),
	\int 
	\f{\mathbf{1}_{\sat_{J\EPS}}
	\bnu
	\,
	\tp_{\prfone|\row\sat_J}
	(\bE_{\prfone}\giv\bar{E}_\row)}
	{\tP(\sat_{J,\EPS}
	\giv\bar{E}_\row,\sat_J)}
	\,d\bE_{\prfone}
	\bigg)
	= 
	\bigg(\cE(\bxi^K),
	\int 
	\bnu
	\,
	\tp_{\prfone|\row\sat_J}
	(\bE_{\prfone}\giv\bar{E}_\row)
	\,d\bE_{\prfone}
	\bigg)
	+\f{O(1)}{\exp(N/\cst)}\]
for a positive constant $\cst$, and an analogous approximation holds for \eqref{e:bsattwo.calc.second.deriv.integral}.
The mean and second moment of $\bnu$ under the measure 
$\tP(\cdot\giv\bar{E}_\row,\sat_J)$ are given exactly by
	\begin{align*}
	\int 
	\bnu
	\,
	\tp_{\prfone|\row\sat_J}
	(\bE_{\prfone}\giv\bar{E}_\row)
	\,d\bE_{\prfone}
	=\int \bnu\,\varphi_{\bxi^J}(\bnu)\,d\bnu
	&= \cE(\bxi^J)\,,\\
	\int 
	\bnu^2
	\,
	\tp_{\prfone|\row\sat_J}
	(\bE_{\prfone}\giv\bar{E}_\row)
	\,d\bE_{\prfone}
	=\int \bnu^2\,\varphi_{\bxi^J}(\bnu)\,d\bnu
	&=\oneM+ \bxi^J \cE(\bxi^J)\,.\end{align*}
Substituting into \eqref{e:bsattwo.calc.first.deriv.integral} and \eqref{e:bsattwo.calc.second.deriv.integral},
and noting that $\cE'(x)=\cE(x)(\cE(x)-x)$, we obtain
	\begin{align}
	\label{e:bsattwo.calc.first.deriv.eval}
	\f{d\log\tP(\sat_K
	\giv\bar{E}_\row,\sat_{J\EPS})}{d\lm}
	\,\bigg|_{\lm=0}
	&=\Big(\cE(\bxi^K),
	\cE(\bxi^J)\Big)+O(\cst)
	= \ltwo{\cE(\bxi)}^2+O(\cst)
	\,,\\
	\f{d^2\log\tP(\sat_K
	\giv\bar{E}_\row,\sat_{J\EPS})}{d\lm^2}
	\,\bigg|_{\lm=0}
	&=-\bigg(\cE(\bxi^K),\bxi^K+
	\Big(\cE(\bxi^K)-\bxi^K\Big)
	\Big(\oneM+\bxi^J\cE(\bxi^J)\Big)
	\bigg)
	+O(\cst)\,,
	\label{e:bsattwo.calc.second.deriv.eval}
	\end{align}
where the last step of \eqref{e:bsattwo.calc.first.deriv.eval} is by the approximation result \eqref{e:innerproduct.taylor.appx} from Lemma~\ref{lem-CLT-approximation}. It follows from the definition~\eqref{e:def.bxi} that $\ltwo{\cE(\bxi)}^2=(1-\bq)\ltwo{\bn}^2 =N(1-\bq)\bpsi$. Combining \eqref{e:sgivenr.first} and \eqref{e:bsattwo.calc.integral} gives
	\beq\label{e:ratio.at.lambda.zero}
	\log\f{\tP(\sat_K
	\giv\bar{E}_\row,\sat_{J,\EPS})}
	{\tP(\sat_{J,\EPS}\giv\bar{E}_\row)}
	\bigg|_{\lm=0}
	=\Big(\oneM,\log\bPhi(\bxi^J)
		-\log\bPhi(\bxi^K)\Big)
	= O(\cst)\,,\eeq
where the last step is another application of \eqref{e:innerproduct.taylor.appx}. Denote the right-hand side of \eqref{e:bsattwo.calc.second.deriv.eval} as $d\equiv d_{M,\Nall,\DELTA,t,\EPS}$. Combining \eqref{e:bsattwo.calc.first.deriv.eval}, \eqref{e:bsattwo.calc.second.deriv.eval}, and \eqref{e:ratio.at.lambda.zero} into a Taylor expansion gives (for small $\lm$)
	\[\log\f{\tP(\sat_K
	\giv\bar{E}_\row,\sat_{J,\EPS})}
	{\tP(\sat_{J,\EPS}\giv\bar{E}_\row)}
	=N\bpsi(1-\bq)\lm
	+\f{N d\lm^2}{2}
	+O\bigg( \cst\Big(
		1+N|\lm|^3\Big)\bigg)\,,\]
and we note that the first-order term exactly cancels with that of \eqref{e:adm.expand}. Altogether we obtain
	\[\log\f{\tP(\sat_K\giv
		\bar{E}_\row,\sat_{J,\EPS})
	\tp_{\admone}(\bar{E}_{\admone})^2}
	{\tP(\sat_{J,\EPS}\giv\bar{E}_\row)
		\tp_{\adm}(\bar{E}_{\adm})}
	= \f{N\lm^2 [d+2\bpsi(1-\bq)]}{2}
	+O\bigg( \cst\Big(
		1+N|\lm|^3\Big)\bigg)\,.\]
By the limit result \eqref{e:innerproduct.state.limit} from Lemma~\ref{lem-CLT-approximation} we have
	\[ \adjustlimits
	\limsup_{\DELTA\downarrow0}
	\limsup_{t\uparrow\infty}
	\limsup_{\EPS\downarrow0}
	\limsup_{\Nall\uparrow\infty} 
	\bigg|\bigg\{
	d_{M,\Nall,\DELTA,t,\EPS}
	-
	\f{d^2\II(\lm)}{d\lm^2}\,\bigg|_{\lm=0}
	\bigg\}
	\bigg|=0\,.\]
The result follows by taking $\rho_{M,\Nall,\DELTA,t,\EPS}=d_{M,\Nall,\DELTA,t,\EPS}+2\bpsi(1-\bq)$.
\end{proof}

\section{Admissibility given row and satisfiability constraints}
\label{s:adm}

\noindent We continue to consider a fixed pair $(J,K)\in\bbH(\lm)$. In this section we shall estimate the remaining factors of Proposition~\ref{p:tilt} that were not computed in Section~\ref{s:row}, namely $\tp_{\adm|\row\bsat}(\bar{E}_\adm\giv\bar{E}_\row)$ as defined by the last equality in \eqref{e:def.adm.cond.density}, and $\tp_{\admone|\row\bsatone}(\bar{E}_{\admone}\giv\bar{E}_\row)$ which is defined analogously as
	\beq\label{e:def.admone.cond.density}
	\tp_{\admone|\row\bsatone}(\bar{E}_{\admone}\giv\bar{E}_\row)
	\equiv
	\int\f{\Ind{E_{\row\admone\qrfone}\in\bsatone}
	\tp_{\adm\qrf}(E_{\admone\qrfone})}
		{\tP(\bsatone\giv E_\row)}
		\,dE_\qrf\,.\eeq
Recall that the events $\bsatone,\bsat$ are defined by \eqref{e:events.B}. 

\subsection{CLT and moderate deviations regime} The first main result of this section is the following:

\begin{ppn}\label{p:adm.clt}
For $\bsatone,\bsat$ as in \eqref{e:events.B},
consider 
$\tp_{\adm|\row\bsat}(\bar{E}_\adm\giv\bar{E}_\row)$
and $\tp_{\admone|\row\bsatone}(\bar{E}_{\admone}\giv\bar{E}_\row)$ as defined by \eqref{e:def.adm.cond.density} and \eqref{e:def.admone.cond.density}.
\begin{enumerate}[a.]
\item \label{p:adm.clt.first}
For any $J\in\bbH$ we have
	\begin{equation*}\label{e:adm.givensr.first}
	\tag{{\footnotesize$\color{Gr}
		\usf{A}_\textsf{1}|\usf{RB}_\textsf{1}$}}
	\f1{\cst}
	\le\tp_{\admone|\row\bsatone}
		(\bar{E}_{\admone}\giv\bar{E}_\row)
	\equiv
	\int \f{\Ind{ \bE_{\row\admone\qrfone} \in\bsatone}
	\tp_{\admone\qrfone}(\bE_{\admone\qrfone})
		\,d\bE_{\qrfone}}
	{\tP(\bsatone\giv\bE_\row)}
	\,\bigg|_{\bE_{\row\admone}=\bar{E}_{\row\admone}}
	\le \cst\,.
	\end{equation*}
\item \label{p:adm.clt.pair}
If $(J,K)\in\bbH(\lm)$ with $\lmin\le\lm\le\lmbig$ for a constant $\lmbig<1$, then
	\begin{equation*}\label{e:adm.givensr.clt}
	\tag{{\footnotesize$\color{Gr}
		\usf{A}|\usf{RB}^\textsf{CLT}$}}
	\tp_{\adm|\row\bsat}
		(\bar{E}_\adm\giv\bar{E}_\row)
	\equiv
	\int \f{\Ind{ \bE_{\row\adm\qrf} \in\bsat}
	\tp_{\adm\qrf}(\bE_{\adm\qrf})
		\,d\bE_{\qrf}}
	{\tP(\bsat\giv\bE_\row)}
	\,\bigg|_{\bE_{\row\adm}=\bar{E}_{\row\adm}}
	\le \cst\,.\end{equation*}
\end{enumerate}
\end{ppn}

\noindent We first supply a technical lemma to be used in the proof. For $1\le s\le t-1$ define the matrices
	{\setlength{\jot}{0pt}\begin{align*}
	\bh(1:s)
	&\equiv
	\begin{pmatrix}
	\bh^{(1)} &\cdots &\bh^{(s)}
	\end{pmatrix}\,,\\
	\bH(1:s)
	&\equiv
	\begin{pmatrix}
	\bH^{(2)} &\cdots &\bH^{(s+1)}
	\end{pmatrix}\,,\end{align*}}%
so $\bh(1:s)$ is $M\times s$ while
$\bH(1:s)$ is $N\times s$. Define
	{\setlength{\jot}{0pt}\begin{align*}
	\bh(1:t)
	&\equiv
	\begin{pmatrix}
	\bh^{(1)} &\cdots
	&\bh^{(t-1)} & \bh^{(t)}
	\end{pmatrix}\,,\\
	\bH(1:t)
	&\equiv
	\begin{pmatrix}
	\bH^{(2)} &\cdots
	&\bH^{(t)}&\bH^{(t+1)}
	\end{pmatrix}\,.\end{align*}}%
Recall that $\hbe_\mu$ denotes a standard basis vector in $\R^M$ while $\dbe_i$ denotes a standard basis vector in $\R^N$. Recall from Remark~\ref{r:cst} our notational convention on quantities $\cst$ and $\cst_\gamma$.

\begin{lem}\label{lem-CLT-empirical}
For any $\gamma>0$ and any cube $A_\gamma \subseteq \mathbb R^t$ of side length $\gamma$ and within distance $1/\gamma$ of the origin,
	{\setlength{\jot}{0pt}\begin{align}
	M^{-1} \#\{1\leq \mu\leq M:
	(\hbe_\mu)^\st \bh(1:t)\in A_\gamma\}
	&\ge 1/\cst_\gamma\,,
	\label{eq-CLT-empirical-h}\\
	N^{-1} \#\{1\leq i\leq N: 
	(\dbe_i)^\st \bH(1:t)
	\in A_\gamma\}
	&\ge 1/\cst_\gamma
	\label{eq-CLT-empirical-H}
	\end{align}}%
with high probability in the limit $\Nall\uparrow\infty$.

\begin{proof}
It follows from \cite[Lemma 1(c)]{bayati2011dynamics} that
	\beq\label{eq-H-norm}
	\sum_{s=1}^t
	\f{\ltwo{\bh^{(s)}}^2}{N}
	+\sum_{s=1}^t\f{\ltwo{\bH^{(s)}}^2}{M}
	\leq\cst\,.\eeq
Let $A_{\gamma,s}\subseteq\R^s$ be the projection of $A_\gamma$ onto the first $s$ coordinates. We will consider
	{\setlength{\jot}{0pt}\begin{align*}
	\textsf{frac}[\bh(1:s)]
	&\equiv M^{-1} \#\set{
	1\le\mu\le M:
	(\hbe_\mu)^\st \bh(1:s)
	\in A_{\gamma,s}}\,,\\
	\textsf{frac}[\bH(1:s)]
	&\equiv N^{-1} \#\set{
	1\le i\le N:
	(\dbe_i)^\st \bH(1:s)
	\in A_{\gamma,s}}\,,\end{align*}}%
taking by definition 
$\textsf{frac}[\bh(1:0)]\equiv1$
and $\textsf{frac}[\bH(1:0)]\equiv1$. 
Let us write $\bmag^{(s)\parallel}$ for the orthogonal projection of $\bmag^{(s)}$ onto $\spn(\bmag^{(r)})_{r\le s-1}$, and
denote $\bmag^{(s)\perp} \equiv \bmag^{(s)}-\bmag^{(s)\parallel}$. For $1\le s\le t$ define (cf.\ \eqref{e:filt.t})
	\[\Filt_{\textsc{tap},s}
	\equiv\sigma\Big(\bn^{(1)},\bmag^{(2)},\bn^{(2)},
	\ldots,\bmag^{(s)}\Big)\,.\]
We write $\cst_{\textsc{tap},s}$ 
for $\Filt_{\textsc{tap},s}$-measurable random variables (not depending on $\gamma$) that stay stochastically bounded as $N\uparrow\infty$. We indicate dependence on $\gamma$ by writing $\cst_{\textsc{tap},s,\gamma}$. By \cite[Lemma 1(a)]{bayati2011dynamics},
conditional on $\Filt_{\textsc{tap},s}$,
 the next iterate $\bh^{(s+1)}$ is distributed as
	\[\sum_{r\le s}
	\alpha_{s,r} \bh^{(r)}
	+\bE^\textup{ind}
	\bmag^{(s)\perp}
	+ \err^{(s)}\]
where $\alpha_{s,r}$ are $\Filt_{\textsc{tap},s}$-measurable coefficients satisfying $|\alpha_{s,r}|\le \cst_{\textsc{tap},s}$; $\bE^\textup{ind}$ denotes an independent copy of $\bE$; and we have $\|\err^{(s)}\|_\infty\to0$ as $N\uparrow\infty$. In addition, it is given by \cite[Lemma 1(g)]{bayati2011dynamics} that $\ltwo{\bmag^{(s)\perp}}\ge N^{1/2}/\cst_{\textsc{tap},s}$. It follows that, given 
$(\bh^{(r)})_{1\le r\le s}$
for $s\ge0$, it holds with probability $1-o_N(1)$ that
	\[\f{\textsf{frac}[\bh(1:s+1)]}
	{\textsf{frac}[\bh(1:s)]}
	\ge \f1{\cst_{\textsc{tap},s,\gamma}}\,.\]
Iterating the bound gives \eqref{eq-CLT-empirical-h}. A similar argument gives, again with probability $1-o_N(1)$, that
	\beq\label{eq-H-recursion}
	\textsf{frac}[\bH(1:t-1)]
	\ge \f1{\cst_{\textsc{tap},s,\gamma}}\,,\eeq
but does not directly give \eqref{eq-CLT-empirical-H} due to the differing construction of $\bH^{(t+1)}$
(\eqref{e:fixed.pt} and \eqref{eq-def-H-circ}). To address this, let
	\[\bar{\Filt}_t
	\equiv\sigma\Big(\bn^{(1)},\bmag^{(2)},\bn^{(2)},
	\ldots,\bmag^{(t)}\equiv\bmag,\bn^{(t)}, \hat \bE
	\Big).\]
We write $\barcst$ for $\bar{\Filt}_t$-measurable random variables (not depending on $\gamma$) that remain stochastically bounded as $N\uparrow\infty$. We indicate dependence on $\gamma$ by writing $\barcst_\gamma$. Now consider \eqref{e:fixed.pt}, or equivalently
	\[\bn
	= F_{\bq} \circ (L_{\bq})^{-1}
	\bigg(\f{\bE\bmag + \hat{\bE}\bhJ}{N^{1/2}}
	-\dkappa\bigg)\,,\quad
	\f{\bE\bmag + \hat{\bE}\bhJ}{N^{1/2}}
	-\dkappa
	\ge \Big(\kappa+\DELTA^{1/2}\Big)\oneM\,.\]
Thanks to the random perturbation $\dkappa$, together with the fact that $|(F_{\bq} \circ (L_{\bq})^{-1})'(x)| \ge 1/\barcst$ for $x \ge \kappa + \DELTA^{1/2}$, we conclude that the random vector $\bn$ has a probability density uniformly bounded above by some $\barcst^M$. On the other hand, we can choose a value $\barcst$ such that $\ltwo{\bn}\le N^{1/2}\barcst'$ except with probability $o_N(1)$. Then, for small positive $\gamma'$, consider the set
	\[S
	\equiv
	\bigg\{\tilde{\bn}\in \mathbb R^M:
	\f{\ltwo{\tilde{\bn}-\mathrm{span}
	(\bn^{(1)}, \ldots, \bn^{(t-1)})}}{N^{1/2}}
	\leq \gamma' \textup{ and }
	\f{\ltwo{\tilde{\bn}}}{N^{1/2}}
	\le \barcst'
	\bigg\}\,.\]
We cover $S$ with $(O(1) N^{1/2}
{\barcst'} / \gamma')^{t-1}$ balls of diameter $\gamma'$; each ball has volume $\le (\gamma')^M$. Combining with the above density bound gives
	\[\oP\Big(\bn \in S \,\Big|\,\bar{\Filt}_t\Big)
	\le
	\bigg(
	\f{O(1) N^{1/2} \barcst' }
		{ \gamma'}\bigg)^{t-1}
	\Big(\gamma' \cdot \barcst\Big)^M\,,\]
which can be made $o_N(1)$ by taking $\gamma'= 1/(2\,\barcst)$. By relabelling $2\,\barcst$ as $\barcst$ we conclude that
	\beq\label{e:far.from.span}
	\f{\ltwo{\bn-\mathrm{span}(
	\bn^{(1)}, \ldots, \bn^{(t-1)})}}
	{N^{1/2}} \ge \f1{\barcst}
	\,,\eeq
with probability $1-o_N(1)$. Therefore we can decompose $\bn=\bn^\parallel+\bn^\perp$ where $\bn^\parallel$ is the orthogonal projection of $\bn$ onto $\spn (\bn^{(r)})_{r\le t-1}$, while $\ltwo{\bn^\perp}\ge N^{1/2}/\barcst$. To complete the proof of the lemma, in light of \eqref{eq-H-recursion} it suffices to argue that
with probability $1-o_N(1)$ we have
	\beq\label{eq-H-circ-to-prove}
	\f{\textsf{frac}[\bH(1:t)]}
		{\textsf{frac}[\bH(1:t-1)]}
	\ge \f1{\barcst_\gamma}\eeq
where $\bH(1:t)$ differs from $\bH(1:t-1)$ by the addition of the column $\bH^{(t+1)}$.
Recall from \eqref{eq-def-H-circ} that
	\[\bH^{(t+1)}
	\equiv \bigg(\f{\bE^\st\bn^\parallel}{N^{1/2}}
		- d^{(t)} \bmag\bigg)
		+\f{\bE^\st\bn^\perp}{N^{1/2}}\]
--- on the right-hand side, the term in parentheses is $\bar{\Filt}_t$-measurable. By similar considerations as in the proof of Lemma~\ref{l:extra.orth}, 
conditional on $\bar{\Filt}_t$, the vector $(\bE^\st\bn^\perp)/\ltwo{\bn^\perp}$ is distributed as a standard gaussian in $\R^N$ subject to the linear constraints
	\beq\label{eq-admissible-constraint}
	\f{(\bE^\st\bn^\perp, \br^{(s)})}{\ltwo{\bn^\perp}}
	=\f{(\bn^\perp, \bE\br^{(s)})}{\ltwo{\bn^\perp}}
	=\f{(\bn^\perp,\bx^{(s)})}{\ltwo{\bn^\perp}}\eeq
for all $1\le s\le t$, with $\bx^{(s)}$ as in
\eqref{e:all.row.constraints} and \eqref{e:explicit.x}. 
Explicitly, with $\stackrel{\textup{d}}{=}$ denoting equality in distribution, we have
	\[\by\equiv\f{\bE^\st\bn^\perp}{N^{1/2}}
	\stackrel{\textup{d}}{=}
	\sum_{s=1}^t
		\f{(\bn^\perp,\bx^{(s)})}
		{N^{1/2}}
		\br^{(s)}
	+ \f{\ltwo{\bn^\perp}}{N^{1/2}}
	\bigg(\Id_{N\times N}
	- \sum_{s=1}^t
	\br^{(s)}(\br^{(s)})^\st\bigg)
	\by^\textup{ind}\]
where $\by^\textup{ind}$ is an independent standard gaussian in $\R^N$. We have $(\bn^\perp,\bx^{(s)})\le \ltwo{\bn^\perp}\ltwo{\bx^{(s)}}\le N\barcst$, and we can choose a large enough threshold $\barcst''$ such that
	\[\f{\#[N]_\textup{bdd}}{N}\ge \f12\,,\quad
	[N]_\textup{bdd}
	\equiv
	\bigg\{1\le i\le N:
	|\br^{(s)}_i| \le \f{\barcst''}{N^{1/2}}
	\textup{ for all }1\le s\le t
	\bigg\}\,.\]
We can bound the conditional mean of 
$\by=(\bE^\st\bn^\perp)/N^{1/2}$
at coordinates in $[N]_\textup{bdd}$ as
	\[\Big|\E\big(
	y_i\,\big|\,\bar{\Filt}_t\big)\Big|
	=\bigg|\sum_{s=1}^t
		\f{(\bn^\perp,\bx^{(s)})}
		{N^{1/2}}
		\br^{(s)}_i\bigg|
	\le \barcst\,,\]
where the bound holds uniformly over $i\in [N]_\textup{bdd}$ (that is to say, there is a single choice of $\barcst$ such that the above bound is valid for all $i\in\#[N]_\textup{bdd}$).
Next, it holds with probability $1-o_N(1)$ that $(\br^{(s)},\by^\textup{ind}) \le \log N$ for all $s$, in which case it holds uniformly over $i\in[N]_\textup{bdd}$ that
	\[\sum_{s=1}^t
	\br^{(s)}_i
	(\br^{(s)},\by^\textup{ind})
	\le \f{(\log N)\barcst}{N^{1/2}}\,.\]
Let $\by_\textup{bdd}$ be the projection of $\by$ onto coordinates in $[N]_\textup{bdd}$: the above shows that
	\[\by_\textup{bdd}
	=\E\big(\by_\textup{bdd}
	\,\big|\,\bar{\Filt}_t\big)
	+\f{O[(\log N)\barcst]}{N^{1/2}}
	+
	\f{\ltwo{\bn^\perp}}{N^{1/2}}
	(\by^\textup{ind})_\textup{bdd}\]
where $(\by^\textup{ind})_\textup{bdd}$
is a standard gaussian in $\#[N]_\textup{bdd}$ dimensions, and $\|\E\big(\by_\textup{bdd}
\,\big|\,\bar{\Filt}_t\big)\|_\infty\le\barcst$. The claim \eqref{eq-H-circ-to-prove} follows by recalling that $\ltwo{\bn^\perp}/N^{1/2} \ge 1/\barcst$.
\end{proof}
\end{lem}

\noindent We now formally prove that in the QR factorization described in Section~\ref{sec-restricted-partition-function}, the change-of-basis matrix $\bXi$ is bounded and nondegenerate:

\begin{cor}\label{c:Xi.nondegenerate}
There exists $\cst=\cst_{M,\Nall,\DELTA,t}$ such that the matrix $\bN$ defined by \eqref{e:bN.before.change.basis} has all singular values bounded between $1/\cst$ and $\cst$. Consequently, in the QR factorization $\bN = \bC\bXi$, the change-of-basis matrix $\bXi$ has all singular values bounded between $1/\cst$ and $\cst$.

\begin{proof}
Write $\bN(r:s)$ for the submatrix of $\bN$ formed by columns $r$ through $s$ (inclusive). It follows from the results of \cite[Lemma 1(a) and 1(g)]{bayati2011dynamics} that $\bN(2:t)$ has all singular values bounded between $1/\cst$ and $\cst$.
Now consider a unit vector $\sss=(\sss_1,\ldots,\sss_t)\in\R^t$, and denote
$\sss(r;s)\equiv(\sss_r,\ldots,\sss_s)$. If $2|\sss_1|$ is smaller than $1/\cst$ which in turn is smaller than the minimal singular value of $\bN(2:t)$, then
	\[\ltwo{\bN\sss}
	=\ltwo{\bn\,\sss_1
		+ \bN(2;t)
		\,\sss(2:t)}
	\ge \ltwo{\bN(2;t)\,\sss(2:t)}
	-|\sss_1|
	\ge \f1{\cst}\bigg(1-\f1{(2\,\cst)^2}\bigg)
		- \f1{2\,\cst}
	\ge \f1{4\,\cst}\]
(having assumed without loss that $\cst$ is large enough for the last bound to hold). If instead $2|\sss_1| \ge\cst$, then
	\[\ltwo{\bN\sss}= \ltwo{(\bn^\parallel
		+\bn^\perp)\,\sss_1+ \bN(2;t)
		\,\sss(2:t)}
	\ge \ltwo{\bn^\perp\,\sss_1}
	\ge \f1{\cst}\,,\]
where the last bound follows from \eqref{e:far.from.span} together with the assumed lower bound on $|\sss_1|$, adjusting $\cst$ as needed. It follows that $\ltwo{\bN\sss} \ge 1/\cst$ uniformly over all unit vectors $\sss\in\R^t$. We likewise have $\ltwo{\bN\sss} \le\cst$ uniformly over all unit vectors $\sss\in\R^t$; the proof of this is more straightforward and is omitted here.\end{proof}\end{cor}

\noindent We next prove another technical result on the decay of the characteristic function for a gaussian conditioned to be larger than a threshold.

\begin{lem}\label{l:chf.decay}
Recall \eqref{e:def.F} that $\varphi_\xi$ denotes the density of a standard gaussian random variable conditioned to be at least $\xi$. Denote the characteristic function
	\[(\varphi_\xi)^\wedge(\tau)
	= 
	\int_{\nu\ge\xi} 
	\f{e^{\iii\tau\nu} \varphi(\nu)}
		{\bPhi(\xi)}\,d\nu
	= 
	\int_{\nu\ge\xi}
	\f{\varphi(\nu-\iii\tau)\,d\nu}
		{\bPhi(\xi)e^{\tau^2/2}}
	\equiv \f{\bPhi(\xi-\iii\tau)}
		{\bPhi(\xi)e^{\tau^2/2}}\,.\]
\begin{enumerate}[a.]
\item \label{l:chf.decay.small} It holds for $|\tau|$ small enough that
	$|(\varphi_\xi)^\wedge(\tau)|\le 
	\exp(-\usf{var}_\xi \tau^2 / 4)$ where
	$\usf{var}_\xi\equiv
	1+\cE(\xi)(\xi-\cE(\xi))$.
\item \label{l:chf.decay.large} It holds for all $\tau\in\R$ that
$|(\varphi_\xi)^\wedge(\tau)|\le 
\max\set{1, \exp(-\tau^2/2)+ 2\cE(\xi)/|\tau|}$.
\end{enumerate}

\begin{proof}
Shifting the mean of a random variable does not change the modulus of its characteristic function, so we have
$|(\varphi_\xi)^\wedge(\tau)|
	=|\E \exp(\iii\tau (\nu-\E\nu))|$.
For small $\tau$, it follows by Taylor expansion that
	\[|(\varphi_\xi)^\wedge(\tau)|
	= 1 - \f{\tau^2}{2}\Var\nu+O(|\tau|^3)\,.\]
A short calculation gives $\Var\nu = \usf{var}_\xi$, and the claim of part~\ref{l:chf.decay.small} follows. For part~\ref{l:chf.decay.large} let us assume $\tau$ is positive; the result for negative $\tau$ follows by an easy modification. By Cauchy's integral theorem applied to the boundary of the domain $\set{z\in\mathbb{C}: \Real z \ge \xi,
	0\le \Imag z \le \tau}$, we have
	\[\bPhi(\xi)
	- \bPhi(\xi-\iii\tau)
	= \iii\int_0^\tau \varphi(\xi-\iii u)\,du
	= \iii\varphi(\xi)
	\int_0^\tau e^{\iii\xi u}
	e^{u^2/2} \,du\,.\]
Rearranging and making the change of variables 
$u = \tau-\bar{u}/\tau$ gives
	\[\bigg|\f{\bPhi(\xi)
	- \bPhi(\xi-\iii\tau)}{\varphi(\xi)
		e^{\tau^2/2}}\bigg|
	\le \int_0^{\tau}
	\f{e^{u^2/2}}{e^{\tau^2/2}}\,du
	= \int_0^{\tau^2}
	\f{\exp\{ ((\bar{u}/\tau)^2-\bar{u})/2\}}
	{e^{\bar{u}/2}}
	\,\f{d\bar{u}}{\tau}
	\le \f1{\tau}
	\int_0^\infty \f{d\bar{u}}{e^{\bar{u}/2}}
	=\f2{\tau}\,.\]
Substituting into the expression for $(\varphi_\xi)^\wedge(\tau)$ gives part~\ref{l:chf.decay.large}.
\end{proof}
\end{lem}

\noindent Further towards proving Proposition~\ref{p:adm.clt}, we first consider
\eqref{e:adm.givensr.first} with 
the simpler event $\sat_J$ in place of $\bsatone\equiv\sat_{J,\EPS}$,
	\[\pp(\bar{E}_{\admone}) 
	\equiv\tp_{\admone|\row\sat_J}
		(\bar{E}_{\admone}\giv\bar{E}_\row)
	\equiv
	\int \f{\Ind{ \bE_{\row\admone\qrfone} \in\sat_J}
	\tp_{\admone\qrfone}(\bE_{\admone\qrfone})
		\,d\bE_{\qrfone}}
	{\tP(\sat_J\giv\bE_\row)}
	\,\bigg|_{\bE_{\row\admone}
	=\bar{E}_{\row\admone}}\,.\]
Recall (see \eqref{e:bN.before.change.basis})
that $\bC$ is the $M\times t$ matrix with columns $\bc^{(1)},\ldots,\bc^{(t)}$. We then have $\bE_{\admone}= \bC^\st\bE\bv$; and it follows from Proposition~\ref{p:basic.sat.giv.row} that the law under $\tP(\cdot\giv\bar{E}_\row,\sat_J)$ of $\bE_{\admone}$ coincides with the law under $\oP$ of $\aaa\equiv \bC^\st (\bnu-\bze^J)$ where $\bnu$ is distributed under $\oP$ as a standard gaussian in $\R^M$ conditioned to be $\ge\bxi^J$ (coordinatewise), with $\bze^J$ as in \eqref{e:def.zeta.J} and $\bxi^J$ as in \eqref{e:def.xi.J}. Thus $\pp$ is simply the density function of $\aaa$. We next bound the gradient of $\pp$:

\begin{lem}\label{l:gradient.bd} The function $\pp(\xx)\equiv \tp_{\admone|\row\sat_J}(\xx\giv\bar{E}_\row)$ satisfies $|\nabla\pp(\xx)|\le\cst$ uniformly over $\xx\in\R^t$. 

\begin{proof} Let $\sss=(\sss_1,\ldots,\sss_t)\in\R^t$. The characteristic function for density $\pp$ is
	\beq\label{e:chf.prod}
	\pp^\wedge(\sss)
	= \int_{\R^t} e^{\iii(\sss,\xx)}\pp(\xx)\,d\xx
	=\f{(\bm{\varphi}_{\bxi^J})^\wedge(\bC\sss)}
		{\exp(\iii(\bze^J)^\st \bC \sss)}\,,\quad
	(\bm{\varphi}_{\bxi^J})^\wedge
	(\btau)
	\equiv
	\prod_{\mu=1}^M
	(\varphi_{\xi^J_\mu})^\wedge
		(\tau_\mu)\eeq
where $\btau=\bC\sss\in\R^M$. Recall that we obtained $\bC$ via the QR factorization $\bN=\bC\bXi$ where $\bN$ is the $M\times t$ matrix given by \eqref{e:bN.before.change.basis}, and $\bXi$ is the $t\times t$ matrix containing the change-of-basis coefficients. It follows from Lemma~\ref{lem-CLT-empirical} that for any cube $A_\gamma\subseteq(\R_{\ge0})^t$ of side length $\gamma$ and within distance $1/\gamma$ of the origin, as $N\uparrow\infty$ we have
	\[\f1M
	\#\bigg\{1\le\mu\le M:
	(\hbe_\mu)^\st \bN
	\in \f{A_\gamma}{N^{1/2}}
	\bigg\}
	\ge \f1{\cst_\gamma}\,.\]
Recall from Corollary~\ref{c:Xi.nondegenerate} that $\bXi$ is nondegenerate, so for $\sss'\equiv\bXi^{-1}\sss$ we have $\bC\sss =\bN\sss'$, while the $\ell^2$ norms of $\sss,\sss'$ are within a $\cst$ factor of one another. It follows that we can choose $\cst$ large enough so that
	\begin{align}
	\label{e:first.estimate}
	\f{M}{\cst}
	&\le \#\bigg\{1\le\mu\le M : 
		|\xi^J_\mu|\le\cst\textup{ and }
		\f{|(\hbe_\mu)^\st\bC\sss|}{|\sss|}
		\le \f{1/\cst}{N^{1/2}} \bigg\}\,,\\
	\f{M}{\cst}
	&\le \#\bigg\{1\le\mu\le M :
		|\xi^J_\mu|\le \cst
		\textup{ and }
	\f{|(\hbe_\mu)^\st\bC\sss|}{|\sss|}
		\ge \f{\cst }{N^{1/2}} \bigg\}\,.
	\label{e:second.estimate}
	\end{align}
If $|\sss|\le N^{1/2}$ then 
\eqref{e:first.estimate}
implies that $|(\hbe_\mu)^\st\bC\sss|$ will be small for an asympotitically positive fraction of indices $\mu$, and combining with Lemma~\ref{l:chf.decay}\ref{l:chf.decay.small} gives
	\[|\pp^\wedge(\sss)|
	\le\f1{\exp(|\sss|^2/\cst)}\,.\]
If $|\sss|\ge N^{1/2}$ then 
\eqref{e:second.estimate} implies that $|(\hbe_\mu)^\st\bC\sss|$ will be large for an asymptotically positive fraction of indices $\mu$, and combining with Lemma~\ref{l:chf.decay}\ref{l:chf.decay.large} gives
	\[|\pp^\wedge(\sss)|\le 
	\f{N^{(t+2)/2}}
	{e^{N/\cst} |\sss|^{t+2}}\,.\]
It follows by Fourier inversion that
	\[2\pi |\nabla\pp(\xx)|
	\le \int_{\R^t}
		|\sss||\pp^\wedge(\sss)|
		\,d\sss
	\le
	\int_{|\sss|\le N^{1/2}}
	\f{|\sss|\,d\sss}
	{\exp(|\sss|^2/\cst)}
	+\int_{|\sss|\ge N^{1/2}}
	\f{N^{(t+2)/2}
	 |\sss|\,d\sss}
	{e^{N/\cst}|\sss|^{t+2}}\le \cst\,,\]
proving the claim for $\pp$. 
\end{proof}
\end{lem}

\begin{cor}\label{c:adm.density.largerevt}
The value $\pp(\bar{E}_{\admone}) \equiv\tp_{\admone|\row\sat_J}(\bar{E}_{\admone}\giv\bar{E}_\row)$ satisfies the bounds $\cst^{-1}\le\pp(\bar{E}_{\admone}) \le \cst$.

\begin{proof}
Recall that $\pp$ is the density of the
 random variable $\aaa=\bC^\st(\bnu-\bze^J)\in\R^t$ where $\bnu$ has the law of a standard gaussian in $\R^M$ conditioned to be $\ge\bxi^J$. It follows that $\E\bnu=\cE(\bxi^J)$ and consequently $\E\aaa=\bC^\st(\cE(\bxi^J)-\bze^J)$. For $J\in\bbH$, we see from a combination of \eqref{e:def.F}, \eqref{e:fixed.pt.n}, \eqref{e:restrict.j}, and \eqref{e:def.zeta.J} that 
	\[\f{\bze^J}{1+O(1/N)}
	=\sqrt{1-\bq}\bn
	=\sqrt{1-\bq} F_{\bq}(\bh)
	=\cE(\bxi)\,,\]
which in turn is very close to $\cE(\bxi^J)$ by \eqref{e:def.xi.J}. It then follows from Lemma~\ref{lem-CLT-approximation} that $\|\E\aaa\|_\infty \le \cst/N^{1/2}$. We also have $\Cov(\aaa) = \bC^\st\bD \bC$ where $\bD$ is the $M\times M$ diagonal matrix with diagonal entries $\Var(\bnu)$.
 
We claim that the $t\times t$ matrix
$\bC^\st\bD \bC$ is bounded and nondegenerate, in the sense that it is possible to find $\cst=\cst_{M,\Nall,t,\DELTA}$ such that all its eigenvalues are bounded between $1/\cst$ and $\cst$ in the limit $N\uparrow\infty$. To this end, recall that $\bC=\bN\bXi^{-1}$ where the $M\times t$ matrix $\bN$ is defined by \eqref{e:bN.before.change.basis}, and the $t\times t$ matrix $\bXi$ is bounded and nondegenerate by Corollary~\ref{c:Xi.nondegenerate}. It therefore suffices to verify that $\bN^\st\bD\bN$ is bounded and nondegenerate. Now note that
for any positive $\gamma$ we have
	\[\bD
	\succcurlyeq \gamma \Big(\Id_{M\times M}
		-\bI_\gamma\Big)\,,\quad
	\bI_\gamma
	\equiv\diag\Big(
	(\Ind{\bD_{\mu,\mu}\le\gamma})_{\mu\le M}
	\Big)\,,\]
where we use $A \succcurlyeq B$ to indicate that $A-B$ is positive-semidefinite. It follows that
	\[\bN^\st\bD\bN
	\succcurlyeq
	\gamma\bigg(
	\bN^\st\bN
	-\bN^\st \bI_\gamma \bN
	\bigg)\,.\]
We have $\bN^\st\bN = \bXi\bC^\st\bC\bXi=\bXi^2$, therefore $\bN^\st\bN$ is bounded and nondegenerate. As for the matrix $\bN^\st \bI_\gamma \bN$, it follows from \cite[Lemma 1(c)]{bayati2011dynamics} that the \emph{entrywise} maximum satisfies
	\[\adjustlimits
	\limsup_{\gamma\downarrow0}
	\limsup_{\Nall\uparrow\infty}\bigg(
	\max\Big\{
	(\bN^\st \bI_\gamma \bN)_{r,s}
	: 1\le r,s\le t
	\Big\}
	\bigg)=0\,.\]
It follows that we can choose $\cst,\cst'$ sufficiently large that the minimum eigenvalue of $\bN^\st\bN$ is at least $2/\cst$, while the maximum eigenvalue of
$\bN^\st \bI_{1/\cst'} \bN$ is at most $1/\cst$, so that altogether
	\[\bN^\st\bD\bN
	\succcurlyeq
	\f1{\cst'}\bigg(
	\f2{\cst}-\f1{\cst}
	\bigg)
	\Id_{t\times t}\,.\]
A simpler argument (omitted) gives that all eigenvalues of $\bN^\st\bD\bN$ are bounded above by some $\cst$, so we have that $\bN^\st\bD\bN$
(hence also $\bC^\st\bD\bC$) is bounded and nondegenerate as desired.

Now recall from \eqref{e:def.Ebar.adm} the definition of $\bar{E}_{\admone}$, and note that for $J\in\bbH$ we have $\|\bar{E}_{\admone}\|_\infty \le C/N^{1/2}$. It follows by the central limit theorem that
	\[\lim_{\Nall\uparrow\infty}
	\f1{(2\Delta)^t}
	\int_{\uu\in[-\Delta,\Delta]^t}
	\pp(\bar{E}_{\admone}+\uu)\,d\uu
	=\f1{(2\Delta)^t}
	\int_{\uu\in[-\Delta,\Delta]^t}
	\f{\exp(-\f12 \uu^\st \bm{\Theta}^{-1} \uu)}
		{(2\pi)^{t/2} \det \bm{\Theta}}\,d\uu\,,\]
and this converges as $\Delta\downarrow0$ to the positive constant $((2\pi)^{t/2} \det \bm{\Theta})^{-1}$. On the other hand, Lemma~\ref{l:gradient.bd} implies
	\beq\label{e:result.of.grad.bd}
	\bigg|\pp(\bar{E}_{\admone})
	- \f1{(2\Delta)^t}
	\int_{\uu\in[-\Delta,\Delta]^t}
	\pp(\bar{E}_{\admone}+\uu)\,d\uu
	\bigg|\le \cst\,\Delta\,,\eeq
which can be made small by taking $\Delta\downarrow0$. Combining the estimates gives the claim.
\end{proof}
\end{cor}

\begin{proof}[Proof of Proposition~\ref{p:adm.clt}]
We first prove \eqref{e:adm.givensr.first}. Abbreviate $\pp_\star(\bar{E}_{\admone})=\tp_{\admone|\row\bsatone}(\bar{E}_{\admone}\giv\bar{E}_\row)$,
and recall that $\bsatone\equiv\sat_{J,\EPS}$
is a subset of $\sat_J$. We already saw in Proposition~\ref{p:basic.sat.giv.row} that
	$\tP(\sat_J\giv\bar{E}_\row)
	=(1-o_N(1))
	\tP(\bsatone\giv\bar{E}_\row)$.
It follows that
	\[\f{\pp_\star(\bar{E}_{\admone})}{1+o_N(1)}
	=
	\int \f{\Ind{ \bE_{\row\admone\qrfone}\in\bsatone}
	\tp_{\admone\qrfone}(\bE_{\admone\qrfone})
		\,d\bE_{\qrfone}}
	{\tP(\sat_J\giv\bE_\row)}
	\,\bigg|_{\bE_{\row\admone}=\bar{E}_{\row\admone}}
	\le \pp(\bar{E}_{\admone})
	\le \cst\,,\]
where the last step uses the upper bound of Corollary~\ref{c:adm.density.largerevt}.
For the lower bound, let us decompose $\bnu\equiv(\bnu^{\ta},\bnu^{\tb})$ where $\bnu^{\ta} \equiv (\nu_\mu)_{\mu\le t}$ and $\bnu^{\tb} \equiv (\nu_\mu)_{t+1\le \mu\le M}$. We can write the joint density of $\bnu$ under $\tP(\cdot\giv\bar{E}_\row,\sat_J)$ as
	\[f(\bnu)\equiv f^{\ta}(\bnu^{\ta})
	\cdot f^{\tb}(\bnu^{\tb})
	\equiv
	\bigg\{\prod_{\mu=1}^t
	\varphi_{\xi_\mu}(\nu_\mu)\bigg\}
	\cdot\bigg\{ \prod_{\mu=t+1}^M
	\varphi_{\xi_\mu}(\nu_\mu)\bigg\}\,.\]
Let us also decompose 
$\bze^J\equiv(\bze^{J,\ta},\bze^{J,\tb})$
and
$\bxi^J\equiv(\bxi^{J,\ta},\bxi^{J,\tb})$
 where $\bze^{J,\ta},\bxi^{\ta}\in\R^t$. Recall that $\bC$ refers to the $M\times t$ matrix with columns $\bc^{(1)},\ldots,\bc^{(t)}$. Let $\bC^{\ta}$ be the topmost $t\times t$ submatrix of $\bC$ and let $\bC^{\tb}$ be the complementary $(M-t)\times t$ submatrix. By re-indexing if needed, we assume $(\cst)^{-1}\le N^{t/2}(\det \bC^{\ta}) \le \cst$. If $\bar{E}_{\admone}=\xx\in\R^t$ is given, then $\bnu^{\ta}$ is uniquely determined as a function of $\xx$ and $\bnu^{\tb}$,
	\[\bnu^{\ta}
	=\bnu^{\ta}(\xx,\bnu^{\tb})
	= \bze^{J,\ta}
	+ ((\bC^{\ta})^\st)^{-1}
	\bigg(
	\xx - (\bC^{\tb})^\st
	(\bnu^{\tb}-\bze^{\tb})
	\bigg)\,.\]
The joint density $f(\bnu)$ is positive if and only if $(\bnu^{\ta},\bnu^{\tb})
\ge(\bxi^{J,\ta},\bxi^{J,\tb})$, and
	\beq\label{e:pp.as.conv}
	\pp(\xx)
	= \int_{\R^{M-t}}
	f^{\ta}(\bnu^{\ta}(\xx,\bnu^{\tb}))
	f^{\tb}(\bnu^{\tb})
		\,d\bnu^{\tb}\,.\eeq
We can find an event $\psatone$ such that
\begin{enumerate}[(i)]
\item\label{i:exp.small}
$\tP(\bnu^{\tb}\notin\psatone\giv\bar{E}_\row,\sat_J)$ is exponentially small with respect to $N$, and 
\item\label{i:v.good.prf}
for all $\bnu^{\tb}\in\psatone$ and all $|\xx|\le C$, either $\bnu^{\ta}(\xx,\bnu^{\tb})\ge\bxi^{J,\ta}$ and
	\[(\bE_{\row},\bE_{\prfone})
	= \bigg(\bar{E}_\row,
	\bT\bv+
	(\bnu^{\ta}(\xx,\bnu^{\tb}),
		\bnu^{\tb})\bigg)
	\in\bsatone\,,\]
or the condition $\bnu^{\ta}(\xx,\bnu^{\tb})\ge\bxi^{J,\ta}$ is violated
(in which case $f^{\ta}(\bnu^{\ta}(\xx,\bnu^{\tb}))=0$).
\end{enumerate}
It follows from \eqref{i:v.good.prf} that for all $|\xx|\le C$ we have
	\[\pp_\star(\xx)
	\ge \pp_\textup{lbd}(\xx)
	\equiv
	\int_{\psatone} 
	f^{\ta}(\bnu^{\ta}(\xx,\bnu^{\tb}))
	f^{\tb}(\bnu^{\tb})
		\,d\bnu^{\tb}
	=
	\pp(\xx)
	-\int_{\R^{M-t}\setminus \psatone} 
	f^{\ta}(\bnu^{\ta}(\xx,\bnu^{\tb}))
	f^{\tb}(\bnu^{\tb})
		\,d\bnu^{\tb}
	\,.\]
The last term is exponentially small in $N$ by \eqref{i:exp.small},
and \eqref{e:adm.givensr.first} follows
by making use of the lower bound in Corollary~\ref{c:adm.density.largerevt}. To prove \eqref{e:adm.givensr.clt}, let us abbreviate the density of $\bar{E}_{\admtwo}$ conditional on $(\bE_\row,\bE_{\prfone})=(\bar{E}_\row,\bT\bv+\bnu)$ as
	\beq\label{e:adm.cond.nu}
	\pp_{\bnu}(\bar{E}_{\admtwo})
	\equiv
	\int
	\f{\Ind{\bE_{\row\admtwo\qrftwo}\in
	\sat_K}
	\tp_{\admtwo\qrftwo}(\bE_{\admtwo\qrftwo}
	\giv\bnu)\,d\bE_{\qrftwo}}
	{\tP(\sat_K \giv \bE_\row,\bnu )}
	\,\bigg|_{\bE_{\row\admtwo}
		=\bar{E}_{\row\admtwo}}\,.\eeq
We shall prove
$\pp_{\bnu}(\bar{E}_{\admtwo})\le\cst$ uniformly over $\lmin\le\lm\le\lmbig$ (for any constant $\lmbig<1$) and all $\bnu$ in the support of the measure 
$\tP(\cdot\giv\bar{E}_\row,\bar{E}_{\admone},\bsatone)$. Since we have already seen that $\pp_\star(\bar{E}_{\admone})\le\cst$, the claim
\eqref{e:adm.givensr.clt} follows by integrating over the law of $\bnu$ under $\tP(\cdot\giv\bar{E}_\row,\bar{E}_{\admone},\bsatone)$.
To bound $\pp_{\bnu}(\bar{E}_{\admtwo})$, note that
$\bE_{\admtwo}=\bC^\st\bE\tbv$ where $\tbv=\lm\bv+\sqrt{1-\lm^2}\bw$ by \eqref{e:c.lambda}.
Therefore, the law under $\tP(\cdot\,|\,\bar{E}_\row,\bE_{\prfone})$ of $\bE_{\admtwo}$ coincides with law under $\oP$ of the random variable
	\beq\label{e:A.nu}
	\aaa_{\lm,\bnu}
	\equiv
	\bC^\st
	\bigg(
	\lm\bnu
	+\sqrt{1-\lm^2}\bom
	-\bze^K
	\bigg)\,.\eeq
where $\bom$ has the law of a standard gaussian in $\R^M$ conditioned to be (coordinatewise) at least
	\beq\label{e:xi.K.lambda}
	\bxi^{K,\lm,\bnu}\equiv
	\f{\bxi^K-\lm\bnu}
	{\sqrt{1-\lm^2}}\,.\eeq
Then $\pp_\bnu$ is simply the density of $\aaa_{\lm,\bnu}$, and it follows similarly to \eqref{e:chf.prod} that for $\sss\in\R^t$,
	\[(\pp_{\bnu})^\wedge(\sss)
	=
	\f{(\bm{\varphi}_{\bxi^{K,\lm,\bnu}})^\wedge
	( \sqrt{1-\lm^2} \bC\sss)}
	{\exp(i (\bze^K-\lm\bnu)^\st \bC \sss)}
	\,.\]
Arguing as in Lemma~\ref{l:gradient.bd} 
gives $|\nabla \pp_{\bnu}(\xx)|\le \cst$
uniformly over all $\xx\in\R^t$. The argument of Corollary~\ref{c:adm.density.largerevt} then gives $|\pp_{\bnu}(\xx)|\le \cst$, thereby
concluding the proof of \eqref{e:adm.givensr.clt}.
\end{proof}

\subsection{Large deviations regime} 

Although the bound \eqref{e:adm.givensr.clt} of Proposition~\ref{p:adm.clt}\ref{p:adm.clt.pair} holds for all $\lmin\le\lm\le\lmbig$ for any constant $\lmbig<1$, we will only make use of it for $|\lm|\le\lmsmall$ for a small constant $\lmsmall>0$. For $|\lm|\ge\lmsmall$ the bound \eqref{e:adm.givensr.clt} turns out to be insufficient, and indeed one would expect in this regime that $\tp_{\adm|\row\bsat} (\bar{E}_\adm\giv\bar{E}_\row)$ decays exponentially in $N$. In the remainder of this section we prove a bound on the rate of decay that will suffice for our purposes. Recall (see \eqref{e:BB}) the definitions of $\BB(\lm,s)$ and $\AAA(\lm)$. Note that for $\bnu$ in the support of $\tP(\cdot\giv\bar{E}_\row,\bar{E}_{\admone},\bsatone)$ we must have 
$\bC^\st(\bnu-\bze^J)=\bar{E}_{\admone}$, so \eqref{e:A.nu} can be rewritten as
$\aaa_{\lm,\bnu}
	=\lm \bar{E}_{\admone}
	+ \sqrt{1-\lm^2}
	\bbb_{\lm,\bnu}$ where
	\beq\label{e:B.nu}
	\bbb_{\lm,\bnu}
	\equiv
	\bC^\st
	\bigg(\bom
		-\f{
		\bze^K-\lm\bze^J}{\sqrt{1-\lm^2}}
		\bigg)\,,\eeq
with $\bom$ as in \eqref{e:A.nu}. Let $\bbb_{\lm,\bnu,t}$ denote the $t$-th entry of $\bbb_{\lm,\bnu}$. We then have the following calculation which offers an interpretation for the function $\BB(\lm,s)$: 

\begin{lem}\label{l:adm.mgf}
For any fixed $s\in\R$, the cumulant-generating function $\mathcal{B}_{\lm,\bnu}(s)\equiv\log \E(\exp(s\bbb_{\lm,\bnu,t}))$ satisfies
	\[\adjustlimits
	\limsup_{\DELTA\downarrow0}
	\limsup_{t\uparrow\infty}
	\limsup_{\EPS\downarrow0}
	\limsup_{\Nall\uparrow\infty} \bigg\{
	\sup_\bnu\bigg|
	\f{\mathcal{B}_{\lm,\bnu}(N^{1/2} s)}
		{N}
	- \BB(\lm,s)
	\bigg|\bigg\}=0\]
where the innermost supremum is taken over all 
$\bnu$ in the support of $\tP(\cdot\giv\bar{E}_\row,\bar{E}_{\admone},\bsatone)$.

\begin{proof}
Recalling the notation of \eqref{e:chf.prod}, we have
	\begin{align*}
	\mathcal{B}_{\lm,\bnu}(s)
	&= -
	\bigg(\bc^{(t)},\f{
		\bze^K-\lm\bze^J}{\sqrt{1-\lm^2}}
		\bigg) s
	+\log 
	(\bm{\varphi}_{\bxi^{K,\lm,\bnu}})^\wedge
	(- i \bc^{(t)} s ) \\
	&=-
	\bigg(\bc^{(t)},\f{
		\bze^K-\lm\bze^J}{\sqrt{1-\lm^2}}
		\bigg) s
	+\f{s^2}{2}
	+\bigg(\oneM,
	\log \f{\bPhi(\bxi^{K,\lm,\bnu}
		-N^{1/2} s\bc^{(t)})}
		{\bPhi(\bxi^{K,\lm,\bnu})}\bigg)\,.\end{align*}
Recall that $\bc^{(t)}$ is fixed by \eqref{e:basis.fixed}; and it follows from \eqref{e:def.zeta.J} that $(\bc^{(t)},\bze^J) =(1-o_N(1)) N^{1/2} \bpsi(1-\bq)$. The result follows by making use of the restriction to $\sat_{J,\EPS}$ together with 
the result \eqref{e:innerproduct.state.limit} from Lemma~\ref{lem-CLT-approximation}.
\end{proof}
\end{lem}

\noindent Recall that $\AAA(\lm)\equiv \inf_s \BB(\lm,s)$. We then have the following bound:

\begin{ppn}
For any fixed $\lm$ we have the asymptotic bound
	\begin{equation*}\label{e:adm.givensr.ld}
	\tag{{\footnotesize$\color{Gr}\usf{A}|\usf{RB}^\usf{LD}$}}
	\adjustlimits
	\limsup_{\DELTA\downarrow0}
	\limsup_{t\uparrow\infty}
	\limsup_{\EPS\downarrow0}
	\limsup_{\Nall\uparrow\infty}
	\f{\log\tp_{\adm|\row\bsat}
	(\bar{E}_\adm\giv\bar{E}_\row)}{\Nall}
	\le \AAA(\lm)\,.\end{equation*}
The convergence is uniform over $\lmin\le\lm\le\lmbig$ for any constant $\lmbig<1$.
\end{ppn}

\begin{proof}
Let $\pp_\nu$ be as in \eqref{e:adm.cond.nu}.
It suffices to prove
	\[\adjustlimits
	\limsup_{\DELTA\downarrow0}
	\limsup_{t\uparrow\infty}
	\limsup_{\EPS\downarrow0}
	\limsup_{\Nall\uparrow\infty}
	\f{\log
	\pp_\nu(\bar{E}_{\admtwo})}{\Nall}
	\le \AAA(\lm)\]
uniformly over all $\bnu$ in the support of $\tP(\cdot\giv\bar{E}_\row,\bar{E}_{\admone},\bsatone)$. Similarly as in \eqref{e:pp.as.conv}, we can express
	\beq\label{e:pp.nu.as.conv}
	\pp_\nu(\xx)
	= \int_{\R^{M-t}} 
	f^{\ta,\lm}(
		\bom^{\ta}(\xx,\bom^{\tb})
		)
	f^{\tb,\lm}(\bom^{\tb})
	\,d\bom^{\tb}\eeq
where $f^{\ta,\lm}(\bom^{\ta})f^{\tb,\lm}(\bom^{\tb})$ gives the joint density of $\bom$
under $\tP(\cdot\giv\bar{E}_\row,\bnu,\sat_K)$, and
	\[\bom^{\ta}
		(\xx,\bom^{\tb})
	\equiv ( (\bC^{\ta})^\st)^{-1}\bigg(
	 \f{\xx-\lm\bar{E}_{\admone}
	+ \bC^\st(\bze^K-\lm\bze^J)
	}{\sqrt{1-\lm^2}}
	+ (\bC^{\tb})^\st \bom^{\tb}
	\bigg)\,.\]
If $\|\bom^{\ta}\|_\infty\ge (N\log N)^{1/2}$
then $f^{\ta,\lm}(\bom^{\ta})
\le \cst \exp(-\Omega(N\log N))$. It follows that the total contribution to \eqref{e:pp.nu.as.conv}
from those $\bom^{\tb}$ for which $\|\bom^{\ta}(\xx,\bom^{\tb})\|_\infty\ge (N\log N)^{1/2}$ satisfies the bound
	\beq\label{e:off.integral}
	\int_{\R^{M-t}}
	\mathbf{1}\bigg\{
	\f{\|\bom^{\ta}(\xx,\bom^{\tb})\|_\infty}
		{(N\log N)^{1/2}}
		\ge 1\bigg\}
	f^{\ta,\lm}(
		\bom^{\ta}(\xx,\bom^{\tb})
		)
	f^{\tb,\lm}(\bom^{\tb})
	\,d\bom^{\tb}
	\le \f{\cst}{\exp(\Omega(N\log N))}\,.\eeq
Recall from the proof of Proposition~\ref{p:adm.clt} that, by re-indexing if needed, we may assume
$(\cst)^{-1}\le N^{t/2}(\det \bC^{\ta}) \le \cst$.
If $\|\bom^{\ta}(\xx,\bom^{\tb})\|_\infty\le (N\log N)^{1/2}$ and we take a \emph{nonnegative} vector $\uu\in\R^t$ with $\|\uu\|_\infty\le \exp(-\Omega((\log N)^2))$, then
	\[\f{f^{\ta,\lm}
	(\bom^{\ta}(\xx,\bom^{\tb}))}{f^{\ta,\lm}
	(\bom^{\ta}(\xx+ (\bC^{\ta})^\st \uu
	,\bom^{\tb}))}
	\le 1+o_N(1)
	\le 2\,.\]
We require nonnegative $\uu$ since otherwise it is possible that the numerator is positive while the denominator is zero. It follows for such $\uu$ that 
	\[\pp_\bnu(\xx)
	\le 2\pp_\bnu(\xx+ (\bC^{\ta})^\st \uu)
	+ \f{\cst}{\exp(\Omega(N\log N))}\,.\]
Combining with the preceding estimate \eqref{e:off.integral} gives
	\[\pp_\bnu(\xx)
	\le 
	\f{\cst}{\Delta^t}
	\int_{\uu\in[0,\Delta]^t}
	\pp_\bnu(\xx+ (\bC^{\ta})^\st\uu)\,d\uu
	+ \f{\cst}{\exp(\Omega(N\log N))}\]
for $\Delta = \exp(-\Omega((\log N)^2))$. 
Now, for $\xx=\bar{E}_{\admtwo}$,
recalling \eqref{e:B.nu} gives
	\[\int_{\uu
		\in[0,\Delta]^t}
	\pp_\bnu(\bar{E}_{\admtwo}
	+ (\bC^{\ta})^\st \uu)\,d\uu
	\le \oP\bigg(
	\bigg|\bbb_{\lm,\bnu}
	-
	\f{\bar{E}_{\admtwo}
	-\lm\bar{E}_{\admone}}
		{\sqrt{1-\lm^2}}
	\bigg| \le \f1{\exp(\Omega((\log N)^2))}
	\bigg)
	\le \oP\bigg(|\bbb_{\lm,\bnu,t}|
		\le \f{O(1)}{N^{1/2}}\bigg)\,,\]
where the last step uses that $\|\bar{E}_{\admone}\|_\infty$ and $\|\bar{E}_{\admtwo}\|_\infty$ are $O(N^{-1/2})$, by \eqref{e:def.Ebar.adm} and the assumption $(J,K)\in\bbH(\lm)$. It follows by Markov's inequality that
	\[\oP\bigg(|\bbb_{\lm,\bnu,t}|
		\le \f{O(1)}{N^{1/2}}\bigg)
	\le N^{O(1)}
	\inf_{|s| \le \log N}
	\E\Big(\exp( N^{1/2} s \bbb_{\lm,\bnu,t} )
	\Big)\,.\]
Combining with Lemma~\ref{l:adm.mgf} gives the result.
\end{proof}

\section{Volume for slices of discrete cube}\label{s:cube}

\noindent In this section we estimate the sizes of the sets $\bbH$, $\bbH_\EPS$, and $\bbH_\EPS(\lm)$ specified in Definition~\ref{d:cube}. We abbreviate $\eigmin$ and $\eigmax$ for the minimum and maximum eigenvalues of a positive-semidefinite matrix.

\subsection{CLT and moderate deviations regime} We first estimate $\#\bbH$ and $\#\bbH_\EPS$:

\begin{ppn}\label{p:count}
For any fixed $t$ and fixed positive $\EPS$ we have,
for $\bH\equiv\bH^{(t)}$
and $\bmag\equiv\bmag^{(t)}\equiv\tnh(\bH)$,
	\begin{equation*}\label{e:count}
	\tag{{\footnotesize$\color{Or}\usf{H}$}}
	\f{\#\bbH_\EPS}
		{1-o_N(1)}
	=
	\#\bbH 
	= \f{\exp\{O(\cst)\}}{N^t}
	\exp\Big\{
	(\oneN,\log(2\ch(\bH)))
	-(\bH,\bmag)
	\Big\}\end{equation*}
As a consequence, the limit \eqref{e:HH} holds.

\begin{proof}
Let $\oQ$ be the uniform probability measure over $J\in\set{-1,+1}^N$. Define the tilted measure $\tQ$ by
	\beq\label{e:tilt.H}
	\f{d\tQ}{d\oQ}
	=\prod_{i=1}^N \f{\exp(H_i J_i)}
		{\ch(H_i)}
	= \exp\bigg\{(\bH,J)
		-(\oneN,\log\ch(\bH))
		\bigg\}
	\,.\eeq
If $\tQE$ denotes expectation under $\tQ$, then $\tQE J=\tnh(\bH)=\bmag$. Define the $N\times 2t$ matrices
	{\setlength{\jot}{0pt}\begin{align}
	\label{e:J.ips}
	\bL
	&\equiv
	\begin{pmatrix}
	\bmag^{(1)} & \cdots & \bmag^{(t)}
	& \bH^{(2)} & \cdots & \bH^{(t)}
	& \bH^{(t+1)}
	\end{pmatrix}\\ \nonumber
	\bL'
	&\equiv
	\begin{pmatrix}
	N^{1/2}\br^{(1)} & \cdots & 
	N^{1/2}\br^{(t)}
	& \bH^{(2)} & \cdots & \bH^{(t)}
	& \bH^{(t+1)}
	\end{pmatrix}
	\end{align}}%
From Definition~\ref{d:cube},
for $J\in\bbH$ we have
	\begin{align}
	\label{e:r.times.J.minus.m}
	\f{(\br^{(s)},J-\bmag)}{N^{1/2}}
	&= \Big( \bj
	- \sqrt{\bq}\br^{(t)},\br^{(s)}\Big)
	= \begin{cases}
	\sqrt{q_{J,t}} 
	- \sqrt{\bq} = O(1/N)
	&\textup{for }s=t\,, \\
	\gamma_{J,s}
	= O(1/N)
	&\textup{for }1\le s\le t-1\,,
	\end{cases}\\
	\label{e:H.times.J.minus.m}
	\f{(\bH^{(s)},J-\bmag)}{N}
	&= 
	\bigg(
	(\sqrt{q_{J,t}}-\sqrt{\bq})\br^{(t)}
	+\sum_s\gamma_{J,s} \br^{(s)}
	+ \sqrt{1-q_J}\bv, \f{\bH^{(s)}}{N^{1/2}}
	\bigg)
	= \f{O(\cst)}{N}\,.
	\end{align}
It follows
from \eqref{e:r.times.J.minus.m}
and \eqref{e:H.times.J.minus.m}
 that we can choose $\cst=\cst_{M,\Nall,t,\DELTA}$ such that
	\beq\label{e:sandwich}
	\bigg\{
	\linf{(\bL')^\st(J-\bmag)}\le \f1{\cst}
	\bigg\}\subseteq
	\set{J\in\bbH}
	\subseteq \bigg\{
	\linf{(\bL')^\st(J-\bmag)}\le\cst
	\bigg\}\,.\eeq
It then follows from the factorization $\bM=\bR\bGa$ (see \eqref{e:bM.before.change.basis}) and the nondegeneracy of $\bGa$ that \eqref{e:sandwich} holds with $\bL$ in place of $\bL'$, adjusting $\cst$ as needed (maintaining stochastic boundedness). We hereafter consider $\yy_J\equiv \bL^\st (J-\bmag)\in\R^{2t}$. Note this can be decomposed as 
	\[\yy_J
	= \sum_{i=1}^N\yy_{J,i}
	\equiv 
	\sum_{i=1}^N
	\bL^\st \dbe_i (\dbe_i)^\st (J-\bmag)
	= 
	\sum_{i=1}^N
	\bL^\st \dbe_i(J_i-m_i)\,.\]
Under the measure $\tQ$, the $\yy_{J,i}\in\R^{2t}$ are mutually independent random variables with mean $\tQE \yy_{J,i}=0$ and covariance $\bSi_i\equiv
\tQE(\yy_{J,i}(\yy_{J,i})^\st)
= (1-(m_i)^2) (\bL^\st\dbe_i)(\bL^\st\dbe_i)^\st
$ for each $i$. By re-indexing let us suppose that $\ltwo{\bL^\st\dbe_i}$ is nondecreasing in $i$, and let
	\[\quad
	\yy_{J,\ta}
	\equiv
	\sum_{i=1}^{\lfloor N/2\rfloor}
	\yy_{J,i}
	= (\bL_\ta)^\st(J-\bmag)
	\,,\quad
	\yy_{J,\tb}
	\equiv \sum_{i=\lfloor N/2\rfloor+1}^N
	\yy_{J,i}
	= (\bL_\tb)^\st(J-\bmag)
	\,,\quad
	\bL\equiv
	\begin{pmatrix}\bL_\ta \\
		\bL_\tb
		\end{pmatrix}\,.\]
The covariance matrices
$\bSi_\ta\equiv \Cov(\yy_{J,\ta})$ and
$\bSi_\tb\equiv \Cov(\yy_{J,\tb})$
are obtained by summing $\bSi_i$ over the corresponding indices $i$. Now observe that
	\[\f1N \sum_{i=1}^N
	\ltwo{\bL^\st\dbe_i}^2
	=\f1N\sum_{s=1}^t
	\bigg(\ltwo{\bH^{(s)}}^2 
	+\ltwo{\bmag^{(s)}}^2 \bigg)
	\le 
	\f{2}N\sum_{s=1}^t
	\ltwo{\bH^{(s)}}^2 
	\le \cst\]
where the last inequality follows by
\eqref{eq-H-norm}. With the re-indexing it follows that
	\beq\label{eq-local-CLT-UI}
	\max_{1\le i\le N/2}
	\eigmax
	(\bSi_i)
	=\max_{1\le i\le N/2}
	\ltwo{\bL^\st\dbe_i}^2
	\le\cst\,,\quad
	\eigmax
	(\bSi_\tb)
	\le
	\sum_{i=\lfloor N/2\rfloor+1}^N
	\ltwo{\bL^\st\dbe_i}^2
	\le N\cst\,.\eeq
The bound on $\eigmax(\bSi_\tb)$ implies (via Chebychev's inequality) that $\tQ(\linf{\yy_{J,\tb}}/N^{1/2} \le \cst') \ge 1/2$ for some choice of $\cst'$. We will apply a triangular array local central limit theorem \cite[Thm.~3.1]{MR3632529} to the sum $\yy_{J,\ta}$ to deduce the following: uniformly over all $\linf{\yy_{J,\tb}}/N^{1/2} \le \cst'$ and all $1/\cst \le c\le \cst$,
	\begin{equation}\label{eq-local-CLT}
	\f1{\cst''} \le N^t
	\tQ\bigg(
	\linf{\yy_{J,\ta}+\yy_{J,\tb}}
	\le c
	\givBig
	\yy_{J,\tb}
	\bigg) \le \cst''\,.
	\end{equation}
Before verifying \eqref{eq-local-CLT} let us see that it implies the claimed result.
Recalling \eqref{e:sandwich} gives
	\[\tQ(\bbH)
	= \f{\exp\{O(\cst'')\}}{N^t}\,.\]
It follows from the law of large numbers that
$\tQ(\bbH_\EPS\giv\bbH)=1-o_N(1)$ for any fixed positive $\EPS$. On the other hand, rearranging
\eqref{e:H.times.J.minus.m} gives $(J,\bH) = (\bmag,\bH) + O(\cst)$, which implies that the Radon--Nikodym derivative \eqref{e:tilt.H} is roughly constant
(up to additive error $O(\cst)$) over $J\in\bbH$, hence also over 
$J\in\bbH_\EPS$.
Therefore
	\[\exp\bigg\{
	(\bH,\bmag)
	-(\oneN,\log\ch(\bH))
	+O(\cst)
	\bigg\}
	\oQ(\bbH)
	= \tQ(\bbH)
	= \f{\exp\{O(\cst'')\}}{N^t}\,,\]
and likewise with $\bbH_\EPS$ in place of $\bbH$. Rearranging these estimates gives \eqref{e:count}, since $\oQ(\bbH)
= \#\bbH/2^N$. The convergence \eqref{e:HH} follows, making use of \cite[Lemma 1(c)]{bayati2011dynamics}.

It remains to prove the claim~\eqref{eq-local-CLT}. For this it suffices to verify the conditions of \cite[Thm.~3.1]{MR3632529} for the random variables $(\yy_{J,i})_{i\le N/2}$. We summarize the criteria as follows:
\begin{enumerate}[I.]
\item \label{e:total.cov}
(cf.\ \cite[eq.~(3.4)~and~(3.6)]{MR3632529})
The total covariance satisfies
$\eigmin(\bSi_\ta)\gg \log \det \bSi_\ta
\gg 1$ in the limit $N\uparrow\infty$.

\item \label{e:bor.max.cov} (cf.\ \cite[eq.~(3.4)]{MR3632529})
There is a uniform bound
$\eigmax(\bSi_i)\le \cst$
over all $1\le i\le N/2$.

\item \label{e:bor.ui} (cf.\ \cite[\textbf{UI}]{MR3632529}) We have
 $\ltwo{\yy_{J,i}} \le \cst$ almost surely for all $1\le i\le N/2$. 

\item \label{e:bor.nl} (cf.\ \cite[\textbf{NL}]{MR3632529}) For any fixed $\gamma\in(0,1)$ we have
	\beq\label{e:nl}
	\sup\bigg\{
	\Big|\E \exp(\iii(\yy_{J,\ta},\sss))\Big|
	: \gamma \le |\sss| \le \f1\gamma
	\bigg\}
	\le \f{o_N(1)}{\det \bSi_\ta}\,.\eeq
\end{enumerate}
In fact, \cite[\textbf{UI}]{MR3632529} is a weaker ``uniform integrability'' condition. Our collection of random variables $\yy_{J,i}$ satisfies the stronger almost sure bound stated in \eqref{e:bor.ui}, since with probability one we have $\ltwo{\yy_{J,i}} \le \ltwo{\bL^\st\dbe_i}$, and $\ltwo{\bL^\st\dbe_i}\le\cst$ by \eqref{eq-local-CLT-UI}. Condition~\eqref{e:bor.max.cov} is directly implied by \eqref{e:bor.ui} (although not by the weaker condition \cite[\textbf{UI}]{MR3632529}).

We next verify condition~\eqref{e:total.cov}. In the following we write $\cst_{\DELTA',\EPS'}\equiv \cst_{M,\Nall,\DELTA,t,\EPS,\DELTA',\EPS'}$ following the convention set by Remark~\ref{r:cst}. Recall from above that
	\[\bSi_\ta
	= \sum_{1\le i\le N/2}
	(1-(m_i)^2) 
	(\bL^\st \dbe_i)
	(\bL^\st \dbe_i)^\st
	= (\bL_\ta)^\st
	\bB
	\bL_\ta\,,\quad
	\bB
	\equiv \diag( (1-(m_i)^2)_{i\le N/2} )\,.\]
Write $\mathfrak{e}_r$ for the $r$-th standard basis vector in $\R^t$. Let $A_{\DELTA',\EPS',r}\subseteq\R^t$ denote a ball of radius $\EPS'$ centered around $\DELTA'\mathfrak{e}_r$. It follows from Lemma~\ref{lem-CLT-empirical} that for any fixed $\DELTA',\EPS'$ we have
	{\setlength{\jot}{0pt}\begin{align*}
	\#\set{ 1\le i\le N :
	(\dbe_i)^\st \bH(1:t) \in 
	A_{\DELTA',\EPS',r}
	}
	&\ge N/\cst_{\DELTA',\EPS'}\,,\\
	\#\set{ 1\le i\le N :
	(\dbe_i)^\st \bH(1:t) \in 
	A_{2\DELTA',\EPS',r}
	} &\ge N/\cst_{\DELTA',\EPS'}\end{align*}}%
for all $1\le r\le t$. This means we can permute the indices $1\le i\le N/2$ such that
	\[\bL_\ta
	= \begin{pmatrix}
	\bL_1 \\ \vdots \\ \bL_k
	\\ \bL_\textup{rest}
	\end{pmatrix}\,,\quad
	\bL_j
	= \begin{pmatrix}
	\tnh(\DELTA') \Id_{t\times t}
	& \DELTA' \Id_{t\times t}\\
	\tnh(2\DELTA') \Id_{t\times t}
	& 2\DELTA' \Id_{t\times t}\\
	\end{pmatrix}
	+O(\EPS')
	\in\R^{2t\times 2t}\,,\]
where $k \ge N/\cst_{\DELTA',\EPS'}$
and $\bL_\textup{rest}$ is $(N-kt)\times 2t$. Since $\tnh$ is nonlinear, for any $\DELTA'>0$ we can choose $\EPS'$ small enough to guarantee that $\bL_j$ is of full rank. From this we can readily see that $\eigmin((\bL_\ta)^\st\bL_\ta)\ge N/\cst$. Arguing similarly as in the proof of Corollary~\ref{c:adm.density.largerevt} gives that $\bSi_\ta=(\bL_\ta)^\st\bB\bL_\ta$ has $\eigmin(\bSi_\ta)\ge N/\cst$ also. We have from \eqref{eq-local-CLT-UI} that $\eigmax(\bSi_\ta)\le N\,\cst$, so $\det \bSi_\ta \le (N\,\cst)^t$. Condition~\eqref{e:total.cov} immediately follows.

It remains to verify the ``nonlattice'' condition \eqref{e:bor.nl}. Writing $p_i\equiv(1+m_i)/2$, we have
	\[\phi(\sss)\equiv
	\Big|\E \exp(\iii(\yy_{J,\ta},\sss))\Big|
	=\prod_{i\le N/2}
	\bigg|
	p_i \exp\{ 2 \iii
	(\bL^\st\dbe_i,\sss)\}
	+(1-p_i)
	\bigg|\,.\]
If $(\bL^\st\dbe_i,\sss)$ is bounded away from $2\pi\mathbb{Z}$ for a positive fraction of indices $i$, then $|\phi(\sss)|$ will be exponentially small with respect to $N$, hence much smaller than the right-hand side of \eqref{e:nl}. By choosing $\DELTA',\EPS'$ small enough (depending on $\gamma$) we can ensure that for all $\gamma\le|\sss|\le1/\gamma$ there is a positive fraction of indices $i$ such that
	\[\f{\gamma^2}{\cst}\le(\bL^\st\dbe_i,\sss)\le \f12\,.\]
Condition~\eqref{e:bor.nl} follows, concluding the proof.
\end{proof}
\end{ppn}

\begin{ppn}\label{p:cube.moderate}
For any fixed $t$ and fixed positive $\EPS$, it holds for small $\lm$ that 
	\begin{equation*}\label{e:pairs.clt}
	\tag{{\footnotesize$\color{Or}\usf{H}^\textsf{2,CLT}$}}
	\f{\#\bbH_\EPS(\lm)}
		{(\#\bbH_\EPS)^2}
	\lesssim
	\f{\#\bbH(\lm)}
		{(\#\bbH)^2}
	\asymp
	\exp\bigg\{-\f{N\lm^2}{2\sigma^2}
	+O(N\lm^3)\bigg\}\end{equation*}
where $\sigma^2\equiv (\sigma_{M,N_\textup{all},\DELTA,t,\EPS})^2$ is of constant order, and is explicitly given by
	\beq\label{e:explicit.sigma.finite}
	\sigma^2\equiv
	\f{\|1-\bmag^2\|^2}{N(1-\bq)^2}\,.\eeq

\begin{proof}
We now let $\oQ$ be the uniform probability measure over pairs $(J,K)\in\set{-1,+1}^N\times\set{-1,+1}^N$, and
	\beq\label{e:tilt.H.pairversion}
	\f{d\tQ}{d\oQ}\equiv\exp\bigg\{
	(\bH,J+K)-2(\oneN,\log\ch(\bH))
	\bigg\}\,.\eeq
Take the matrix $M\in\R^{N\times 2t}$ as in \eqref{e:J.ips}, and consider the random variable
	\[\yy
	\equiv (\yy_J,\yy_K,
	\ell_{JK})
	\equiv
	\bigg(\bL^\st(J-\bmag),
	\bL^\st(K-\bmag),
	\f{(J-\bmag,K-\bmag)}{1-\bq}
	\bigg)
	\in\R^{4t+1}\,.\]
For $\zz\in\R^{4t+1}$ let $\Lambda(\zz)=\log\tQE\exp\{(\zz,\yy)\}$ be the cumulant-generating function of $\yy$ under $\tQ$. Suppose that $\zz_\lm$ solves $\Lambda'(\zz_\lm)=\mathfrak{e}_\lm$ where $\mathfrak{e}_\lm\equiv(0,\ldots,0,N\lm)\in\R^{4t+1}$. Let
	\beq\label{e:tilt.H.lambda}
	\f{d\tQ_\lm}{d\tQ}
	\equiv \exp\bigg\{
	(\zz_\lm,\yy)
	-\Lambda(\zz_\lm)
	\bigg\}\,.\eeq
Then $\tQE_\lm\yy=\mathfrak{e}_\lm$,
and it follows from the local central limit theorem 
(similarly as in the proof of Proposition~\ref{p:count})
that
	\[\tQ_\lm(\bbH(\lm))
	=\tQ_\lm\bigg( \|\yy-\mathfrak{e}_\lm\|_\infty
		=\f{O(1)}{N}
		\bigg)
	=
	\f{\exp\{O(\cst)\}}{N^{2t+1/2}}\,.\]
On the other hand, both \eqref{e:tilt.H.pairversion} and \eqref{e:tilt.H.lambda} are roughly constant over $(J,K)\in\bbH(\lm)$, yielding
	\[\f{\exp\{O(\cst)\}}{N^{2t+1/2}}=
	\tQ_\lm(\bbH(\lm))
	= \oQ(\bbH(\lm))
	\exp\bigg\{
	(\zz_\lm,\yy)
	-\Lambda(\zz_\lm)
	+2(\bH,\bmag)
	-2(\oneN,\log\ch(\bH))
	\bigg\}\,.\]
Recalling $\oQ(\bbH(\lm))=\#\bbH(\lm)/4^N$ and rearranging gives
	\[\exp\bigg\{
		(\zz_\lm,\mathfrak{e}_\lm)-\Lambda(\zz_\lm)
	\bigg\}
	\# \bbH(\lm)
	\asymp 
	\f1{N^{2t+1/2}}
	\exp\bigg\{
	2(\oneN,\log(2\ch(\bH)))
	-2(\bH,\bmag)
	\bigg\}
	\asymp (\# \bbH)^2\,,\]
where the last relation is by Proposition~\ref{p:count}. It remains to estimate $\zz_\lm$ and $\Lambda(\zz_\lm)$ when $\lm$ is small. To this end note that $\Lambda(\mathbf{0})=0$, $\Lambda'(\mathbf{0})=\mathbf{0}$,
and $\Lambda''(\mathbf{0})$ is the covariance matrix of $\yy$ under $\tQ$. It is a block diagonal matrix
	\[\Lambda''(\mathbf{0})
	=\begin{pmatrix}
	\bA & &\\
	&\bA &\\
	& & N\sigma^2\\
	\end{pmatrix}\]
where $\sigma^2$ is as in the statement of the proposition, and $\bA = \bL^\st \diag(1-\bmag^2)^{-1} \bL \in\R^{2t\times2t}$
which has all singular values of order $N$. For small $\lm$ we will have
$\zz_\lm = \Lambda''(\mathbf{0})^{-1}\mathfrak{e}_\lm
+O(\lm^2)$, therefore
	\[\Lambda(\zz_\lm)-(\zz_\lm,\mathfrak{e}_\lm)
	=-\f{(\mathfrak{e}_\lm)^\st
		\Lambda''(\mathbf{0})^{-1}\mathfrak{e}_\lm}{2}
		+O(N\lm^3)
	= - \f{N\lm^2}{2\sigma^2}\,,\]
proving the claim.
\end{proof}
\end{ppn}

\subsection{Large deviations regime}

For general $\lm$ we do not have an easy way to bound the size of $\bbH(\lm)$. However, it is relatively straightforward to bound the size of the more restricted set $\bbH_\EPS(\lm)$. We begin with some notations. Throughout what follows we let $H\in\R$,
$m\equiv \tnh H$, and
	\[p \equiv \f{1+m}{2}\,.\]
We will use the notations $H,m,p$ with the understanding that there is a bijective correspondence among all three. With this in mind, let $P_{H,D}$ be the probability distribution on $\set{-1,+1}^2$ given by (cf.\ \eqref{e:P.H.D.intro})
	\beq\label{e:P.H.D}
	P_{H,D}=\begin{pmatrix}
	p^2+D/4 & p(1-p)-D/4\\
	p(1-p)-D/4 & (1-p)^2+D/4
	\end{pmatrix}\\
	=\f14
	\begin{pmatrix}
	(1+m)^2+D&1-m^2-D\\
	1-m^2-D&(1-m)^2+D
	\end{pmatrix}
	\,,\eeq
where for the distribution to be nonnegative we must have
	\beq\label{e:delta.bounds}
	\f{-(1-|m|)^2}{4}
	=
	-\min\set{p,1-p}^2 \le \f{D}{4}\le p(1-p)
	= \f{1-m^2}{4}\,.\eeq
We write $\ent$ for the Shannon entropy of a distribution, and let
	\[\Gamma(H,D)
	\equiv
	\ent(P_{H,D})
	\equiv
	\sum_{x\in\set{-1,+1}^2}
	P_{H,D}(x)\log\f1{P_{H,D}(x)}
	\in[0,\log 4]\,.\]
Consider a pair $(J_i,K_i)$ distributed according to $P_{H,D}$, and write $\mathbf{E}_{H,D}$ for expectation over this law: then
	{\setlength{\jot}{0pt}
	\begin{align}\nonumber
	\mathbf{E}_{H,D}(J_i)=\mathbf{E}_{H,D}(K_i)
	&=m\,,\\
	\mathbf{E}_{H,D}(J_i K_i)-m^2
	&=D\,.
	\label{e:QHD.expectations}
	\end{align}}%
For any $a,b\in\R$, we let $A\equiv\exp(2a)$ and $B\equiv\exp(2a)$, and define a measure $Q_{A,B}$ on $\set{-1,+1}^2$ by
	\beq\label{e:QAB}
	Q_{A,B}(J_i,K_i) = \f{4\exp\{
	a(J_i K_i-1)+b(J_i+K_i)
	\}}{B+1/B+2/A}\,.\eeq
Write $\bar{\mathbf{E}}_{A,B}$ for expectation under $Q_{A,B}$: then
	\begin{align}\nonumber
	\bar{\mathbf{E}}_{A,B}(J_i)
	= \bar{\mathbf{E}}_{A,B}(K_i)
	&= \f{B-1/B}{B+1/B+2/A}
	\,,\\
	\bar{\mathbf{E}}_{A,B}(J_i K_i)
	&=\f{B+1/B-2/A}{B+1/B+2/A}\,.
	\label{e:QAB.marginal}
	\end{align}
Recalling $m\equiv \tnh H$, we write
$\Delta\equiv\Delta_H(A)\equiv\sqrt{A^2+m^2-(Am)^2}$,
and define 
	\beq\label{e:B.A}
	B_H(A)
	\equiv
	\f{A(1+m)}{\Delta-m}
	= \f{\Delta+m}{A(1-m)}\,,\quad
	S_H(A)
	\equiv \f1{B_H(A)^{\sgn(H)}}\,.\eeq
Note that $B_H(A)$ is one of two conjugate solutions to the equation
	\beq\label{e:rel.B.with.m}
	\bar{\mathbf{E}}_{A,B}(J_i)
	= \f{B-1/B}{B+1/B+1/A}
	=\tnh H \equiv m\eeq
(further discussed below).
We then define
	\beq\label{e:D.A}
	D_H(A) \equiv
	\f{X+1/X-2/A}{X+1/X+2/A}-m^2\,\bigg|_{X=B_H(A)}
	=\f{X+1/X-2/A}{X+1/X+2/A}-m^2\,\bigg|_{X=S_H(A)}\,,\eeq
or equivalently $D_H(A) = \bar{\mathbf{E}}_{A,B}(J_i K_i)-m^2$ (cf.\ \eqref{e:QHD.expectations}~and~\eqref{e:QAB.marginal}). In the following lemma we record some basic properties of these functions. (In particular, we will see that \eqref{e:D.A} coincides with our earlier definition \eqref{e:D.H.intro}.)

\begin{lem}\label{l:algebra}
For $H\in\R$ and $A\in(0,\infty)$ the following hold:
\begin{enumerate}[a.]
\item \label{l:algebra.good}
If $H=0$ then $B_H(A)=S_H(A)=1$ as noted above. For general $H$ we have the symmetries
	\beq\label{e:root.sign.symm}
	\f{S_H(A)}{S_{-H}(A)}
	=\f{D_H(A)}{D_{-H}(A)}=1\,.\eeq
For $H\ne0$, the function $A\mapsto S_H(A)$ is strictly increasing over $A\in(0,\infty)$, sandwiched by its boundary values
	\beq\label{e:S.root.bounds}
	0=S_H(0) < S_H(\infty)
	= \bigg(\f{1-|m|}{1+|m|}\bigg)^{1/2}\,.\eeq
The function $A\mapsto D_H(A)/4$ is also strictly increasing over $A\in(0,\infty)$, sandwiched by its 
boundary values
	\beq\label{e:delta.bounds.two}
	-\min\set{p^2,(1-p)^2}
	=-\f{(1-|m|)^2}{4}
	= \f{D_H(0)}{4} < \f{D_H(\infty)}{4}
	= \f{1-m^2}{4} = p(1-p)\,,\eeq
(cf.\ \eqref{e:delta.bounds}), with $ D_H(1)=0$ and $(D_H)'(1)= (1-m^2)^2/2$.
\item \label{l:algebra.conjugate}
For any fixed $H\in\R$ and $A\in(0,\infty)$, the value $B_H(A)$ is one of two conjugate solutions to the equation $\bar{\mathbf{E}}_{A,B} J_i=m$
(see~\eqref{e:QAB.marginal}). The other solution is
$B_{H,\textup{conj}}(A) \equiv -(\Delta-m)/[A(1-m)]$, which is nonpositive. If we define $D_{H,\textup{conj}}(A)$ as in \eqref{e:D.A} but with $B_{H,\textup{conj}}(A)$ in place of $B_H(A)$, then
$D_{H,\textup{conj}}(A)$ violates the bounds \eqref{e:delta.bounds.two} except in the trivial case $A=H=0$ where $\Delta=0$. As a result, the pair $(B,D)\equiv(B_H(A),D_H(A))$ is the unique solution to the equations
	{\setlength{\jot}{0pt}\begin{align*}
	\mathbf{E}_{H,D}(J_i)
		&=\bar{\mathbf{E}}_{A,B}(J_i)\,,\\
	\mathbf{E}_{H,D}(J_iK_i)
		&=\bar{\mathbf{E}}_{A,B}(J_iK_i)\end{align*}}%
such that $B\ge0$ and $D$ satisfies \eqref{e:delta.bounds.two}. Equivalently, if $H\in\R$ and $A\in(0,\infty)$ are given, then $(B,D)$ is the unique pair such that $P_{H,D}=Q_{A,B}$ is a valid probability measure on $\set{-1,+1}^2$.
\end{enumerate}

\begin{proof}
Take nonzero $H\in\R$ and $m\equiv \tnh H$. As before, let $\Delta\equiv \sqrt{A^2+m^2-(Am)^2}$. Then
	\[B_H(A)
	=\f{m + \Delta}{A(1-m)}
	= \f{m^2-\Delta^2}{A(1-m)(m-\Delta)}
	= \f{A(1+m)}{m-\Delta}
	= \f1{B_{-H}(A)}\,,\]
from which \eqref{e:root.sign.symm} follows. Next we calculate the derivative
	\[(B_H)'(A)
	=-\f{m(m+\Delta)}{A^2(1-m)\Delta}\]
which has the same sign as $-H$. It follows that $S_H(A) \equiv 1/B_H(A)^{\sgn H}$ is strictly increasing on $A\in(0,\infty)$ and is sandwiched between its boundary values $S_H(0)$ and $S_H(\infty)$,
as given by \eqref{e:S.root.bounds}. Rearranging $\bar{\mathbf{E}}_{A,B} (J_i)=m$ gives a quadratic equation in $B$, with one root given by $B_H(A)$. The conjugate root is
	\[B_{H,\textup{conj}}(A)
	= \f{m-\Delta}{A(1-m)}
	= \f{m^2-\Delta^2}{A(1-m)(m+\Delta)}
	= -\f{A(1+m)}{m+\Delta}
	= -\f{1+m}{1-m} \f1{B_H(A)}\,,\]
from which it follows that 
$S_{H,\textup{conj}}(A)\equiv (-B_{H,\textup{conj}}(A))^{\sgn(H)}$
satisfies the simple relation
	\[S_{H,\textup{conj}}(A)
	=\bigg(\f{1+|m|}{1-|m|}\bigg) S_H(A)\,.\]
We can use the above relation between $B_H(A)$ and $B_{H,\textup{conj}}(A)$ to write
	\beq\label{e:D.A.expand}
	D_H(A)
	= \f{\f{m+\Delta}{A(1-m)}
		-\f{m-\Delta}{A(1+m)}
		-2/A}
	{\f{m+\Delta}{A(1-m)}
		-\f{m-\Delta}{A(1+m)}
		+2/A}-m^2
	= \f{(1-m^2)(\Delta-1)}
		{\Delta+1}
	=\f{(\Delta-1)^2}{A^2-1}
	= \f{(A^2-1)(1-m^2)^2}{(\Delta+1)^2}
	\,.\eeq
Then $\sgn D_H(A) = \sgn(A^2-1)$, and for continuity at $A=1$ we take $D_H(1)=0$. We also calculate that
	\beq\label{e:deriv.D.A}
	(D_H)'(A) =\f{2A(\Delta-1)^2}
		{(A^2-1)^2\Delta} 
	= \f{2A D_H(A)}{(A^2-1)\Delta}
	= \f{2 A(1-m^2)^2}{\Delta(\Delta+1)^2}\,,\eeq
where for continuity at $A=1$ we take $(D_H)'(1) = (1-m^2)^2/2$. Thus $D_H(A)$ is strictly increasing on $A\in(0,\infty)$, sandwiched between its boundary values \eqref{e:delta.bounds.two}. This concludes the proof of part~\ref{l:algebra.good}. For part~\ref{l:algebra.conjugate}, note that
	\[D_{H,\textup{conj}}(A)
	= \f{(\Delta+1)^2}{A^2-1}
	= \f{(\Delta^2-1)^2}{(A^2-1)(\Delta-1)^2}
	= \f{(A^2-1)(1-m^2)^2}{(\Delta-1)^2}
	= \f{(1-m^2)^2}{D_H(A)}\,.\]
From this relation it is clear that
$D_H(A)$ is consistent with the bounds \eqref{e:delta.bounds} while $D_{H,\textup{conj}}(A)$ is not, except in the case $A=H=0$ where $\Delta=0$. The conclusion follows.
\end{proof}
\end{lem}

\noindent We see from \eqref{e:D.A.expand}
that \eqref{e:D.H.intro} and \eqref{e:D.A} coincide. 
Recall from \eqref{e:def.ell.A} the definition of $\ell(A)$. It follows from Lemma~\ref{l:algebra} that $A\mapsto \ell(A)$ is strictly increasing on $A\in(0,\infty)$, sandwiched by its boundary values as given by \eqref{e:lmin}. The inverse $\lm\mapsto \ell^{-1}(\lm) \equiv A(\lm)$ is well-defined for $\lmin\le\lm\le1$, and we can let
	\beq\label{e:frH.A}
	\frH(A)
	\equiv- 2\HH_\star +
	\int \Gamma\Big(\psi^{1/2}z,
		D_{\psi^{1/2}z}(A)
		\Big)\varphi(z)\,dz
	\,\bigg|_{q=\sq,\psi=\spsi}\,\eeq
and $\HH(\lm)\equiv \frH(\ell^{-1}(\lm))$
as in \eqref{e:HH.pair}. We upper bound the size of $\#\bbH_\EPS(\lm)$ as follows:

\begin{ppn}\label{p:cube.ld}
As claimed in \eqref{e:pairs.ld.intro}, we have
for any fixed $\lm$ that 
	\begin{equation*}\label{e:pairs.ld}
	\tag{{\footnotesize$\color{Or}
	\usf{H}^\textsf{2,LD}$}}
	\adjustlimits
	\limsup_{\DELTA\downarrow0}
	\limsup_{t\uparrow\infty}
	\limsup_{\EPS\downarrow0}
	\limsup_{\Nall\uparrow\infty}
	\f{\log\#\bbH_\EPS(\lm)}
		{\Nall}
	\le 2\HH_\star+\HH(\lm)\,.\end{equation*}
The convergence holds uniformly over $\lmin\le\lm\le\lmbig$ for any constant $\lmbig<1$.

\begin{proof}
As in the proof of Proposition~\ref{p:cube.moderate},
let $\oQ$ be the uniform measure on
$\set{-1,+1}^N\times\set{-1,+1}^N$. For $a\in\R$ and $\mathbf{b}\in\R^N$,
let $A\equiv \exp(2a)$ and $\mathbf{B}\equiv\exp(2\mathbf{b})$, and let
(cf.\ \eqref{e:QAB})
	\beq\label{e:LDQ}
	\f{d\LDQ}{d\oQ}
	=\prod_{i=1}^N
		\f{4\exp(a (J_i K_i-1) + b_i (J_i+K_i))}
		{B_i+1/B_i+2/A}
	=\f{4^N\exp\{-Na+
	a(J,K)+(\mathbf{b},J+K)\}}
	{\exp\{(\oneN,\log(\mathbf{B}+
	1/\mathbf{B}+2/A ))\}}\,.\eeq
Write $\LDE$ for expectation over $\LDQ$. We then set $\mathbf{B}=B_{\bH}(A)$ (applying coordinatewise the function of \eqref{e:B.A}), resulting in
$\LDE J=\LDE K = \bmag$. Further, with $D_{\bH}(A)$ as defined by \eqref{e:D.A}, we have
	\[\f{\LDE (J-\bmag,K-\bmag)}{N(1-\bq)}
	=\f{\LDE (J,K)
		-\|\bmag\|^2}{N(1-\bq)}
	=\f{(\oneN,D_{\bH}(A))}{N(1-\bq)}\,.\]
We see from \cite[Lemma 1(c)]{bayati2011dynamics}
that in the limit ($\Nall\uparrow\infty,\EPS\downarrow0,t\uparrow\infty,\DELTA\downarrow 0$) the above tends to $\ell(A)$ as defined by \eqref{e:def.ell.A}. For us it suffices to simply set $A=\ell^{-1}(\lm)$. For any $J,K\in\bbH$, we have from \eqref{e:restrict.j} and \eqref{e:restrict.k} that
	\[\f{(J,K)}{N}
	=(\bj,\bk)
	= \sqrt{ q_{J,t}q_{K,t} }
	+\sum_{s=1}^{t-1} 
	\gamma_{J,s}\gamma_{K,s}
	+\sqrt{(1-q_J)(1-q_K)} \lm_{J,K}\,.\]
where we recall that $\lm_{J,K}\equiv(\bv,\tilde{\bv})$ by definition. By the definitions of $\bbH$
and $\bbH(\lm)$, the above simplifies to 
	\[\f{(J,K)}{N}
	= \bq + (1-\bq) \lm + \f{O(1)}{N}\]
uniformly over all pairs $(J,K)\in\bbH(\lm)$. Combining with \eqref{e:q.psi.recursion}~and~\eqref{e:def.ell.A} gives
	\begin{align}\nonumber
	&\adjustlimits 
	\limsup_{\DELTA\downarrow0}
	\limsup_{t\uparrow\infty}
	\limsup_{\EPS\downarrow 0} 
	\limsup_{\Nall\uparrow\infty}
	\bigg\{-a + a\f{(J,K)}{N}\bigg\}
	=-a(1-\sq)(1-\lm)\\
	&\qquad =-
	\int\bigg\{
	\bigg( 1-(\tnh(\psi^{1/2}z))^2
		-D_{\psi^{1/2}z}(A)
	\bigg) \log A \bigg\} \varphi(z)\,dz
	\,\bigg|_{q=\sq,\psi=\spsi}\,.
	\label{e:J.K.scalar.prod.limit.identity}
	\end{align}
Next, it follows from the definition of $\bbH_\EPS$ that for all $J\in\bbH_\EPS$ we have
(recalling $2\mathbf{b}=\log\mathbf{B}$)
	\beq\label{e:JK.tilt.limit.val}
	\adjustlimits 
	\limsup_{\DELTA\downarrow0}
	\limsup_{t\uparrow\infty}
	\limsup_{\EPS\downarrow 0} 
	\limsup_{\Nall\uparrow\infty}
	\f{2(\mathbf{b},J)}{N}
	= \int (\tnh(\psi^{1/2}z)
		\log B_{\psi^{1/2}z}(A)
	\,\varphi(z)\,dz
	\,\bigg|_{\psi=\spsi}\,.\eeq
From \eqref{e:J.K.scalar.prod.limit.identity} and \eqref{e:JK.tilt.limit.val} we see that the Radon--Nikodym derivative \eqref{e:LDQ} is roughly constant over pairs $(J,K)\in\bbH_\EPS(\lm)$. In particular,
with $\ETA$ an error tending to zero in the manner of \eqref{e:ETA}, we can lower bound
	\begin{align*}
	\f1N
	\log \bigg(\f1{4^N} \f{d\LDQ}{d\oQ}\bigg)
	+\ETA
	&\ge -\int \bigg\{
	\log\bigg(B_{\psi^{1/2}z}(A)
		+\f1{B_{\psi^{1/2}z}(A)}+\f2A\bigg)
	-(\tnh(\psi^{1/2}z)
		\log B_{\psi^{1/2}z}(A)\\
	&\qquad 
	+\bigg( 1-(\tnh(\psi^{1/2}z))^2
		-D_{\psi^{1/2}z}(A)\bigg)\log A
	\bigg\}
	\,\varphi(z)\,dz
	\,\bigg|_{q=\sq,\psi=\spsi}\\
	&= -\int \Gamma(\psi^{1/2}z,
	D_{\psi^{1/2}z}(A))
	\,\varphi(z)\,dz
	\,\bigg|_{q=\sq,\psi=\spsi}
	= - \frH(A)-2\HH_\star\,,\end{align*}
where the last equality above is by \eqref{e:frH.A}.
The second-to-last equality is obtained by integrating over $H=\psi^{1/2}z$ the following algebraic identity: for any $H\in\R$ and $A\in(0,\infty)$, we have
	\begin{align*}
	\Gamma(H,D_H(A))
	&= \ent(P_{H,D_H(A)})
	= \ent\bigg(
	\f1{B+1/B+2/A}
	 \begin{pmatrix}
	B & 1/A\\ 1/A & 1/B
	\end{pmatrix}\bigg)\,\bigg|_{B=B_H(A)}\\
	&= \log(B+1/B+2/A)
	- \f{B-1/B}{B+1/B+2/A} \log B
	- \f{2/A}{B+1/B+2/A}\log A
	\,\bigg|_{B=B_H(A)}\\
	&= \log(B+1/B+2/A)
		- (\tnh H) \log B
		- \f{1-(\tnh H)^2-D_H(A)}{2}\log A\,,\end{align*}
having used at the very last step \eqref{e:rel.B.with.m}~and~\eqref{e:D.A}. It follows that
	\beq\label{e:cube.ubd}
	1
	\ge\LDQ(\bbH_\EPS(\lm))
	= \mathbf{E}\bigg[
	\f{d\LDQ}{d\oQ}
	\mathbf{1}\{\bbH_\EPS(\lm)\}
	\bigg]
	\ge
	\f{\#\bbH_\EPS(\lm)}
	{\exp\{N(\frH(A)+2\HH_\star+\ETA)\}}\,.\eeq
Rearranging gives the claimed bound 
since we set $A=\ell^{-1}(\lm)$, giving
$\frH(A)=\HH(\lm)$.
\end{proof}
\end{ppn}

\noindent For $p\in[0,1]$ we shall hereafter abbreviate 
	\[\ent(p)
	\equiv p \log \f1{p}
		+(1-p)\log\f1{1-p}\]
for the entropy of the $\textup{Bernoulli}(p)$ distribution. Note the identity
	\beq\label{e:binary.entropy.ldp}
	\ent\bigg(\f{1+\tnh(H)}{2}\bigg)
	=\log(2\ch H) - H\tnh H\,.\eeq
By \eqref{e:binary.entropy.ldp} together with gaussian integration by parts, \eqref{e:HH} can be rewritten as
	\beq\label{e:HH.entropy}
	\HH_\star 
	= \int \ent
	\bigg(\f{1+\tnh(\psi^{1/2}z)}{2}\bigg)
	\,\varphi(z)\,dz\,.\eeq
Returning to the definition \eqref{e:HH.pair} of $\HH(\lm)$, and recalling Lemma~\ref{l:algebra}, we note that at $\lm=0$ we have $A=\ell^{-1}(\lm)=1$ and
$D_H(A)=0$. Meanwhile, at $\lm=1$ we have
$A=\ell^{-1}(\lm)=\infty$ and
$D_H(A)=4p(1-p)$. The corresponding distributions are
	\[P_{H,D_H(1)}
	= \begin{pmatrix}
	p^2 & p(1-p) \\
	p(1-p) & (1-p)^2
	\end{pmatrix}\,,\quad
	P_{H,D_H(\infty)}
	= \begin{pmatrix}
	p & 0 \\ 0 & 1-p
	\end{pmatrix}\,,\]
so we see that $\Gamma(H,D_H(1)) = 2\ent(p)$ while $\Gamma(H,D_H(\infty)) = \ent(p)$. Substituting into \eqref{e:HH.pair} gives $\HH(0)=\frH(1)=0$ while $\HH(1)=\frH(\infty)=-\HH_\star$.

We conclude this section by deriving formulas for the (first and second) derivatives of $\HH(\lm)$ which will be used in later sections. We begin with an alternative derivation of \eqref{e:binary.entropy.ldp}: let $\oQ$ be the uniform measure on $\set{-1,+1}^N$, and consider the change of measure (cf.\ \eqref{e:tilt.H})
	\beq\label{e:binary.identity.tilt}
	\f{d\tQ}{d\oQ}
	= \prod_{i=1}^N \f{\exp(H J_i)}{\ch H}\,.\eeq
For $-1\le x\le 1$ we let $\mathbb{H}[x]$ be the set of all $J\in\set{-1,+1}^N$ with empirical mean near $x$:
	\[\mathbb{H}[x]
	\equiv
	\bigg\{ J\in\set{-1,+1}^N:
	\bigg|\f{(J,\oneN)}N - x\bigg| \le \f1N
	\bigg\}\,.\]
Then $\set{-1,+1}^N$ is covered by the sets $\mathbb{H}[x]$ for $x=2k/N-1$ with $k\in\set{0,1,\ldots,N}$. For such $x$, 
	\beq\label{e:Hx.entropy}
	\oQ(\mathbb{H}[x])
	= \f{\#\mathbb{H}[x]}{2^N}
	= \f1{2^N} \binom{N}{k}
	= \exp\bigg\{ N\bigg[
	\ent\bigg(
		\f{1+x}{2}\bigg)
	-\log2
	+o_N(1)\bigg]
	\bigg\}\,.\eeq
Since the Radon--Nikodym derivative \eqref{e:binary.identity.tilt} is constant over $\mathbb{H}[x]$, we have
	\beq\label{e:Hx.RN}
	\tQ(\mathbb{H}[x])
	= \oQ(\mathbb{H}[x])
	\f{\exp(N H x)}{(\ch H)^N}\,.\eeq
On the other hand, it is clear from \eqref{e:binary.identity.tilt} that under the measure $\tQ$ we have $\tQE J_i=\tnh H$ for all $i$, and the empirical mean $(J,\oneN)/N$ will be exponentially well concentrated around $\tnh H$. This means that
$\tQ(\mathbb{H}[x])$ is approximately maximized at $x=\tnh H$, with
 $\tQ(\mathbb{H}(\tnh H)) = \exp\{o_N(1)\}$. Combining with \eqref{e:Hx.entropy} and \eqref{e:Hx.RN} gives
	\[\exp\{o_N(1)\}
	=\tQ(\mathbb{H}(\tnh H))
	=
	\exp\bigg\{ N\bigg[
	\ent\bigg(
		\f{1+\tnh H}{2}\bigg)
	-\log2
	+o_N(1)\bigg]
	\bigg\}
	\f{\exp(N H \tnh H)}{(\ch H)^N}\,.\]
Taking $N\uparrow\infty$ and rearranging gives \eqref{e:binary.entropy.ldp}. Of course, this derivation is overkill for \eqref{e:binary.entropy.ldp} which can be obtained by simple algebra. However we next apply a similar method to obtain identities for $\HH(\lm)$ which are not so straightforward to prove by direct algebraic manipulation.

We now let $\oQ$ stand for the uniform probability measure on pairs $(J,K)\in\set{-1,+1}^N\times\set{-1,+1}^N$, and consider the change of measure (cf.\ \eqref{e:QAB})
	\beq\label{e:pairs.identity.tilt}
	\f{d\tQ}{d\oQ}
	=\prod_{i=1}^N
	\f{4\exp\{
	a(J_i K_i-1)+b(J_i+K_i)
	\}}{B+1/B+2/A}\,.\eeq
Let $\mathbb{H}[H,D]$ be the set of pairs $(J,K)$ having empirical measure close to the measure $P_{H,D}$ of \eqref{e:P.H.D}:
	\[\mathbb{H}[H,D]
	= \bigg\{
	(J,K) \in (\set{-1,+1}^N)^2
	: 
	\bigg|\f{\#\set{i\le N : (J_i,K_i)=
	\sigma}}{N}
	-P_{H,D}(\sigma)\bigg|\le\f1N
	\textup{ for all } \sigma\in\set{-1,+1}^2
	\bigg\}\,.\]
For $(J,K)\in \mathbb{H}[H,D]$, it follows from \eqref{e:QHD.expectations} that
(with $m\equiv\tnh H$ as usual)
	\[\f{(J,\oneN)}{N} = m + \f{O(1)}{N}
	= \f{(K,\oneN)}{N}
	\,,\quad
	\f{(J,K)}N = D+m^2 + \f{O(1)}{N}\,.\]
Analogously to \eqref{e:Hx.entropy}, we have
	\beq\label{e:Hx.entropy.pairversion}
	\oQ(\mathbb{H}[H,D])
	= \f{\#\mathbb{H}[H,D]}{4^N}
	= \exp\bigg\{
	N\bigg[
	\Gamma(H,D) - \log 4 + o_N(1)
	\bigg]
	\bigg\}\,.\eeq
Analogously to \eqref{e:Hx.RN}, we have
	\beq\label{e:Hx.RN.pairversion}
	\tQ(\mathbb{H}[H,D])
	= \oQ(\mathbb{H}[H,D])
	\f{4^N 
	\exp\{
	N[ a(D+m^2-1)+b(2m) ]
	\}
	}{(B+1/B+2/A)^N}\eeq
On the other hand, we see from \eqref{e:pairs.identity.tilt} that in the measure $\tQ$ we have (cf.~\eqref{e:QAB.marginal})
	\[\tQE J_i = \tQE K_i = \f{B-1/B}{B+1/B+2/A}\,,\quad
	\tQE(J_i K_i) = \f{B+1/B-2/A}{B+1/B+2/A}\,.\]
The corresponding empirical means
$(J,\oneN)/N$, $(K,\oneN)/N$, and 
$(J,K)/N$ will be exponentially well concentrated about these values. It follows that $\tQ(\mathbb{H}[H,D])$ is maximized at the value $(H,D)$ such that
	\[H=H_{A,B} = 
	\atnh\bigg( \f{B-1/B}{B+1/B+2/A}\bigg)\,,\quad
	D=D_{A,B} = \f{B+1/B-2/A}{B+1/B+2/A} 
		- \tnh(H_{A,B})^2\]
--- equivalently, such that $B=B_H(A)$ and $D=D_H(A)$
(see \eqref{e:B.A}~and~\eqref{e:D.A}).
Combining \eqref{e:Hx.entropy.pairversion}~and~\eqref{e:Hx.RN.pairversion} gives
	\[\tQ(\mathbb{H}[H,D])
	=\f{ \exp\{
	N[\Gamma(H,D)
	+(D+m^2)(\log A)/2
	+ m\log B
	+o_N(1)]\}}{(B+1/B+2/A)^N
		A^{N/2}
	}\,,\]
where the denominator does not depend on $(H,D)$. By taking $N\uparrow\infty$ we see that	
	\[\LL(H,D)
	= \Gamma(H,D)+(D+(\tnh H)^2)\f{\log A}{2}
	+(\tnh H)(\log B)\]
is maximized at $(H,D)=(H_{A,B},D_{A,B})$. From the stationarity equations we obtain
	\beq\label{e:entropy.gprime}
	\f{\pd\Gamma(H,D_H(A))}{\pd D}
	=-\f{\log A}{2}\,,\quad
	\f{\pd\Gamma(H,D_H(A))}{\pd m}
	= - m\log A - \log B_H(A)\,.\eeq
(The identity \eqref{e:entropy.gprime} will be useful to us in later sections. It can also be obtained by a purely algebraic derivation, but we found the above calculation to be more conceptually simple.) Substituting into \eqref{e:def.ell.A} and combining with \eqref{e:HH.pair} gives with $H\equiv \psi^{1/2}z$,
	\[\frH'(A)
	= \int
	\f{\pd\Gamma(H,D_H(A))}{\pd D}
	(D_H)'(A)\,\varphi(z)\,dz
	= -\f{(1-q)\log A}{2} \ell'(A)\,.\]
Since $\HH(\lm)=\frH(A(\lm))$ with $A(\lm)\equiv \ell^{-1}(\lm)$, we conclude that
	\beq\label{e:dHH.dlambda}
	\HH'(\lm)
	=-\f{(1-q)\log A(\lm)}{2}\,,\quad
	\HH''(\lm)
	=-\f{(1-q)A'(\lm)}{2 A(\lm)}
	=-\f{(1-q)}{2 A(\lm) \ell'(A(\lm))}\,.\eeq
Since $A(\lm)$ and $\ell'(A(\lm))$ are both positive, we see that $\HH(\lm)$ is a concave function of $\lm$.
Note also that for $\sigma$ defined by Proposition~\ref{p:cube.moderate} we have
	\beq\label{e:sigma.limit.ellprime}
	\adjustlimits\limsup_{\DELTA\downarrow0}
	\limsup_{t\uparrow\infty}
	\limsup_{\EPS\downarrow0}
	\limsup_{\Nall\uparrow\infty}
	\bigg|\f{1}
	{(\sigma_{M,N_\textup{all},\DELTA,t,\EPS})^2}
	+ \HH''(0)\bigg|=0\eeq
by making use of \cite[Lemma 1(c)]{bayati2011dynamics}. Thus \eqref{e:pairs.ld} is consistent with our earlier bound \eqref{e:pairs.clt}.

\section{Conclusion}\label{s:conclusion}

\noindent Define the following numerical constants:
{\setlength{\jot}{0pt}\begin{alignat}{3}
	\nonumber
	\albd&\equiv\alval\,,\quad
	&\qlbd&\equiv\qlval\,,\quad
	&\psilbd&\equiv\psilval\,,\\ \nonumber
	& &\qlu&\equiv\qluval
	&\psilu&\equiv\psiluval\,,\\ \nonumber
	\aubd&\equiv\auval\,,\quad
	&\qul&\equiv\qulval\,,\quad
	&\psiul&\equiv\psiulval\,,\\
	& & \qubd&\equiv\quval\,,\quad
	&\psiubd&\equiv\psiuval
	\label{e:numbers}
	\end{alignat}}%
Note that $\qlbd<\qlu<\qul<\qubd$
and likewise for $\psi$. Let
	\beq\label{e:gamma.bds}
	\gamma(q)\equiv \bigg(\f{q}{1-q}\bigg)^{1/2}\,.\eeq
Denote $\gamlbd\equiv\gamma(\qlbd)$
and similarly $\gamlu,\gamul,\gamubd$.

\begin{lem}\label{l:intbyparts}
If $(q,\psi)$ solves the fixed-point relation
$\psi = R(q,\alpha)$
and $q=P(\psi)$, then
	\[I=\alpha \int
	\cE\bigg(\f{\kappa-q^{1/2}z}{\sqrt{1-q}}
	\bigg) \f{\kappa-q^{-1/2}z}{\sqrt{1-q}}
	\,\varphi(z)\,dz
	=\psi(1-q)\,.\]

\begin{proof}
Using gaussian integration by parts and the identity
$\cE'(x)=\cE(x)(\cE(x)-x)$, we calculate
	\begin{align*}
	I&= \alpha \int 
	\bigg\{
	\cE\bigg(\f{\kappa-q^{1/2}z}{\sqrt{1-q}}
	\bigg) \f{\kappa}{\sqrt{1-q}}
		+ \cE'
	\bigg(\f{\kappa-q^{1/2}z}{\sqrt{1-q}}\bigg)
	\f1{1-q}\bigg\}
	\,\varphi(z)\,dz
	\\
	&= \alpha \int
	\bigg\{ F_q(q^{1/2}z)^2
	- \cE\bigg(\f{\kappa-q^{1/2}z}{\sqrt{1-q}}
	\bigg) \f{q(\kappa-q^{-1/2}z)}{(1-q)^{3/2}}
	\bigg\}
	= \psi - \f{q I}{1-q}\,.\end{align*}
Rearranging gives $I=\psi(1-q)$ as claimed.
\end{proof}
\end{lem}

\begin{cor}\label{c:gg.decreasing}
The function $\GG_\star(\alpha)$
is a decreasing function on $\alpha\in(\albd,\aubd)$.

\begin{proof}
Denote $q\equiv \sq(\alpha)$ and $\psi\equiv\spsi(\alpha)$. We first show that $(q,\psi)$ is a stationary point of the function $\GG(\alpha,q,\psi)$ defined by \eqref{e:GG}. Indeed,
 Lemma~\ref{l:intbyparts} gives
	\beq\label{e:dGdq.portion}
	\alpha \f{d}{dq}
	\int \log\bPhi\bigg(
	\f{\kappa-q^{1/2}z}{\sqrt{1-q}}
	\bigg)\,\varphi(z)\,dz
	= -\alpha\int
	\cE\bigg(\f{\kappa-q^{1/2}z}{\sqrt{1-q}}
	\bigg)
		\f{\kappa-q^{-1/2}z}{2(1-q)^{3/2}}
		\,\varphi(z)\,dz
	= -\f{I}{2(1-q)} = -\f{\psi}{2}\,,\eeq
from which it follows that $\pd\GG/\pd q=0$.
A direct calculation gives
	\beq\label{e:dGdpsi.portion}
	\f{d}{d\psi}
	\int \log(2\ch(\psi^{1/2}z))\,\varphi(z)\,dz
	= \int 
	\f{\tnh(\psi^{1/2}z) z}{2\psi^{1/2}}
	\,\varphi(z)\,dz
	=
	\int \f{\tnh'(\psi^{1/2}z)}{2}
	\,\varphi(z)\,dz
	= \f{1-q}{2}\,,\eeq
from which it follows that $\pd\GG/\pd\psi=0$. It follows that
	\[\f{d \GG_\star(\alpha)}{d\alpha}
	= \f{\pd\GG(\alpha,q,\psi)}{\pd\alpha}
	= \int\log\bPhi\bigg(\f{\kappa-q^{1/2}z}{\sqrt{1-q}}
	\bigg)\,\varphi(z)\,dz\]
which is negative.
\end{proof}
\end{cor}

\begin{lem}\label{l:recursion.monotone}
As defined by \eqref{e:q.psi.recursion},
the functions $P$ and $R(\cdot,\alpha)$ are nondecreasing.

\begin{proof}
Since $\tnh(\psi^{1/2}z)$ has the same sign as $z$
while $\tnh'(\psi^{1/2}z)\in(0,1]$, the derivative
	\[P'(\psi)
	= \int 
	\f{\tnh(\psi^{1/2}z)
	\tnh'(\psi^{1/2}z)
	z}{\psi^{1/2}}
	\,\varphi(z)\,dz\]
is nonnegative. We next calculate
	\begin{align}\nonumber
	\f{\pd R}{\pd q}
	&= \f{\alpha}{(1-q)^2}\int 
	\cE\bigg(\f{\kappa-q^{1/2}z}{\sqrt{1-q}}\bigg)^2
	\bigg\{
	1 + \bigg[
	\cE\bigg(\f{\kappa-q^{1/2}z}{\sqrt{1-q}}\bigg)
	- \f{\kappa-q^{1/2}z}{\sqrt{1-q}} \bigg]
	\f{\kappa - q^{-1/2} z}{\sqrt{1-q}}
	\bigg\}
	\,\varphi(z)\,dz\\
	&= \f{R(q)}{1-q}
	+ \f{\alpha}{(1-q)^2}\int 
	\cE\bigg(\f{\kappa-q^{1/2}z}{\sqrt{1-q}}\bigg)
	\cE'\bigg(\f{\kappa-q^{1/2}z}{\sqrt{1-q}}\bigg)
	\f{\kappa - q^{-1/2} z}{\sqrt{1-q}}
	\,\varphi(z)\,dz\,.\label{e:dRdq}
	\end{align}
The right-hand side of \eqref{e:dRdq} is the sum of two terms where the first is clearly nonnegative. We claim that the second term is nonnegative also. To this end, let
	\beq\label{e:xi.zeta}
	\xi_{q,z}
	\equiv\f{\kappa-q^{1/2}z}{\sqrt{1-q}}\,,\quad
	\zeta_{q,z}
	\equiv\f{\kappa-q^{-1/2}z}{\sqrt{1-q}}
	= 2(1-q) \f{d\xi_{q,z}}{dq}\,.\eeq
Consider the change of variables $\bar{z}=2\kappa q^{1/2}-z$. We have the relations $\zeta_{q,\bar{z}}=-\zeta_{q,z}$ and $\xi_{q,\bar{z}}-\xi_{q,z} = 2q\zeta_{q,z}$. Note also that $\zeta_{q,z}\ge0$ if and only if $\kappa-q^{-1/2}z\ge0$. The second term
of \eqref{e:dRdq} can be expressed as
	\[\f{\alpha}{(1-q)^2}
	\int_{z\le q^{1/2}\kappa}
	\bigg(
	\cE(\xi_{q,z})\cE'(\xi_{q,z})
	-\cE(\xi_{q,\bar{z}})\cE'(\xi_{q,\bar{z}})
	\bigg)
	\zeta_{q,z} \,\varphi(z)\,dz\,,\]
which is nonnegative because $\cE$ and $\cE'$ are both nondecreasing functions (see Lemma~\ref{l:EE} below). Therefore $\pd R/\pd q$ is also nonnegative, concluding the proof.
\end{proof}
\end{lem}

\begin{lem}[computer-assisted]
\label{l:into}
For $R(q,\alpha)$ as defined by \eqref{e:q.psi.recursion} the following hold:
\begin{enumerate}[a.]
\item\label{l:into.a} The function $q\mapsto R(q,\albd)$
maps $(\qlbd,\qlu)$ into $(\psilbd,\psilu)$;
\item\label{l:into.b} The function $q\mapsto R(q,\aubd)$
maps $(\qul,\qubd)$ into $(\psiul,\psiubd)$.
\end{enumerate}
As a consequence, $R(q,\alpha)\in(\psilbd,\psiubd)$
for all $(q,\alpha)\in(\qlbd,\qubd)\times(\albd,\aubd)$.
\begin{proof}
It is clear that $R(q,\alpha)$ is lower bounded by
	\[R_\lbd(q,\alpha)\equiv
	\alpha \int_{|z|\le 10}
	F_q(q^{1/2}z)^2\,\varphi(z)\,dz\,.\]
Moreover, since $\cE(x)\le 1+|x|$ (proved in Lemma~\ref{l:EE} below) we can bound
	\begin{align*}
	\aubd \int_{|z|\ge 10} 
	F_{\qubd}(\sqrt{\qubd} z)^2\,\varphi(z)\,dz
	&\le \f{\aubd}{1-\qubd}
	\int_{|z|\ge 10} \bigg(
	1 + \bigg| \f{\kappa-\sqrt{\qubd}z}
		{\sqrt{1-\qubd}} \bigg|
	\bigg)^2\,\varphi(z)\,dz \\
	&= \f{2\aubd}{1-\qubd}
	\bigg(
	\bPhi(10)
	+ \f{2 \sqrt{\qubd} \phi(10)}{\sqrt{1-\qubd}}
	+ \f{\qubd ( 10 \phi(10) + \bPhi(10))}
		{1-\qubd}
	\bigg)\,.\end{align*}
The last expression can be evaluated by computer to very high precision, and we find that it is smaller than $1/10^{20}$. It follows that
for all $(q,\alpha)\in(\qlbd,\qubd)\times(\albd,\aubd)$, the function $R(q,\alpha)$ is upper bounded by
	\[R_\ubd(q,\alpha)
	\equiv \alpha \int_{|z|\le 10}
	F_q(q^{1/2}z)^2\,\varphi(z)\,dz
	\le \f1{10^{20}}\,.\]
Both $R_\lbd$ and $R_\ubd$ can be evaluated by computer, so we can verify that
	{\setlength{\jot}{0pt}\begin{align*}
	R_\lbd(\qlbd,\albd) &> \psilbd\,,\\
	R_\ubd(\qubd,\albd) &< \psilu\,,\end{align*}}%
which proves part~\ref{l:into.a}.
Part~\ref{l:into.b} is proved similarly.
The last claim follows since 
$R(q,\alpha)$ is also nondecreasing in $\alpha$.
\end{proof}
\end{lem}

\begin{lem}[computer-assisted]\label{l:at}
For all $\alpha\in(\albd,\aubd)$ the condition \eqref{e:at} holds. In fact we have
	\[\sup_{(q,\alpha)\in G}
	 \f{dP(R(q,\alpha))}{dq} \le 0.96\]
where $G\equiv(\qlbd,\qubd)\times(\albd,\aubd)$.

\begin{proof}
We calculated $P'(\psi)$ in the proof of Lemma~\ref{l:recursion.monotone}.
It can be simplified as
	\[P'(\psi)
	= \int \f{2-\ch(2\psi^{1/2}z)}{\ch(\psi^{1/2}z)^4}
	\,\varphi(z)\,dz\,.\]
We separate the right-hand side into two terms and calculate
	\begin{align*}
	\f{d}{d\psi}
	\int \f{2}{\ch(\psi^{1/2}z)^4}\,\varphi(z)\,dz
	&= \int \f{-4\sh(\psi^{1/2}z)}
		{\ch(\psi^{1/2}z)^5} 
		\f{z}{\psi^{1/2}}\,\varphi(z)\,dz
	\le0\,,\\
	\f{d}{d\psi}
	\int \f{\ch(2\psi^{1/2}z)}{\ch(\psi^{1/2}z)^4}
		\,\varphi(z)\,dz
	&= \int \f{-2\sh(\psi^{1/2}z)^3}{\ch(\psi^{1/2}z)^5}
	\f{z}{\psi^{1/2}}
	\,\varphi(z)\,dz\le0\end{align*}
It follows that for all $\psi\in(\psilbd,\psiubd)$ we have
	\[P'(\psi)
	\le
	\int \f{2}{\ch(\sqrt{\psilbd}z)^4}
		\,\varphi(z)\,dz
	-\int \f{\ch(2\sqrt{\psiubd}z)}
		{\ch(\sqrt{\psiubd}z)^4}
		\,\varphi(z)\,dz
	\le 0.08\,,\]
where the last bound is computer-verified. We turn to $\pd R/\pd q$ which was also computed in the proof of Lemma~\ref{l:recursion.monotone}. Recalling \eqref{e:xi.zeta}, let us denote
	{\setlength{\jot}{0pt}
	\begin{alignat}{2}\nonumber
	\xiubd&\equiv 
	\max\set{\xi_{\qlbd,z},\xi_{\qubd,z}}\,,\quad
	&\zeubd&\equiv
	\max\set{\zeta_{\qlbd,z},\zeta_{\qubd,z}}\,,\\
	\label{e:xi.zeta.bds}
	\xilbd&\equiv 
	\min\set{\xi_{\qlbd,z},\xi_{\qubd,z}}\,,\quad
	&\zelbd&\equiv
	\min\set{\zeta_{\qlbd,z},\zeta_{\qubd,z}}\,.
	\end{alignat}}
By \eqref{e:dRdq} and the monotonicity of $R(\cdot,\alpha)$ proved in Lemma~\ref{l:recursion.monotone}, it holds for all $(q,\alpha)\in G$ that
	\[\f{\pd R(q,\alpha)}{\pd q}
	\le \f{R(\qubd)}{1-\qubd}
	+ \f{\aubd}{(1-\qubd)^2}
	\int\cE(\xiubd) \cE'(\xiubd) \zeubd
		\,\varphi(z)\,dz
	\le 12\,,\]
where the last bound is again computer-verified.
It follows from Lemma~\ref{l:into} that
	\[\sup_{(q,\alpha)\in G}
	 \f{dP(R(q,\alpha))}{dq}
	\le
	\bigg\{\sup_{\psi\in(\psilbd,\psiubd)}
	P'(\psi)\bigg\}\bigg\{
	\sup_{(q,\alpha)\in G}
	\f{\pd R(q,\alpha)}{\pd q}
	\bigg\}\,.\]
Multiplying the previous bounds gives the result.
\end{proof}
\end{lem}

\begin{cor}[computer-assisted] \label{c:q.psi.bds}
For $\alpha\in(\albd,\aubd)$ there is a unique pair
of values $\sq\in(\qlbd,\qubd)$ and $\spsi\in(\psilbd,\psiubd)$ satisfying $\spsi=R(\sq,\alpha)$ and $\sq=P(\spsi)$.
We have $\sq(\aubd)\in(\qul,\qubd)$
and $\sq(\albd)\in(\qlbd,\qlu)$;
and as a consequence
$\sq(\alpha)\in(\qlbd,\qubd)$
for all $\alpha\in(\albd,\aubd)$.

\begin{proof}
It follows by Lemma~\ref{l:at} that for any $\alpha\in(\albd,\aubd)$ the map $q\mapsto P(R(q,\alpha))-q$ is strictly decreasing on the interval $(\qlbd,\qubd)$, and so has at most one zero.
We verify by computer that
	{\setlength{\jot}{0pt}\begin{align*}
	P(R(\qlbd,\albd))-\qlbd
	&> 0 > P(R(\qlu,\albd))-\qlu\,,\\
	P(R(\qul,\qubd))-\qul
	&> 0 > P(R(\qubd,\qubd))-\qubd\,,\end{align*}}%
from which it follows that 
$\sq(\albd)\in(\qlbd,\qlu)$
and $\sq(\aubd)\in(\qul,\qubd)$.
It follows from Lemma~\ref{l:recursion.monotone}
that for any fixed $q$ the map
$\alpha\mapsto P(R(q,\alpha))$ is 
nondecreasing; it follows that
for all $\alpha\in(\albd,\aubd)$
there is a unique
$\sq(\alpha)\in(\albd,\aubd)$,
which is nondecreasing in $\alpha$.
It then follows by Lemma~\ref{l:into}
that $\spsi(\alpha)\in(\psilbd,\psiubd)$, concluding the proof.
\end{proof}
\end{cor}

\begin{cor}[computer-assisted]\label{c:gg.bds}
We have $\GG_\star(\albd)>0>\GG_\star(\aubd)$,
so $\alpha_\star\in(\albd,\aubd)$.

\begin{proof}
Making use of \eqref{e:dGdq.portion}
and \eqref{e:dGdpsi.portion} we have
	\[\GG_\star(\aubd)
	\le -\f{\psiul(1-\qubd)}{2}
	+\int\log(2\ch(\sqrt{\psiubd}z))
	\,\varphi(z)\,dz
	+\alpha\int\log\bPhi
	\bigg(\f{\kappa-\sqrt{\qul}z}{\sqrt{1-\qul}}
	\bigg)\,\varphi(z)\,dz < -\f1{10^{12}}\]
where the last bound is computer-verified.
Similarly we have
	\[\GG_\star(\albd)\ge-\f{\psilu(1-\qlbd)}{2}
	+\int \log(2\ch(\sqrt{\psilbd}z))
	\,\varphi(z)\,dz
	+\alpha\int \log\bPhi
	\bigg(\f{\kappa-\sqrt{\qlu}z}{\sqrt{1-\qlu}}
	\bigg)\,\varphi(z)\,dz> \f1{10^{12}}\]
where the last bound is again computer-verified.
Recalling Corollary~\ref{c:gg.decreasing},
it follows that $\alpha\mapsto \GG_\star(\alpha)$
has a unique zero $\alpha_\star\in(\albd,\aubd)$ as claimed.
\end{proof}
\end{cor}

\begin{proof}[Proof of Proposition~\ref{p:gardner}]
Follows by combining Lemma~\ref{l:at}
with Corollaries~\ref{c:gg.decreasing}, \ref{c:q.psi.bds}, and \ref{c:gg.bds}.
\end{proof}

\noindent We now conclude our moment calculation:

\begin{proof}[Proof of Theorem~\ref{t:second}] Write $\Filt\equiv\sigma(\textsf{DATA})$ for $\textsf{DATA}\equiv\textsf{DATA}_{M,\Nall,\DELTA,t}$ as given by \eqref{e:filt.t}. Recall from Section~\ref{sec-restricted-partition-function} that if there exists no $\bhJ$ satisfying \eqref{e:def.bhJ} then we simply set $\bZ_{\DELTA,t,\EPS}\equiv0$. The event that the desired $\bhJ$ exists is $\Filt$-measurable. On the event, we define $\bZ_{\DELTA,t,\EPS}$ by \eqref{e:def.restricted.part.fn}, and the conditional first moment is given by (recalling $\sat_{J,\EPS}\equiv \sat_{J,\DELTA,t,\EPS}$)
	\beq\label{e:write.out.first.mmt}
	\E\Big(\bZ_{\DELTA,t,\EPS}
		\,\Big|\,
		\Filt\Big)
	=\sum_{J\in\bbH_\EPS}
	\oP\Big(\sat_{J,\EPS}
	\,\Big|\, \Filt \Big)\,.\eeq
The probability $\oP(\sat_{J,\EPS}\,|\, \Filt )$ is precisely the left-hand side of \eqref{e:changeofmsr.first}. The terms on the right-hand side of \eqref{e:changeofmsr.first}
were computed in Sections~\ref{s:row}~and~\ref{s:adm}: in particular, it follows by \eqref{e:conv.to.PP} and \eqref{e:adm.givensr.first} together that
	\[\adjustlimits
	\liminf_{\DELTA\downarrow0}
	\liminf_{t\uparrow\infty}
	\liminf_{\EPS\downarrow0}
	\liminf_{\Nall\uparrow\infty}
	\f{\log\oP(\sat_{J,\EPS}
	\,|\, \Filt )}{\Nall}
	\ge\PP_\star\,.\]
Combining with \eqref{e:count} gives
	\[\adjustlimits
	\liminf_{\DELTA\downarrow0}
	\liminf_{t\uparrow\infty}
	\liminf_{\EPS\downarrow0}
	\liminf_{\Nall\uparrow\infty}
	\f{\log \E(\bZ_{\DELTA,t,\EPS}
	\,|\,\Filt)}{\Nall}
	\ge \HH_\star+\PP_\star\]
on the event that there exists $\bhJ$ satisfying \eqref{e:def.bhJ}. This event has positive probability by Proposition~\ref{p:kr}, so we have proved the first moment bound~\eqref{e:first.mmt.gg}. For the second moment, if $\bhJ$ does not exist then $\bZ_{\DELTA,t,\EPS}=0$ and \eqref{e:second.mmt.bound} trivially holds, so we again restrict to the event that we have the desired $\bhJ$. Let $\Lambda_N$ denote the set of values $\lmin\le\lambda\le1$ with $N\lambda$ integer-valued. Define $\Lambda_{N,\usf{CLT}}$ to be the subset of values $\lambda\in\Lambda_N$ with $|\lambda| \le \lmsmall$ (a small positive constant to be chosen), and set 
$\Lambda_{N,\usf{LD}}\equiv \Lambda_N\setminus \Lambda_{N,\usf{CLT}}$. 
Then analogously to \eqref{e:write.out.first.mmt} we have (recalling $\sat_K\equiv\sat_{K,\DELTA,t}$)
	\beq\label{e:second.moment.separate.regimes}
	\E\Big((\bZ_{\DELTA,t,\EPS})^2
		\,\Big|\,\Filt\Big)
	\le
	\sum_{\substack{\lambda\in
		\Lambda_{N,\usf{CLT}},\\
	(J,K)\in\bbH_\EPS(\lambda)}}
	\oP(\sat_{J,\EPS}
	\sat_K\,|\, \Filt)
	+
	\sum_{\substack{\lambda\in
		\Lambda_{N,\usf{LD}},\\
	(J,K)\in\bbH_\EPS(\lambda)}
	}
	\oP(\sat_{J,\EPS}
	\sat_K\,|\, \Filt)
	\,,\eeq
where $\oP(\sat_{J,\EPS}\sat_K\,|\,\Filt)$ is precisely the left-hand side of \eqref{e:changeofmsr}. We use \eqref{e:changeofmsr.first} and \eqref{e:changeofmsr} to express
	\beq\label{e:changeofmsr.ratio}
	\f{\oP(\sat_{J,\EPS}\sat_K
	\,|\, \Filt)}
	{\oP(\sat_{J,\EPS}
	\,|\, \Filt)^2}
	= \f{\tP(\sat_K\giv
		\bar{E}_\row,\sat_{J,\EPS})
	\,\tp_{\admone}(E_{\admone})^2}
	{\tP(\sat_J\giv\bar{E}_\row)
	\,\tp_{\adm}(E_\adm)}
	\f{\tp_{\adm|\row\bsat}
		(E_\adm\giv E_\row)}
	{\tp_{\admone|\row\bsatone}
	(E_{\admone}\giv E_\row)^2}\,.\eeq
On the right-hand side of \eqref{e:changeofmsr.ratio}, the last term in the denominator is of constant order by \eqref{e:adm.givensr.first}, while the first term in the denominator is estimated by \eqref{e:PP.cst}. For small $\lm$, the other terms of the right-hand side of \eqref{e:changeofmsr.ratio} are estimated by \eqref{e:PP.clt.regime} and \eqref{e:adm.givensr.clt}, while the cardinality of $\bbH_\EPS(\lm)$ is estimated by \eqref{e:pairs.clt}. Altogether it gives that
	\begin{align}\nonumber
	&\sum_{\lambda\in\Lambda_{N,\usf{CLT}}}
	\sum_{(J,K)\in\bbH_\EPS(\lambda)}
	\oP(\sat_{J,\EPS}\sat_K\,|\, \Filt)
	\le
	\sum_{\lambda\in\Lambda_{N,\usf{CLT}}}
	\sum_{(J,K)\in\bbH_\EPS(\lambda)}
	\oP(\sat_{J,\EPS}\,|\,\Filt)^2
	\exp\bigg\{N \bigg( 
	\f{c\lm^2}{2}
	+O(\cst(\lmsmall)^3)
	\bigg)
	\bigg\}\\\nonumber
	&\le 
	\bigg\{\#\bbH_\EPS \exp\bigg\{
	\Big(\oneM,\log\bPhi(\bxi)\Big)
	+ \f{N\bpsi(1-\bq)}{2}\bigg\}\bigg\}^2
	\sum_{\lambda\in\Lambda_{N,\usf{CLT}}}
	\f{\#\bbH_\EPS(\lm)}
		{(\#\bbH_\EPS)^2}
	\exp\bigg\{\f{Nc\lm^2}{2}
	+O(\cst(\lmsmall)^3)
	\bigg\}\\
	&\le \bigg\{
	\E\Big(\bZ_{\DELTA,t,\EPS}
	\,\Big|\,\Filt\Big)
	\bigg\}^2
	\sum_{\lambda\in\Lambda_{N,\usf{CLT}}}
	\exp\bigg\{\f{N\lm^2}{2}
		\bigg(c-\f1{\sigma^2}\bigg)
	+O(\cst(\lmsmall)^3)
	\bigg\}
	\le \cst\,\bigg\{
	\E\Big(\bZ_{\DELTA,t,\EPS}
	\,\Big|\,\Filt\Big)
	\bigg\}^2\,,
	\label{e:second.mmt.bound.clt}
	\end{align}
where the last equality holds for sufficiently small positive $\lmsmall$, since by \eqref{e:conv.to.PP.second.deriv}
and \eqref{e:sigma.limit.ellprime} we see that
(for $\rho$ as in \eqref{e:PP.clt.regime}
and $\sigma$ as in \eqref{e:explicit.sigma.finite},
both depending on $M,\Nall,\DELTA,t,\EPS$),
	\[\adjustlimits
	\limsup_{\DELTA\downarrow0}
	\limsup_{t\uparrow\infty}
	\limsup_{\EPS\downarrow0}
	\limsup_{\Nall\uparrow\infty} 
	\bigg(\rho_{M,\Nall,\DELTA,t,\EPS}
	-\f1{(\sigma_{M,\Nall,\DELTA,t,\EPS})^2}\bigg)
	\le \PP''(0)+\HH''(0)\]
which is negative for all $\alpha\in(\albd,\aubd)$. This shows that the first term on the right-hand side of \eqref{e:second.moment.separate.regimes} is upper bounded (up to a $\cst$ factor) by the square of the conditional first moment, as desired. It remains to bound the second term on the right-hand side of \eqref{e:second.moment.separate.regimes}. It follows by \eqref{e:conv.to.PP.pair},
\eqref{e:adm.givensr.ld}, and \eqref{e:pairs.ld} that
	\[\adjustlimits
	\limsup_{\DELTA\downarrow0}
	\limsup_{t\uparrow\infty}
	\limsup_{\EPS\downarrow0}
	\limsup_{\Nall\uparrow\infty} 
	\f1{\Nall}
	\log\bigg\{
	\sum_{(J,K)\in\bbH_\EPS(\lambda)}
	\oP(\sat_{J,\EPS}
	\sat_K\,|\, \Filt)\bigg\}
	\le 2\GG_{\star\kall}(\aall)
	+\SE_{\kall,\aall}(\lm)\,,\]
as was claimed in \eqref{e:informal.statement.SE}. For $\kall=0$ and \emph{all} $\aall\in(\albd,\aubd)$, we have by exact calculation (i.e., without any numerical evaluations) that the function
$\SE(\lm)\equiv\SE_{0,\aall}(\lm)$
satisfies $\SE(0)=\SE'(0)=0$,
$\SE(1)=-\GG_\star(\aall)$,
and $\SE'(\lm)\to\infty$ as $\lm\uparrow1$
(the last claim can be seen from Proposition~\ref{p:sqrt} below). Moreover, recall from the discussion following \eqref{e:BB} that $\AAA\le0$, with $\AAA(0)=\AAA'(0)=0$ and $\AAA''(0)\le0$;
this implies $\SE''(0) \le \PP''(0)+\HH''(0)$.
We then verify numerically that $\PP''(0)+\HH''(0)<0$ (Proposition~\ref{p:grid.result}\ref{p:grid.result.c} below). In combination with Condition~\ref{c:G} it follows that for $\alpha'\in(\albd,\asat)$ sufficiently close to $\asat$, it holds for all $\aall\in(\alpha',\asat)$ that the $\SE_{0,\aall}(\lm)$ \emph{is uniquely maximized at $\lm=0$}. As a result
	\beq\label{e:second.mmt.bound.ld}
	\sum_{\lambda\in\Lambda_{N,\usf{LD}}}
	\sum_{(J,K)\in\bbH_\EPS(\lambda)}
	\oP(\sat_{J,\EPS}\sat_K\,|\, \Filt)
	\le \bigg(
	\exp\bigg\{
	2 N \GG_\star(\alpha) - \f{N}{\cst}
	\bigg\}
	+\exp\bigg\{
	N \GG_\star(\alpha) 
	\bigg\}\bigg)
	\exp(N\acute{\ETA})\,.\eeq
Combining \eqref{e:second.mmt.bound.clt} and \eqref{e:second.mmt.bound.ld} proves the second moment bound \eqref{e:second.mmt.bound}.\end{proof}

\noindent In the remaining sections we outline our plan for verifying Condition~\ref{c:G}. For this analysis it suffices to work in the limit where $\kall=\kappa$ and $\aall=\alpha$. To simplify the notation, for the remainder of the paper we denote these parameters simply by $\kappa$ and $\alpha$; and we let $\sq=q_{\star\kappa}(\alpha)$ and $\spsi=\psi_{\star\kappa}(\alpha)$.

\section{Quantitative estimates of limiting exponents}\label{s:quant.bounds}

\noindent Throughout the following we fix $\kappa\equiv0$. We assume $\alpha\in(\albd,\aubd)$ and denote $(q,\psi)\equiv (\sq(\alpha),\spsi(\alpha))$ as given by Proposition~\ref{p:gardner}. In this section we provide quantitative estimates on the limiting functions $\PP(\lm)$, $\AAA(\lm)$, and $\HH(\lm)$ obtained in the previous sections. The key definitions (\eqref{e:def.ell.A}, \eqref{e:HH.pair}, \eqref{e:PP.pair}, \eqref{e:II}, \eqref{e:BB}) are repeated below: we write
	\beq\label{e:II.repeat}
	\II_s(\lm)
	\equiv
	\alpha\int\int\log
	\bPhi\bigg(
	\f{\gamma z-\lm\nu}{\sqrt{1-\lm^2}}
	- \f{\cE(\gamma z) \cdot s}
		{\psi^{1/2}\sqrt{1-q}}
	\bigg)\,
	\varphi_{\gamma z}(\nu)\,d\nu
	\,\varphi(z)\,dz\,,\eeq
and abbreviate $\II(\lm) \equiv \II_0(\lm)$. We then define
	\begin{align}\label{e:PP.pair.repeat}
	\PP(\lm)
	&= -\PP_\star +\f{\psi(1-q)(1-\lm)}{2(1+\lm)} + \II(\lm)
	= -\f{\psi(1-q)\lm}{1+\lm}
	+\II(\lm)-\II(0)
	\,,\\
	\label{e:BB.repeat}
	\AAA(\lm)
	&\equiv \inf_s\bigg\{
	\f{s^2}{2}
	- \psi^{1/2}\sqrt{1-q} \bigg(\f{1-\lm}{1+\lm}\bigg)^{1/2} s
	+\II_s(\lm) - \II(\lm)
	\bigg\}\,.
	\end{align}
We recall that $\PP(0)=\AAA(0)=0$, $\PP(1)=-\PP_\star$,
and $\AAA(\lm)\le0$ for all $\lm$. We also have
\begin{align}
	\label{e:def.ell.A.repeat}
	\ell(A)
	&\equiv \int \f{D_{\psi^{1/2} z}(A)}{1-q}
	\varphi(z)\,dz\,,\\
	\label{e:def.HH.pair.repeat}
	\HH(\ell(A)) \equiv \frH(A)
	&\equiv-2\HH_\star+
	\int \Gamma\Big(\psi^{1/2}z,
		D_{\psi^{1/2}z}(A)
		\Big)\varphi(z)\,dz\,,
	\end{align}
where $D_H(A)$ is given by \eqref{e:D.A} and $\Gamma$ is the entropy of the probability distribution \eqref{e:P.H.D}. We recall that $\HH(0)=0$, $\HH(1)=-\HH_\star$, and $A\mapsto\ell(A)$ is an increasing bijection from $(0,\infty)$ to $(-\lmin,1)$. Throughout the following we will use the notations $H=\psi^{1/2}z$, $m\equiv\tnh H$, and $p\equiv(1+m)/2$  with the understanding that there is a bijective correspondence among the variables $H,m,p$. Let $\SE(\lm)\equiv\HH(\lm)+\PP(\lm)+\AAA(\lm)$.

\subsection{Bounds for highly correlated regime}

We first estimate $\HH(\lm)$ and $\PP(\lm)$ near $\lm=1$:

\begin{ppn}\label{p:HH.PP.near.lambda.one}
For all $0.98\le\lm<1$ we have
$\HH(\lm)+\PP(\lm) < \HH(1)+\PP(1)=-(\HH_\star+\PP_\star)$.
\end{ppn}

\noindent To begin the proof we first estimate $\HH(\lm)$ near $\lm=1$, corresponding to $A\uparrow\infty$.

\begin{lem}[computer-assisted]
\label{l:ell.A.large.A.improved}
For all $A\in(0,\infty)$ we have
	\[1-\ell(A)
	\le \f{1.78}{A} + \f{4.3}{A^2}\,.
	\]
As a consequence, $1-\ell(A) \le 1.83/A$ for all $A \ge 100$.

\begin{proof}
From the relation $q=P(\psi)$
(\eqref{e:q.psi.recursion}~and~Proposition~\ref{p:gardner})
and the definition
\eqref{e:def.ell.A.repeat} of $\ell(A)$, we can express
	\beq\label{e:diff.ell.A}
	\f{(1-q)(1-\ell(A))}{2/A}
	= \int e_H(A) \,\varphi(z)\,dz
	\eeq
where $H\equiv\psi^{1/2}z$ and, with $\Delta\equiv\Delta_H(A)\equiv\sqrt{A^2+m^2-(Am)^2}$
as before, we define
	\[
	e_H(A)
	\equiv \f{1-m^2-D_H(A)}{2/A}
	= \f{A(1-m^2)}{\Delta+1}
	= \f{A(\Delta-1)}{(A+1)(A-1)}\,.\]
Since $\Delta$ is always sandwiched between $1$ and $A$, it holds for all $A\in(0,\infty)$ that
	\[0\le 
	e_H(A)\le \f{\Delta-1}{A-1}
	\le \f{\Delta}{A}
	= \bigg( 1-m^2 + \f{m^2}{A^2} \bigg)^{1/2}
	\le \sqrt{1-m^2} + \f{|m|}{A}\,.
	\]
Next we note that for all real $z$,
	\[
	\f{d\sqrt{1-\tnh(\psi^{1/2}z)^2}}{d\psi}
	= -\f{\tnh(\psi^{1/2}z)\tnh'(\psi^{1/2}z)}{\sqrt{1-\tnh(\psi^{1/2}z)^2}}
	\f{z}{2\psi^{1/2}}\le0\,,
	\]
since $\tnh'(H) = 1-\tnh(H)^2 \ge0$  while $z$ and $\tnh(\psi^{1/2}z)$ have the same sign.
We also have
	\[
	\f{d\tnh(\psi^{1/2}z)}{d\psi}
	= \f{\tnh'(\psi^{1/2}z) z}{2\psi^{1/2}}
	\]
which has the same sign as $z$. Substituting into \eqref{e:diff.ell.A} gives
	\[
	1-\ell(A)
	\le \f{2}{A(1-\qubd)}
	\int_{-9}^9
	\bigg\{
	\sqrt{1-
	\tnh(\sqrt{\psilbd}z)^2}
	+ \f{|\tnh(\sqrt{\psiubd}z)|}{A}
	\bigg\}
	\,\varphi(z)\,dz
	\le \f{1.78}{A} + \f{4.3}{A^2}\,,
	\]
where the last bound is by numerical integration.
This proves the lemma.
\end{proof}
\end{lem}

\begin{cor}\label{c:entropy.near.one}
Parametrizing $\lm\equiv1-\iota$,
it holds for all $0\le\iota\le 0.025$ that
	\[
	\HH(\lm)-\HH(1)
	\le 0.81\,\iota
	 + \f\iota2\log\f1\iota\,.
	\]

\begin{proof}
Let us abbreviate $\iota(A)\equiv 1-\ell(A)$. We find by numerical integration that
$\iota(100)>0.025$. Then, since $\iota$ is a decreasing function of $A$, it follows that for all $s\le0.025$ we have $\iota^{-1}(s)>100$. Substituting into the conclusion of Lemma~\ref{l:ell.A.large.A.improved} gives
$\iota(\iota^{-1}(s)) \le 1.83/\iota^{-1}(s)$, that is, $\iota^{-1}(s) \le 1.83/s$, for all $s\le0.025$. Substituting into \eqref{e:dHH.dlambda} in turn gives, for all $s\le 0.025$,
	\[
	\f{d\HH'(1-s)}{ds}
	= \f{\log \iota^{-1}(s)}{2}
	\le \f{\log (1.83/s)}{2}\,.
	\]
Integrating this bound from $s=0$ to $s=\iota$ gives
	\[
	\HH(\lm)
	-\HH(1)
	\le\bigg(  \f{\log1.83}{2}+\f12\bigg) \iota
	+ \f{\iota}{2}\log\f1\iota\,,
	\]
from which the claim follows.
\end{proof}\end{cor}

\noindent We next estimate $\PP(\lm)$ near $\lm=1$:
\begin{ppn}\label{p:sqrt}
For $\lm\equiv1-\iota$
with $0\le\iota\le 0.025$,
	\[
	\f{\PP(\lm)-\PP(1)}{\iota^{1/2}}
	\le 0.285\,\iota^{1/2}
	-0.45\,.
	\]

\begin{proof}
The middle term of \eqref{e:PP.pair.repeat} can be bounded as
	\beq\label{e:pos.sqrt}
	\f{\psi(1-q)\iota}{2(2-\iota)}
	\le
	\f{\psiubd(1-\qlbd)\iota}{2(2-0.025)}
	\le 0.285 \, \iota\,.\eeq
Making the change of variables
$x=\sqrt{1-\lm^2}u$ we rewrite
	\beq\label{e:II.rescaled}
	\f{\II(\lm)}{\sqrt{1-\lm^2}}
	= \alpha
	\int_{-\infty}^\infty
	\int_0^\infty
	\log\bPhi(c_\lm \psi z-\lm u)
	\,
	\f{\varphi(\gamma z + \sqrt{1-\lm^2}u)}
		{\bPhi(\gamma z)}
		\varphi(z)\,dz\,.\eeq
Assume $0\le \llm\le\lm\le1$; denote
$L_\lbd\equiv (\llm)^2$
and $c_\ubd\equiv\sqrt{1-\ulm}/\sqrt{1+\llm}$.
Then
	\[
	\log\bPhi(c_\lm \psi z-\lm u)
	\le\begin{cases}
	\log\bPhi(-u)
		&\textup{for }z\ge0\,,\\
	\log\bPhi(c_\ubd\psiubd z - u)
		&\textup{for }z\le0\,.
	\end{cases}
	\]
since $\log\bPhi$ is a decreasing function. We also have
	\[
	\f{\varphi(\gamma z + \sqrt{1-\lm^2}u)}
		{\bPhi(\gamma z)}
	\ge \begin{cases}
	\displaystyle
	\f{\varphi(\gamubd z 
		+\sqrt{1-L_\lbd}u)}
		{\bPhi(\gamlbd z)}
	&\textup{for }z\ge0\,,\\
	\displaystyle
	\f{\cE(\gamubd z)}
		{\exp\{ u^2(1-L_\lbd)/2\} }
	&\textup{for }z\le0\,.
	\end{cases}
	\]
Substituting these bounds into \eqref{e:II.rescaled}
gives
	\begin{align*}
	\f{\II(\lm)}{\iota^{1/2}}
	&\le
	\albd \sqrt{1+\lm_\lbd}\bigg\{
	\int_0^9\int_0^9
	\log\bPhi(-u)
	\f{\varphi(\gamubd z 
		+\sqrt{1-L_\lbd}u)}
		{\bPhi(\gamlbd z)}
	\,\varphi(z)\,dz\\
	&\qquad\qquad\qquad\qquad\quad +
	\int_{-9}^0\int_0^9
	\log\bPhi(c_\ubd\psiubd z-u)
	\f{\cE(\gamubd z)}
		{\exp\{ u^2(1-L_\lbd)/2\} }
	\,\varphi(z)\,dz
	\bigg\}\,.
	\end{align*}
For $\lm_\lbd= 1-0.025$
the right-hand side is $\le -0.45$.
Combining with \eqref{e:pos.sqrt}
proves the claim.
\end{proof}
\end{ppn}

\begin{proof}[Proof of Proposition~\ref{p:HH.PP.near.lambda.one}]
Combining Corollaries~\ref{c:entropy.near.one}~and~\ref{p:sqrt} gives
that for $\lm\equiv1-\iota$ with $\iota\le0.025$,
	\[
	\f{\HH(\lm)+\PP(\lm)
	-\HH(1)-\PP(1)}{\iota^{1/2}}
	\le \iota^{1/2}\bigg( 1.1 + \f{\log(1/\iota)}{2}
		\bigg) - 0.45\,.
	\]
This is negative for all $\iota\le0.02$ so the claim is proved.
\end{proof}

\subsection{Bounds on parametrization} 

We next give computable bounds on the mapping $A\mapsto \ell(A)$.

\begin{lem}\label{l:ell.bounds}
For all $\alpha\in(\albd,\aubd)$, the function $\ell(A)$ of \eqref{e:def.ell.A} is sandwiched between
	\begin{align*}
	\ell_\Inn(A)
	&\equiv
	\int_{|z|\le9}
	\f{D_{\sqrt{\psiubd} z}(A)}{1-\qlbd}
	\,\varphi(z)\,dz\,,\\
	\ell_\Out(A)
	&\equiv  \int_{|z|\le9}
	\f{D_{\sqrt{\psilbd}z}(A)}{1-\qubd}
		\,\varphi(z)\,dz
		+ \f{\sgn(A-1)}{10^{15}}\,.
	\end{align*}
where $\sgn(\ell_\Out(A)-\ell_\Inn(A))=\sgn(A-1)$. Moreover we have $|\ell_\Out(A)-\ell_\Inn(A)| \le 1.4/10^{11}$ uniformly over $A$.

\begin{proof}
Recalling \eqref{e:D.A.expand}, we calculate
	\[\f{dD_H(A)}{dm}
	=-\f{2m(\Delta-1)}{\Delta}\,.
	\]
Then, for $H\equiv\psi^{1/2}z$ and $m\equiv\tnh(H)$,
we have
	\[
	\f{dD_{\psi^{1/2} z}(A)}{d\psi}
	=- 
	\f{m(\Delta-1)}{\Delta}
	\f{z\tnh'(\psi^{1/2}z)}{\psi^{1/2}}
	\]
which has the opposite sign as $A-1$. It follows that $\ell$ is sandwiched between $\ell_\Inn$ and $\ell_\Out$ as claimed. Next we claim that $a(x)=x \tnh'(x) =x/(\ch x)^2 $ satisfies $|a(x)| \le 1/2^{1/2}$ uniformly over all real $x$. By symmetry and using that $\ch x\ge1$ for all $x$, it suffices to prove $x/(\ch x) \le 1/2^{1/2}$ for all $x$. The function $\ch x-2^{1/2} x$ is strictly convex, and (by calculus) uniquely minimized at $x = \ash(2^{1/2})$ where the value is strictly positive. It follows from this that we have $|a(x)| \le |x|/(\ch x) \le 1/2^{1/2}$, as claimed. Now recall that for $A\ge1$ we have $\Delta\ge1$, so
	\[
	0\le -\f{dD_{\psi^{1/2} z}(A)}{d\psi}
	\le \f1{\psi} \bigg(1-\f1\Delta\bigg)
	m \cdot \psi^{1/2}z \tnh'(\psi^{1/2}z)
	\le \f1{2^{1/2}\psi}\,.
	\]
Similarly, for $A\le1$ we can use that 
$\Delta$ is sandwiched between $|m|$ and $1$ to conclude
	\[
	0\le \f{dD_{\psi^{1/2} z}(A)}{d\psi}
	= \f{m}{\Delta}(1-\Delta)
	\f{\psi^{1/2}z \tnh'(\psi^{1/2}z)}{\psi}
	\le \f1{2^{1/2}\psi}\,.
	\]
Next we decompose
$\ell_\Out(A)-\ell_\Inn(A) =
\sgn(A-1)/10^{15} + \err_1 + \err_2$ where
	\begin{align*}
	\err_1&= \int_{|z|\le9}
	D_{\sqrt{\psiubd}z}(A)
	\bigg(
	\f1{1-\qubd}-\f1{1-\qlbd}
	\bigg)\,\varphi(z)\,dz\,,\\
	\err_2&=\int_{|z|\le9}
	\f{D_{\sqrt{\psilbd}z}(A)-D_{\sqrt{\psiubd}z}(A)}
		{1-\qubd}\,\varphi(z)\,dz\,.
	\end{align*}
Since $|D_H(A)|\le1$ for all $A$ and all $H$, we have
	\[|\err_1|
	\le\f1{1-\qubd}-\f1{1-\qlbd}
	\le \f{5}{10^{12}}\,.
	\]
By the above bounds on $dD_{\psi^{1/2}z}(A)/d\psi$, we have
	\[
	|\err_2|
	\le \f{\psiubd-\psilbd}{2^{1/2}\psilbd(1-\qubd)}
	\le \f{8}{10^{12}}\,.
	\]
Combining the bounds gives
$|\ell_\Out(A)-\ell_\Inn(A)|
\le 1.4/10^{11}$ as claimed.
\end{proof}
\end{lem}

\noindent We now make another change of variables
	\beq\label{e:A.tau}
	\tau=\f{A-1}{A+1}
	=\tnh\bigg(\f{\log A}{2}\bigg)
	\,,\quad
	A=A(\tau) = \exp(2\atnh(\tau))\,,
	\eeq
and let $\lm(\tau)\equiv\ell(A(\tau))$.
Let $\olm(\tau)\equiv\ell_\Out(A(\tau))$ and
$\ilm(\tau)\equiv\ell_\Inn(A(\tau))$. Then let
$\ulm(\tau)\equiv\max\set{\olm(\tau),\ilm(\tau)}$
and $\ilm(\tau)\equiv\min\set{\olm(\tau),\ilm(\tau)}$.

\section{Grid search bounds}\label{s:grid}

\noindent Recall that $\SE(\lm)$
denotes the sum of the functions
 $\PP(\lm),\AAA(\lm),\HH(\lm)$
defined by \eqref{e:PP.pair.repeat}, \eqref{e:BB.repeat}, and
\eqref{e:def.HH.pair.repeat}. Let
	\beq\label{e:QQ.fixed}
	\QQ(\lm)
	\equiv \f{0.2^2}{2}
	- 0.2\,\psi^{1/2}\sqrt{1-q}
	\bigg(\f{1-\lm}{1+\lm}\bigg)^{1/2} 
	-\f{\psi(1-q)\lm}{1+\lm}
	+\II_s(\lm) - \II(0)\,,
	\eeq
and note that $\PP(\lm)+\AAA(\lm) \le \QQ(\lm)$. Let
	{\setlength{\jot}{0pt}\begin{align*}
	\PGG(\lm)&\equiv \HH(\lm)+\PP(\lm)\,,\\
	\QGG(\lm)&\equiv \HH(\lm)+\QQ(\lm)\,,
	\end{align*}}%
so that $\SE(\lm)\le \min\set{\PGG(\lm),\QGG(\lm)}$ for all $\lm$.

\begin{ppn}\label{p:grid.result}
For all $\alpha\in(\albd,\aubd)$ the following hold:
\begin{enumerate}[a.]
\item\label{p:grid.result.a}
$\PGG(\lm)$ is negative for all $\lm\in[0.2,0.98]$,
and $\QGG(\lm)$ is negative for all 
$\lm\in[-\lmin,-0.125]$,
 
\item \label{p:grid.result.b}
$(\PGG)'(\lm)$ has the opposite sign from $\lm$
for all $\lm \in [-0.125,-0.03] \cup [0.05,0.2]$.

\item \label{p:grid.result.c}
$(\PGG)''(\lm)$ is negative for $\lm\in [-0.03,0.05]$;
\end{enumerate}
As a consequence, $\SE(\lm)\le0$ for all $\lm$, with equality only at $\lm=0$ and $\lm=1$.
\end{ppn}

\subsection{Bounds on function value}

Suppose $\lmin \le \llm \le \ulm\le1$. We will assume that $\llm,\ulm$ have the same sign
(i.e., that they are either both nonnegative or both nonpositive). As a result, if we define $L_\ubd \equiv \max\set{(\llm)^2,(\ulm)^2}$ and  $L_\lbd \equiv \min\set{(\llm)^2,(\ulm)^2}$, we will have $L_\lbd \le \lm^2\le L_\ubd$ for all $\llm\le\lm\le\ulm$. Next, since $c_\lm$ is decreasing in $\lm$, we define
	\[
	c_\lbd\equiv \f{\sqrt{1-\ulm}}{\sqrt{1+\ulm}}\,,\quad
	c_\ubd\equiv \f{\sqrt{1-\llm}}{\sqrt{1+\llm}}\,,
	\]
so that $c_\lbd \le c_\lm\le c_\ubd$ for all $\llm\le\lm\le\ulm$. We abbreviate
	\[
	d(z,x)
	\equiv
	\f{\varphi(\gamma z + x)\varphi(z)}
	{\bPhi(\gamma z)}\,,
	\]
and note that for all $\llm\le\lm\le\ulm$ we have $d_\lbd\le d\le d_\ubd$ where
	\[
	d_\ubd(z,x)\equiv
	d_\ubd(z,x,\llm,\ulm)
	\equiv\begin{cases}
	\varphi(z)\cE(\gamubd z)
	\exp\{- x^2/2 - \gamlbd zx \}
	&\textup{for }z\ge0\,,\\
	\varphi(z)\cE(\gamlbd z)
	\exp\{ -x^2/2 - \gamubd zx \}
	&\textup{for }z\le0\,,\\
	\end{cases}
	\]
and $d_\lbd$ is defined by similar considerations. Likewise, we abbreviate
	\beq\label{e:g.lm}
	g \equiv g_\lm(z,x)
	\equiv \f{\gamma (1-\lm)z-\lm x }{\sqrt{1-\lm^2}}
	= \gamma  c_\lm z - \f{\lm x}{\sqrt{1-\lm^2}}\,,
	\eeq
and note that for all $\llm\le\lm\le\ulm$ we have $g_\lbd \le g \le g_\ubd$ where
	\[g_\ubd \equiv g_\ubd(z,u)
	=\begin{cases}
	\gamubd c_\ubd z
	- \llm x/\sqrt{1-(\llm)^2}
	&\textup{for }z\ge0\,,\\
	\gamlbd c_\lbd z
	- \llm x/\sqrt{1-(\llm)^2}
	&\textup{for }z\le0\,,
	\end{cases}\]
and $g_\lbd$ is defined by similar considerations.
For any fixed $s\ge0$ we shall abbreviate
	\[
	S(z)
	\equiv \f{\cE(\gamma z) s}{\psi^{1/2}\sqrt{1-q}}\,.
	\]
We then have $0\le S_\lbd(z) \le S(z) \le S_\ubd(z)$ where
	\[
	S_\ubd(z)
	=\begin{cases}
	\cE(\gamubd z) s
		/ [\sqrt{\psi_\lbd} \sqrt{1-\qubd}]
		& \textup{for }z\ge0\,,\\
	\cE(\gamlbd z) s
		/ [\sqrt{\psi_\lbd} \sqrt{1-\qubd}]
		& \textup{for }z\le0\,,\\
	\end{cases}
	\]
and $S_\lbd(z)$ is defined by similar considerations.
It follows that for any $s\ge0$,
	\[
	\II_s(\lm)
	\le
	\albd
	\int_{-9}^9\int_0^9
	\log\bPhi(g_\lbd
		- S_\ubd(z)) d_\lbd(z,x)\,du\,dz
	\equiv \II_{s,\ubd}(\llm,\ulm)\,.\]
Abbreviate $\II_\ubd\equiv\II_{0,\ubd}$.
Next, since $\cE$ is an increasing function, we have
	\[
	\f{d\II(0)}{d\gamma}
	=- \alpha\int\cE(\gamma z) z \varphi(z)\,dz
	=-\alpha\int_0^\infty z [\cE(\gamma z)
		-\cE(-\gamma z)]\,\varphi(z)\,dz
	\le0\,,
	\]
from which it follows that
	\[
	\II(0)
	= \alpha\int\log\bPhi(\gamma z)
	\,\varphi(z)\,dz
	\ge\II_\lbd(0)\equiv
	\aubd\int_{-9}^9
	\log\bPhi(\gamubd z)\,
	\varphi(z)\,dz
	-\f1{10^{15}}
	\,.
	\]
Substituting these bounds into \eqref{e:PP.pair.repeat} gives
	\[
	\PP(\lm)
	\le 
	-\f{\psilbd(1-\qubd)\llm}{1+\llm}
	+\II_\ubd(\llm,\ulm)-\II_\ubd(0)
	\equiv 
	\PP_\ubd(\llm,\ulm)
	\,.
	\]
On the other hand,
substituting into \eqref{e:QQ.fixed} gives
	\[
	\QQ(\lm)
	\le
	\f{0.2^2}{2}
	- 0.2\,\sqrt{\psilbd}
	\sqrt{1-\qubd} c_\lbd
	-\f{\psilbd(1-\qubd)\llm}{1+\llm}
	+\II_{0.2,\ubd}(\llm,\ulm)-\II_\ubd(0)
	\equiv \QQ_\ubd(\llm,\ulm)\,.
	\]
We also note that
	\[
	\f{d}{d\psi}
	\ent\bigg(\f{1+\tnh(\psi^{1/2}z)}{2}\bigg)
	=
	\ent'\bigg(\f{1+\tnh(\psi^{1/2}z)}{2}\bigg)
	\f{\tnh'(\psi^{1/2}z)z}{4\psi^{1/2}}\le0\,,
	\]
since for all real $x$ we have $\sgn(\tnh x)=\sgn(x)$,
$\tnh'(x) = 1-\tnh(x)^2\ge0$, and
	\[\sgn \ent'\bigg(\f{1+x}{2}\bigg)
	=-\sgn x\,.
	\]
It follows by recalling \eqref{e:HH.entropy} that 
	\[
	\HH_\star
	\ge\HHlbd\equiv
	\int_{-9}^9 \ent\bigg(
	\f{1+\tnh(\sqrt{\psiubd}z)}{2}
		\bigg)
	\,\varphi(z)\,dz
	\]
Next, recalling that $m\equiv\tnh H$, we calculate
	\[
	\f{\pd\Gamma(H,D)}{\pd m}
	=
	\f12\log\f{(1-m)^2+D}{(1+m)^2+D}
	+\f{m}{2}\log\f
	{(1-m^2-D)^2}
	{((1+m)^2+D)((1-m)^2+D)}\,,
	\]
which has the opposite sign from $m$. 
Next, making use of \eqref{e:entropy.gprime}, we calculate
	\[
	\f{\pd \Gamma(H,D_H(A))}{\pd D}
	\f{dD_H(A)}{dm}
	= 
	\f{\log A}{2} \cdot
	\f{2m(\Delta-1)}{\Delta}
	\]
which has the same sign as $m$. It follows that
	\[\frH(A)\le\frH_\ubd(A) \equiv
	-2\HHlbd
	+ \f1{10^{15}}
	+ \int_{-9}^9
	\Gamma( \sqrt{\psilbd}z
	, D_{\sqrt{\psiubd}z}(A) )
	\,\varphi(z)\,dz
	\]

\begin{proof}[Proof of Proposition~\ref{p:grid.result}\ref{p:grid.result.a}]
Suppose $-1\le\ltau\le\utau\le1$ where
$\ltau,\utau$ have the same sign.
Let $\otau$ (resp.\ $\itau$) be the one of $\ltau,\utau$ which is larger (resp.\ smaller) in magnitude.
Recalling Lemma~\ref{l:ell.bounds}, let
 $\llm\equiv\llm(\ltau)$
and $\ulm=\ulm(\utau)$. It follows from the above that
for all $\ltau\le\tau\le\utau$,
	\[
	\PGG(\lm(\tau))
	\le\bPGGubd(\ltau,\utau)
	\equiv\frH_\ubd(A(\itau))
	+\PP_\ubd(\llm,\ulm)\,.
	\]
Let us abbreviate $\mathbf{s}=\llbracket\ltau,\utau,\eta\rrbracket$ for the vector with entries
	\[
	\mathbf{s}_j
	= \ltau + \eta(j-1)\,,\quad
	1\le j\le \bigg\lceil \f{\utau-\ltau}{\eta}
		\bigg\rceil\,.
	\]
For any vectors $\mathbf{s}',\mathbf{s}''$ we write 
$(\mathbf{s}',\mathbf{s}'')$ for their concatenation.
Then let
	{\setlength{\jot}{0pt}\begin{align*}
	\mathbf{t}
	&=(
	\llbracket0.24,0.284,0.001\rrbracket,
	\llbracket0.285,0.315,0.002\rrbracket,
	\llbracket0.318,0.342,0.003\rrbracket,
	\llbracket0.346,0.366,0.004\rrbracket,\\
	&\qquad\llbracket0.371,0.386,0.005\rrbracket,
	\llbracket0.392,0.404,0.006\rrbracket,
	(0.411,0.418,0.425,0.433,0.441),\\
	&\qquad\llbracket0.45,0.57,0.01\rrbracket,
	\llbracket0.59,0.67,0.02\rrbracket,
	\llbracket0.7,0.76,0.03\rrbracket,
	\llbracket0.8,0.94,0.04\rrbracket,
	(0.95,0.98,0.99))\,.
	\end{align*}}%
We find by a numerical integration package that
$\bPGGubd(\mathbf{t}_i,\mathbf{t}_{i+1})$
is negative for all $i$, which implies that $\PGG(\lm(\tau))$ is negative
for all $0.24\le\tau\le0.99$. We then also verify that $\ulm(0.24) < 0.2$ while $\llm(0.99)>0.98$,
so we can conclude that $\PGG(\lm)$
is negative for all $0.2 \le \lm \le 0.98$, as claimed.
On the other hand, for
	{\setlength{\jot}{0pt}\begin{align*}
	\mathbf{t}
	&=(
	\llbracket0.18, 0.209, 0.001\rrbracket,
	\llbracket0.21, 0.236, 0.002\rrbracket,
	\llbracket0.238, 0.268, 0.003\rrbracket,\\
	&\qquad\llbracket0.271, 0.343, 0.004\rrbracket,
	\llbracket0.347, 0.419, 0.006\rrbracket,
	\llbracket0.425, 0.513, 0.008\rrbracket,\\
	&\qquad
	\llbracket0.52, 0.77, .01\rrbracket,
	(0.78,0.8,0.82,0.84,0.86,0.89,0.93,1)
	)
	\end{align*}}%
we find that $\bQGGubd(-\mathbf{t}_{i+1},-\mathbf{t}_i)$ is negative for all $i$, so
$\QGG(\lm(\tau))$ is negative
for all $-1\le\tau\le-0.18$.
We then verify that
$\llm(-0.18)>-0.125$, so we can conclude that
$\QGG(\lm)$ is negative for all $\lmin\le\lm\le -0.125$.
\end{proof}

\subsection{First derivative bounds}

The first derivative of $\II_s(\lm)$
with respect to $\lm$ is given by
	\beq\label{e:dII.dlambda}
	(\II_s)'(\lm)
	= \f{\alpha}{1-\lm^2}
	\int_{-\infty}^\infty\int_0^\infty
	\cE \bigg(
	\f{(1-\lm)\gamma z-\lm x}{\sqrt{1-\lm^2}}
	-S(z)
	\bigg)
	\,
	\f{(1-\lm)\gamma z+x}{\sqrt{1-\lm^2}}
	\,
	\f{\varphi(\gamma z+x)}
		{\bPhi(\gamma z)}
	\,\varphi(z)\,dx\,dz\,,
	\eeq
which we shall bound uniformly over all $\albd\le\alpha\le\aubd$ and all $\llm\le\lm\le\ulm$. Let us decompose
	\beq\label{e:dII.dlambda.threeparts}
	(\II_s)'(\lm)
	= \sum_{i=1}^3(\II_s)'(\lm)_i
	\eeq
where $(\II_s)'(\lm)_1$ is the contribution to \eqref{e:dII.dlambda} obtained by integrating over $\set{z\ge0,x\ge0}$ (where the integrand is nonnegative); $(\II_s)'(\lm)_2$ is the contribution from $\set{z\le0,0\le x\le -\gamma z(1-\lm)}$ (where the integrand is nonpositive); and  $(\II_s)'(\lm)_3$ is the contribution from $\set{z\le0, x \ge -\gamma z(1-\lm)}$ (where the integrand is nonnegative). Let
	\beq\label{e:dII.dlambda.abs}
	\II'(\lm)_\textup{abs}
	\equiv 
	\f{\alpha}{1-\lm^2}
	\int_{-\infty}^\infty\int_0^\infty
	\cE \bigg(
	\f{(1-\lm)\gamma z-\lm x}{\sqrt{1-\lm^2}}
	\bigg)
	\,
	\bigg|
	\f{(1-\lm)\gamma z+x}{\sqrt{1-\lm^2}}\bigg|
	\,
	\f{\varphi(\gamma z+x)}
		{\bPhi(\gamma z)}
	\,\varphi(z)\,dx\,dz\,.
	\eeq
and note that $|(\II_s)'(\lm)|\le \II'(\lm)_\textup{abs}$ for any $s\ge0$.

\begin{lem}\label{l:dII.dlambda.truncate}
For the decomposition \eqref{e:dII.dlambda.threeparts}
and $s\ge0$ the following hold:
\begin{enumerate}[a.]
\item\label{l:dII.dlambda.truncate.a}
For $0\le \lm\le 0.95$, the contribution to 
$(\II_s)'(\lm)_1$ from $\set{z\ge6.5
\textup{ or } x\ge 8\sqrt{1-\lm^2}}$ is at most $1/10^6$.
\item\label{l:dII.dlambda.truncate.b}
For general $\lm$, the contribution to
$-(\II_s)'(\lm)_2$ from $\set{z\le-3.3}$
is at most $(1.1/10^7)/\sqrt{1-\lm^2}$.

\item\label{l:dII.dlambda.truncate.c}
 For $0\le \lm\le 0.95$, the contribution to 
$(\II_s)'(\lm)_3$ from $\set{z\le-3.3
\textup{ or } x \ge 9\sqrt{1-\lm^2}}$ is at most $2/10^6$.

\end{enumerate}

\begin{proof}
We will use repeatedly that $\cE(y)\le 1+|y|$ for all real $y$, and $\cE(y) \le 2\varphi(y)$ for $y\le 0$. For $x\ge0$,
	\beq\label{e:dzx.prelim.bound}
	d(z,x)
	\equiv
	\f{\varphi(\gamma z+x)\varphi(z)}{\bPhi(\gamma z)}
	\le \begin{cases}
	\varphi(z)\cdot
	\sqrt{2\pi}(1+\gamma z)\varphi(x)
		&\textup{if }z\ge0\,,\\
	\varphi(z)\cdot 2\varphi(\gamma z+x)
		&\textup{if }z\le0\,.
	\end{cases}
	\eeq
Then, for $0\le\lm\le0.95$,
the contribution to \eqref{e:dII.dlambda.abs} from 
$\set{z\ge 6.5,x\ge0}$ is upper bounded by
	\begin{align}
	\label{e:dII.dlambda.large.zpos}
	&\f{\alpha \sqrt{2\pi}}{1-\lm^2}
	\int_{6.5}^\infty
	\cE(c_\lm \gamma z)
	\cE(\gamma z)\varphi(z)
	\int_0^\infty
	\bigg (c_\lm \gamma z + 
	\f{x}{\sqrt{1-\lm^2}}
	\bigg)
	\varphi(x)\,dx\,dz \\
	&\le
	\f{\alpha \sqrt{2\pi}}{1-\lm^2}
	\int_{6.5}^\infty
	(1+\gamma z)^2\varphi(z)
	\bigg( \f{\gamma z}{2} + 
	\f{1/\sqrt{2\pi}}{\sqrt{1-\lm^2}}
	\bigg)
	\,dz\le \f{6.5}{10^7}\,.\nonumber
	\end{align}
For general $\lm$, the contribution to \eqref{e:dII.dlambda.abs} from 
$\set{z\le -3.3, 0\le x\le -(1-\lm)\gamma z/\lm}$ is at most
	\begin{align}\nonumber
	&\le \f{4\alpha}{1-\lm^2}
	\int_{-\infty}^{3.3}
	\int_0^\infty
	\varphi\bigg(
	c_\lm \gamma z - \f{\lm x}{\sqrt{1-\lm^2}}
	\bigg)
	\bigg|
	\f{(1-\lm)\gamma z+x}{\sqrt{1-\lm^2}}
	\bigg|
	\varphi(\gamma z+x)\,\varphi(z)\,dx\,dz \\
	\nonumber
	&=\f{4\alpha}{1-\lm^2}
	\int_{-\infty}^{3.3}
	\varphi(\gamma z)\varphi(z)
	\int_0^\infty
	\bigg|
	\f{(1-\lm)\gamma z+x}{\sqrt{1-\lm^2}}
	\bigg|
	\varphi\bigg(
	 \f{(1-\lm)\gamma z+x}{\sqrt{1-\lm^2}} \bigg)
	\,dx\,dz \\
	&\le \f{4\alpha}
		{\sqrt{1-\lm^2}}
	\int_{-\infty}^{-3.3} 
	\varphi(\gamma z)\varphi(z)
	\int_{-\infty}^\infty |u|\varphi(u)\,du
	 \,dz
	\le \f{1.1/10^7}{\sqrt{1-\lm^2}}
	\stackrel{\star}{\le} 
	\f{4}{10^7}\,,
	\label{e:dII.dlambda.large.zneg}
	\end{align}
where (above and throughout the rest of this proof) the inequality marked $\star$ holds provided  $0\le \lm\le 0.95$. Next, for $\lm\ge0$, the contribution to \eqref{e:dII.dlambda.abs} from the set $\set{-3.3\le z\le 0, x \ge 9\sqrt{1-\lm^2}}$ is at most
(cf.~\eqref{e:dII.dlambda.large.zneg})
	\begin{align}\nonumber
	&\f{4\alpha}{\sqrt{1-\lm^2}}
	\int_{-3.3}^0
	\varphi(z)
	\int_9^\infty
	|c_\lm \gamma z+u|
	\varphi(c_\lm \gamma z-\lm u)
	\varphi(\gamma z+\sqrt{1-\lm^2}u)
	\,du\,dz \\
	&=
	\f{4\alpha}{\sqrt{1-\lm^2}}
	\int_{-3.3}^0\varphi(\gamma z)\varphi(z)
	\int_9^\infty
	|c_\lm \gamma z+u|
	\varphi(c_\lm \gamma z+u)
	\,du\,dz \nonumber\\
	&\le \f{4\alpha}{\sqrt{1-\lm^2}}
	\int_{-\infty}^0\varphi(\gamma z)\varphi(z)
	\int_{5.2}^\infty u\varphi(u)\,du\,dz
	\le
	\f{2.4/10^7}{\sqrt{1-\lm^2}}
	\stackrel{\star}{\le}\f8{10^7}\,.
	\label{e:dII.dlambda.zneg.large.x}
	\end{align}
For $\lm\ge0.6$, on the set $\set{0\le z\le6.5, x \ge 8\sqrt{1-\lm^2}}$ we have
 $c_\lm \gamma z -\lm u \le0$, as well as
$\lm c_\lm \gamma z \le 2.3$. 
Therefore the resulting contribution
from this set to \eqref{e:dII.dlambda.abs} is upper bounded by (cf.~\eqref{e:dII.dlambda.large.zneg}~and~\eqref{e:dII.dlambda.zneg.large.x})
	\begin{align}\nonumber
	&\f{2\alpha \sqrt{2\pi}}{\sqrt{1-\lm^2}}
	\int_0^{6.5}
	\cE(\gamma z)\varphi(z)
	\int_8^\infty
	(c_\lm \gamma z + u)
	\varphi(c_\lm \gamma z - \lm u)
		\varphi(\sqrt{1-\lm^2}u)\,du\,dz\\
	\nonumber
	&= \f{2\alpha \sqrt{2\pi}}{\sqrt{1-\lm^2}}
	\int_0^{6.5}
	\cE(\gamma z)\varphi(z)
	\varphi(\gamma z(1-\lm))
	\int_8^\infty
	(c_\lm \gamma z + u)
	\varphi(u - \lm c_\lm \gamma z)\,du\,dz \\
	&\le \f{2\alpha}{\sqrt{1-\lm^2}}
	\int_0^\infty (1+\gamma z)
	\varphi(z)
	\int_{5.7}^\infty
	(c_\lm \gamma z(1+\lm)+ u)
	\varphi(u)\,du \,dz
	\le \f{8/10^8}{\sqrt{1-\lm^2}}
	\stackrel{\star}{\le} \f3{10^7}\,.
	\label{e:dII.dlambda.lmlarge}
	\end{align}
Finally, for $0\le\lm\le0.6$, the contribution to \eqref{e:dII.dlambda.abs} from  $\set{0\le z\le 6.5, x \ge 8\sqrt{1-\lm^2}}$ is at most (cf.~\eqref{e:dII.dlambda.large.zpos}) 
	\begin{align}\nonumber
	&\le
	\f{\alpha\sqrt{2\pi}}
	{1-\lm^2}\int_0^\infty
	\cE(c_\lm\gamma z)
	\cE(\gamma z) \varphi(z)
	\int_{5.7}^\infty
	\bigg( c_\lm \gamma z +
	\f{x}{\sqrt{1-\lm^2}}
	\bigg)
	\varphi(x)
	\,dx\,dz\\
	&\le
	\f{\alpha\sqrt{2\pi}}
	{1-\lm^2}\int_0^\infty
	(1+\gamma z)^2 \varphi(z)
	\bigg\{
	\f{\gamma z}{10^{10}}
	+ \f{6/10^{10}}{\sqrt{1-\lm^2}}
	\bigg\}\,dz
	\le
	\f{7/10^{10}}{1-\lm^2}
	+\f{3/10^9}{(1-\lm^2)^{3/2}}
	\le \f6{10^9}\,.
	\label{e:dII.dlambda.lmsmall}
 	\end{align}
Part~\ref{l:dII.dlambda.truncate.a} follows by combining
\eqref{e:dII.dlambda.large.zpos}, \eqref{e:dII.dlambda.lmlarge}, and \eqref{e:dII.dlambda.lmsmall}.
Part~\ref{l:dII.dlambda.truncate.b} follows directly from \eqref{e:dII.dlambda.large.zneg}.
Finally, part~\ref{l:dII.dlambda.truncate.c} follows
by combining
\eqref{e:dII.dlambda.large.zneg}
and \eqref{e:dII.dlambda.zneg.large.x}. This concludes the proof.
\end{proof}
\end{lem}

\noindent For $\llm\le\ulm$,
let $\KK_{s,\ubd,i}\equiv \KK_{s,\ubd,i}(\llm,\ulm)$
for $i=1,2,3$ be defined by
	\begin{align*}
	\KK_{s,\ubd,1}
	&\equiv
	\f{\aubd}{\sqrt{1-L_\ubd}}
	\int_0^{6.5} \int_0^8
	(c_\ubd \gamubd z + u)
	\f{\cE( c_\ubd \gamubd z - \llm u-S_\lbd(z))
	\varphi(\gamlbd z + \sqrt{1-L_\ubd}u)
	\varphi(z)}
	{\bPhi(\gamubd z)}\,du\,dz\,,\\
	\KK_{s,\ubd,2}
	&\equiv -
	\f{\albd c_\lbd (\gamlbd)^2}
		{1+\ulm}
		\int_{-3.3}^0 \int_0^1
	z^2(1-u)
	\f{\cE(c_\ubd \gamubd (1+\ulm u)z
	-S_\ubd(z))
	\varphi(\gamubd z(1-(1-\ulm)u))\varphi(z)}
		{\bPhi(\gamubd z)}
		\,du\,dz\,,\\
	\KK_{s,\ubd,3}
	&\equiv
	\f{\aubd/\sqrt{2\pi}}
		{\sqrt{1-L_\ubd}}
	\int_{-3.3}^0\int_0^9
	u
	\f{\cE( c_\lbd \gamlbd (1+\llm)z
		-\llm u -S_\lbd(z)) 
	\varphi(z)}
		{
		\bPhi(\gamlbd z)
		\exp \{ \f12(
		L_\lbd( \gamlbd z)^2
			+ u^2(1-L_\ubd)
		+zu \gamubd \ulm \sqrt{1-L_\lbd}
		)
		\}
		}\,du\,dz\,.
	\end{align*}
Define $\KK_{s,\lbd,i}\equiv \KK_{s,\lbd,i}(\llm,\ulm)$ similarly
by exchanging the appearances of $\lbd$ and $\ubd$ in the above expressions. We then have the following:

\begin{cor}\label{c:dII.dlambda.bounds}
For the decomposition \eqref{e:dII.dlambda.threeparts}, $s\ge0$, and $\llm\le\lm\le\ulm$, we have
	\begin{alignat*}{2}
	\sum_{i=1}^3\II_s(\lm)_i
	&\ge 
	\KK_{s,\lbd}(\llm,\ulm)
	\equiv
	\sum_{i=1}^3\KK_{s,\lbd,i}(\llm,\ulm)
	-\f{1.1/10^7}{\sqrt{1-L_\ubd}}\quad
	&&\textup{for }
	\lmin\le\llm\le\ulm\le1
	\,,\\
	\sum_{i=1}^3\II_s(\lm)_i
	&\le\KK_{s,\ubd}(\llm,\ulm)
	\equiv
	\sum_{i=1}^3\KK_{s,\ubd,i}(\llm,\ulm)
	+ \f{3}{10^6}\quad
	&&\textup{for }
	0\le \llm\le\ulm\le 0.95\,,
	\end{alignat*}
provided that $\llm,\ulm$ have the same sign.

\begin{proof}
The expression for $\KK_{s,\ubd,1}$ is derived by making the change of variables $x=\sqrt{1-\lm^2}u$
for $u\ge0$. For $\KK_{s,\ubd,2}$ we instead make the change of variables $x = -\gamma z(1-\lm)u$ for $0\le u\le 1$. Lastly, for $\KK_{s,\ubd,2}$ we make the change of variables $x =  -\gamma z(1-\lm)+u$ for $u\ge0$. The lower bound then follows from Lemma~\ref{l:dII.dlambda.truncate} part~\ref{l:dII.dlambda.truncate.b}, while the upper bound follows by 
Lemma~\ref{l:dII.dlambda.truncate}
parts~\ref{l:dII.dlambda.truncate.a}~and~\ref{l:dII.dlambda.truncate.c}.
\end{proof}
\end{cor}

\noindent
Abbreviate $\KK_\lbd\equiv\KK_{0,\lbd}$ and 
 $\KK_\ubd\equiv\KK_{0,\ubd}$. Substituting 
the result of Corollary~\ref{c:dII.dlambda.bounds}
into \eqref{e:PP.pair.repeat} gives that
	\begin{alignat*}{2}
	\PP'(\lm)
	&\ge
	(\PP')_\lbd(\llm,\ulm)
	\equiv
	-\f{\psi_\ubd(1-\qlbd)}{(1+\llm)^2}
	+ \KK_\lbd(\llm,\ulm)\quad
	&&\textup{for }
	\lmin\le\llm\le\ulm\le1
	\,,\\
	\PP'(\lm)
	&\le
	(\PP')_\ubd(\llm,\ulm)
	\equiv
	 -\f{\psi_\lbd(1-\qubd)}{(1+\ulm)^2}
	+ \KK_\lbd(\llm,\ulm)\quad
	&&\textup{for }
	0\le \llm\le\ulm\le 0.95\,,
	\end{alignat*}
for all $\llm\le\lm\le\ulm$, provided $\llm$ and $\ulm$ have the same sign. We also have
	\[-\f{(1-\qlbd)\log A(\utau)}{2}
	\le \HH'(\lm(\tau))
	= -\f{(1-q)\log A(\tau)}{2}
	\le 
	-\f{(1-\qubd)\log A(\ltau)}{2}
	\,.
	\]
for all  $\ltau\le\tau\le\utau$, again provided
$\ltau$ and $\utau$ have the same sign.

\begin{proof}[Proof of Proposition~\ref{p:grid.result}\ref{p:grid.result.b}]
Take $\ltau,\utau,\llm,\ulm$ as in the proof of 
Proposition~\ref{p:grid.result}\ref{p:grid.result.a}.
Then
	\begin{alignat*}{2}
	(\PGG)'(\lm(\tau))
	&\ge (d\bPGG)_\lbd(\ltau,\ltau)
	\equiv
	-\f{(1-q)\log A(\tau)}{2}
	+(\PP')_\lbd(\llm,\ulm)\quad
	&&\textup{for }
	-1\le\ltau\le\tau\le\utau\le1\,,\\
	(\PGG)'(\lm(\tau))
	&\le
	 (d\bPGG)_\ubd
	(\ltau,\ltau)
	\equiv -\f{(1-\qubd)\log A(\ltau)}{2}+
	(\PP')_\ubd(\llm,\ulm)\quad
	&&\textup{for }
	0\le\ltau\le\tau\le\utau\le0.95\,,
	\end{alignat*}
using that $\ulm(0.95)<0.95$. Let
	{\setlength{\jot}{0pt}\begin{align*}
	\mathbf{t}
	&\equiv(
	\llbracket0.06,0.076,0.001\rrbracket,
	\llbracket0.078,0.098,0.002\rrbracket,
	\llbracket0.101,0.116,0.003\rrbracket,
	\llbracket0.12,0.14,0.004\rrbracket,\\
	&\qquad(0.145,0.15,0.156,0.162,0.168),
	\llbracket0.175,0.21,0.007\rrbracket,
	\llbracket0.21,0.26,0.01\rrbracket
	)\,.
	\end{align*}}%
We find by a numerical integration package that
$(d\bPGG)_\ubd(\mathbf{t}_i,\mathbf{t}_{i+1})$
is negative for all  $i$. This implies that
$(\PGG)'(\lm(\tau))$ is negative for all 
$0.06\le\tau\le0.26$. We then also verify that
$\llm(0.26)>0.2$
while $\ulm(0.06)<0.05$, so we conclude that
$(\PGG)'(\lm)$ is negative for all $0.05\le\lm\le0.2$.
On the other hand, for 
	{\setlength{\jot}{0pt}\begin{align*}
	\mathbf{t}
	&\equiv
	(0.03,0.031,0.032,0.033,0.034,0.036,0.038,0.041,
	0.045,0.05,0.056,0.063,0.071
   ,\\
   	&\qquad 0.081,0.092,0.104,0.116,0.128,0.14,0.15,
	0.159,0.166,0.172,0.177,0.18,0.185,0.19)
	\end{align*}}%
we find that
$(d\bPGG)_\lbd(-\mathbf{t}_{i+1},-\mathbf{t}_i)$
is positive for all $i$, so that
$(\PGG)'(\lm(\tau))$
is positive for all 
$-0.19 \le \tau \le -0.03$.
We then verify that 
$\ulm(-0.19)<-0.125$
while
$\llm(-0.03)>-0.03$,
so we conclude that $(\PGG)'(\lm)$ is positive for all 
$-0.125\le\lm\le-0.03$.
\end{proof}

\subsection{Second derivative bounds}
As before, let $\lmin \le \llm \le \ulm\le1$
where $\llm,\ulm$ are either both nonnegative or both nonpositive. Recall from \eqref{e:g.lm} the definition of $g\equiv g_\lm(z,x)$.  The first and second derivative of $g$
 with respect to $\lm$ are given by
	\[
	\f{\pd g}{\pd\lm}
	=-\f{\gamma z(1-\lm)+x}{(1-\lm^2)^{3/2}}\,,\quad
	\f{\pd^2 g}{\pd\lm^2}
	= \f{\gamma z(1-\lm)(1-2\lm) - 3\lm x}
		{(1-\lm^2)^{5/2}}\,.
	\]
Then, recalling that $\II(\lm)\equiv \II_0(\lm)$, we have
	\beq\label{e:ddII.ddlambda}
	\II''(\lm)
	= -\alpha\int\int
	\bigg\{
	\cE(g)
	\f{\gamma z(1-\lm)(1-2\lm) - 3\lm x}
		{(1-\lm^2)^{5/2}}
	+\cE'(g) 
	\bigg(\f{\gamma z(1-\lm)+x}
		{(1-\lm^2)^{3/2}}\bigg)^2
	\bigg\}
	\,d(z,x)\,dx\,dz\,.
	\eeq
Writing $x_+$ for the positive part of $x$, we shall also consider
	\[\WW(\lm)
	\equiv
	\alpha\int_{-\infty}^\infty \int_0^\infty
	\cE(g)
	\bigg(
	\f{3\lm x-\gamma z(1-\lm)(1-2\lm)}
		{(1-\lm^2)^{5/2}}\bigg)_+
	d(z,x)
	\,dx\,dz\,.
	\]

\begin{lem}\label{l:second.derivative.truncate}
For $|\lm|\le0.1$, the total contribution to $\WW(\lm)$
from
the complement of 
	{\setlength{\jot}{0pt}\begin{align}\nonumber
	&\set{0\le z\le 6.5,0\le x\le 6.5}\\
	&\qquad\cup
	\set{-5\le z\le0,0\le x\le11}
	\label{e:secondderiv.truncate}
	\end{align}}%
is upper bounded by $1/10^6$.

\begin{proof}Using $\cE(y)\le1+|y|$, we have
for all $|\lm|\le0.1$ and all real-valued $z,x$ that
	\begin{align*}
	p(z,x)&\equiv\cE(g) 
	\bigg|
	\f{3\lm x-\gamma z(1-\lm)(1-2\lm)}
		{(1-\lm^2)^{5/2}}\bigg|
	\le
	\bigg(1+
	c_\lm \gamma |z| 
	+ \f{|\lm|x}{\sqrt{1-\lm^2}}\bigg)
	\bigg( 0.4|x| + 1.6|z|
	\bigg) 
	\\
	&\le
	\bigg(1+ 1.3 |z| 
	+ 0.11|x|\bigg)
	\bigg( 0.4|x| + 1.6|z|
	\bigg) 
	\,.\end{align*}
By restricting the values of $z$ or $x$ we obtain slightly simpler bounds
	\[
	p(z,x)\le
	\begin{cases}
	(0.5+2.3|z|+2.1z^2)x^2
	&\textup{if }|x|\ge1\,,\\
	(3.7+1.1|x|+0.05x^2)z^2
	&\textup{if }|z|\ge1\,.\\
	\end{cases}
	\]
Applying \eqref{e:dzx.prelim.bound},
the contribution to $\WW(\lm)$
from $\set{z\ge0,x\ge 6.5}$ is at most
	\[\sqrt{2\pi}
	\alpha 
	\int_0^\infty
	(1+\gamma z)\varphi(z)
	\bigg(0.5+2.3|z|+2.1z^2\bigg)
	\int_{6.5}^\infty
	x^2\,\varphi(x)\,dx\,dz
	\le \f3{10^8}\,,
	\]
for all $|\lm|\le0.1$. Similarly, the contribution from $\set{z\ge 6.5,x\ge0}$ is at most
	\[
	\sqrt{2\pi}\alpha
	\int
	(1+\gamma z)\varphi(z) z^2
	\int_0^\infty
	\bigg(3.7+1.1|x|+0.05x^2\bigg)
	\,\varphi(x)\,dx
	\,dz
	\le \f{8}{10^8}\,,
	\]
again for all $|\lm|\le0.1$.
Now consider $z\le0$.
If $\lm\le0$, then the integrand of $\WW(\lm)$ is zero unless
	\[
	0\le x\le 
	\f{\gamma z(1-\lm)(1-2\lm)}{3\lm(1-\lm^2)^{5/2}}
	\le \f{\gamma z(1-\lm)}{2\lm}\,,
	\]
where the last inequality uses 
$-0.1\le\lm\le0$. In this case we obtain
	\[
	g 
	= \f{(1-\lm)\gamma z-\lm x}{\sqrt{1-\lm^2}}
	\le 
	\f{(1-\lm)\gamma z}{2\sqrt{1-\lm^2}} \le0\,.
	\]
The same bound holds if $z\le0$, $\lm\ge0$, $x\ge0$.
Thus the contribution to $\WW(\lm)$
from $\set{z\le -5,x\ge0}$ is at most
	\begin{align*}
	&4\alpha\int_{-\infty}^{-5}
	\varphi(z)
	\varphi\bigg(
	\f{(1-\lm)\gamma z}{2\sqrt{1-\lm^2}}
	\bigg)
	\int_0^\infty
	\bigg(0.4\,x+1.6\,|z|\bigg)
	\varphi(\gamma z+x)\,dx\,dz \\
	&\qquad\le
	4\alpha\int_{-\infty}^{-5}
	\varphi(z)
	\varphi(0.5\,z)
	|z|
	\int_{-\infty}^\infty
	\bigg(0.4(x+\gamma)
	+1.6\bigg)
	\varphi(x)\,dx\,dz
	\le \f{9}{10^8}\,,
	\end{align*}
for all $|\lm|\le0.1$. Similarly, the contribution
to $\WW(\lm)$ from
$\set{-5\le z\le0, x \ge 11}$ is at most
	\begin{align*}
	&4\alpha
	\int_{-5}^0 \varphi(z)
	\varphi(0.5\,z)
	\int_{11}^\infty
	\bigg(0.4\,x+1.6\,|z|\bigg)
	\varphi(\gamma z+x)\,dx\,dz \\
	&\qquad\le4\alpha
	\int_{-5}^0 \varphi(z)
	\varphi(0.5\,z)
	\bigg(0.4(1+ \gamma|z|)  +1.6\,|z|\bigg)
	\int_{5.3}^\infty
	x
	\varphi(x)\,dx\,dz
	\le \f4{10^7}\,.
	\end{align*}
Combining these estimates gives the claimed bound.
\end{proof}
\end{lem}

\noindent 
Suppose $-0.1\le\llm\le\lm\le\ulm\le0.1$
where $\llm,\ulm$ have the same sign.
We now consider the double integral
\eqref{e:ddII.ddlambda} restricted to the region
\eqref{e:secondderiv.truncate}, and decompose it into four terms, which we denote $\II'(\lm)_i$ for $1\le i\le 4$. For each term we  obtain a bound
$\II'(\lm)_i \le \MM_{\ubd,i}\equiv
\MM_{\ubd,i}(\llm,\ulm)$ which holds for all
$\llm\le\lm\le\ulm$. We begin with
	\begin{align*}
	&\II''(\lm)_1
	\equiv - \alpha
	\int_0^{6.5}\int_0^{6.5}
	\gamma z
	\bigg\{
	\f{\cE'(g) 2x(1-\lm)}{(1-\lm^2)^3}
	+\cE(g) \f{(1-\lm)(1-2\lm)}{(1-\lm^2)^{5/2}}
	\bigg\}
	\,d(z,x)\,dx\,dz \\
	&\quad\le
	-\albd\int_0^{6.5}\int_0^{6.5}
	\gamlbd z
	\bigg\{
	\f{\cE'(g_\lbd) 2x(1-\ulm)}
		{(1-(\ulm)^2)^3}
	+\cE(g_\lbd) \f{(1-\ulm)
		(1-2\ulm)}{(1-(\ulm)^2)^{5/2}}
	\bigg\}
	\,d_\lbd(z,x)\,dx\,dz
	\equiv\MM_{\ubd,1}\,.
	\end{align*}
Next we let $\II''(\lm)_2$ be defined as $\II''(\lm)_1$, but integrating over the region $\set{-5\le z\le0,0\le x\le 11}$. Then
	\[\II''(\lm)_2
	\le -\aubd\int_{-5}^0\int_0^{11}
	\gamubd z
	\bigg\{
	\f{\cE'(g_\ubd) 2x(1-\llm)}{(1-(\llm)^2)^3}
	+\cE(g_\ubd)
	\f{(1-\llm)(1-2\llm)}{(1-(\llm)^2)^{5/2}}
	\bigg\}
	\,d_\ubd(z,x)\,dx\,dz
	\equiv\MM_{\ubd,2}\]
by similar considerations. Next we have
	\begin{align*}
	\II''(\lm)_3
	&\equiv -\alpha
	\int_{-5}^{6.5}\int_0^{11}
	\cE'(g) 
	\f{(\gamma z)^2 (1-\lm)^2 + x^2}{(1-\lm^2)^3}
	\,d(z,x)\,dx\,dz\\
	&\le -\albd 
	\int_{-5}^{6.5}\int_0^{11}
	\cE(g_\lbd)
	\bigg\{
	\f{(\gamlbd z)^2}{(1-\ulm)(1+\ulm)^3}
	+ \f{x^2}{(1-L_\lbd)^3}
	\bigg\}
	\,d_\lbd(z,x)\,dx\,dz
	\equiv \MM_{\ubd,3}\,.
	\end{align*}
The last term is given by
	\[\II''(\lm)_4
	= \f{3\alpha\lm}{(1-\lm^2)^{5/2}}
	\int_{-5}^{6.5}\int_0^{11}
	\cE(g) x \, d(z,x)\,dx\,dz\]
which we bound according to the sign of $\lm$:
	\[\II''(\lm)_4
	\le \MM_{\ubd,4}
	\equiv
	\begin{cases}\displaystyle
	\f{3\aubd\ulm}{(1-(\ulm)^2)^{5/2}}
	\int_{-5}^{6.5}\int_0^{11}
	\cE(g_\ubd) x \, d_\ubd(z,x)\,dx\,dz
	&\textup{for }\lm\ge0\,,\\
	\displaystyle
	\f{3\albd\ulm}{(1-(\ulm)^2)^{5/2}}
	\int_{-5}^{6.5}\int_0^{11}
	\cE(g_\lbd) x \, d_\lbd(z,x)\,dx\,dz
	&\textup{for }\lm\le0\,.
	\end{cases}\]
Combining these bounds and recalling Lemma~\ref{l:second.derivative.truncate} gives altogether
	\[\II''(\lm)\le \sum_{i=1}^4
	\MM_{\ubd,i}(\llm,\ulm)
	+ \f1{10^6}\equiv \MM_\ubd(\llm,\ulm)\,,\]
and substituting into \eqref{e:PP.pair.repeat} gives
	\beq\label{e:PP.second.derivative.ubd}
	\PP''(\lm)
	\le\f{2\psiubd(1-\qlbd)}{(1+\llm)^3}+
	\MM_\ubd(\llm,\ulm)\eeq
for all $\llm\le\lm\le\ulm$. We also have from \eqref{e:dHH.dlambda} that
	\beq\label{e:ddH.ddlambda.repeat}
	\HH''(\ell(A))
	=-\f{(1-q) A'(\ell(A))}{2 A}
	=-\f{(1-q) }{2 A \ell'(A)}\,.
	\eeq
Recall $H\equiv \psi^{1/2}z$ and $m\equiv \tnh H$.
It follows from
\eqref{e:def.ell.A.repeat} and
\eqref{e:deriv.D.A} that
	\[
	\ell'(A)
	= \int
	\f{(D_H)'(A)}{1-q}\,\varphi(z)\,dz
	= \int
	\f{2A(1-m^2)^2}{(1-q)\Delta(\Delta+1)^2}
	\,\varphi(z)\,dz\,,
	\]
where $\Delta\equiv\Delta(m,A)\equiv \sqrt{A^2(1-m)^2+m^2}$. Note for all $0\le A_\lbd\le A$ the quantity $\Delta(m,A)$ is lower bounded by $\Delta(m,A_\lbd)$. It follows that
	\[
	\ell'(A)
	\le\f1{10^{15}}+
	\int_{-9}^9
	\f{2A_\ubd (1-(\tnh y)^2)^2}
	{(1-\qubd)
	\Delta(\tnh y,A_\lbd)
	(\Delta(\tnh y,A_\lbd)
	+1)^2
	}\,\varphi\bigg(\f{y}{\sqrt{\psiubd}}\bigg)\,
	\f{dy}{\sqrt{\psilbd}}
	\equiv
	\LL_\ubd(A_\lbd,A_\ubd)\,.
	\]
Substituting into 
\eqref{e:ddH.ddlambda.repeat} gives
	\[
	\HH''(\ell(A))
	\le
	 -\f{(1-\qubd)/(2A_\ubd)}
		{\LL_\ubd(A_\lbd,
		A_\ubd)}
	\]
for all $A_\lbd \le A \le A_\ubd$. 

\begin{proof}[Proof of Proposition~\ref{p:grid.result}\ref{p:grid.result.c}]
Take $\ltau,\utau,\llm,\ulm$ as in the proof of 
Proposition~\ref{p:grid.result}
parts~\ref{p:grid.result.a}~and~\ref{p:grid.result.b}.
If we assume further that
$-0.1\le\llm\le\ulm\le0.1$, then it follows from the above calculations that
	\[
	(\PGG)''(\lm(\tau))
	\le (d^2\bPGG)_\ubd(\ltau,\utau)
	\equiv
	-\f{(1-\qubd)/(2A(\utau))}
		{\LL_\ubd(A(\ltau),
		A(\utau))}+
	\f{2\psiubd(1-\qlbd)}
		{(1+\llm)^3}
		+\MM_\ubd(\llm,\ulm)
	\]
for all $\ltau\le\tau\le\utau$. We then define
	{\setlength{\jot}{0pt}\begin{align*}
	\mathbf{t}
	&\equiv
	(-0.043,-0.039,-0.035,-0.03,-0.025,
	-0.019,-0.013,-0.007,\\
	&\qquad 0,0.007,0.015,0.024,
   0.033,0.043,0.054,0.066,0.078)
	\end{align*}}%
and evaluate
$(d^2\bPGG)_\ubd(\mathbf{t}_i,\mathbf{t}_{i+1})$ to be negative for all $i$, so that
$(\PGG)''(\lm(\tau))$ is negative for all
$-0.043 \le \tau \le 0.078$.
We then verify that 
$\ulm(-0.043) < -0.03$
and $\llm(0.078) > 0.05$,
so the claim is proved.
\end{proof}
 
\noindent Condition~\ref{c:G} follows by combining Propositions~\ref{p:HH.PP.near.lambda.one}~and~\ref{p:grid.result}.

\section{Estimates for special functions}\label{s:specialfns}

\begin{lem}\label{l:EE}
The function $\cE(x)\equiv\varphi(x)/\bPhi(x)$ satisfies the following for all $x\in\R$:
\begin{enumerate}[a.]
\item\label{l:EE.a} $\max\set{0,x}<\cE(x)<1+|x|
\equiv \bar{e}(|x|)$;
\item\label{l:EE.b} $\cE'(x)=\cE(x)(\cE(x)-x)\in(0,1)$;
\item\label{l:EE.c} $\cE''(x)=\cE(x)[(2\cE(x)-x)(\cE(x)-x)-1]\in(0,1)$.
\item\label{l:EE.d} $\cE^{(3)}(x)= -2(1-\cE'(x))\cE'(x)+(2\cE(x)-x)\cE''(x) \in (-1/2,13)$.
\end{enumerate}

\begin{proof}
Clearly $\cE$ is positive for all real $x$.
For $x\ge0$, a well-known gaussian tail bound  gives
	\beq\label{e:gaussian.tail}
	x < \cE(x)
	< x+\f1x\,,
	\eeq
so we obtain $\cE(x)>\max\set{0,x}$. The expressions for $\cE'(x)$ and $\cE''(x)$ are by direct calculation, and the bounds are obtained as follows. First, $\cE(x)>\max\set{0,x}$ implies $\cE'(x)>0$. Next, a gaussian random variable conditioned to be at least $x$ has mean $\cE(x)$ and variance $\usf{var}_x=1-\cE(x)(\cE(x)-x)$ (which must be positive). It follows that $\cE'(x)=1-\usf{var}_x<1$, so part~\ref{l:EE.b} is proved. The bound $\cE''(x)>0$ is proved by \cite{MR0054890}. On the other hand, the bound on $\cE'(x)$, together with the upper bound in \eqref{e:gaussian.tail}, gives for all $x\ge0$ that
	\beq\label{e:dEE.prelimubd}
	\cE''(x)
	= -\bigg(1-\cE'(x)\bigg)\cE(x)
	+\bigg(\cE(x)-x\bigg)\cE'(x)
	< \bigg(\cE(x)-x\bigg)\cE'(x)
	< \cE(x)-x < \f1x\,.
	\eeq
Since $\cE'<1$, the function $\cE(x)-x$
is decreasing in $x$, so for all $x\ge -0.45$ we have
	\[\cE''(x)
	\le \cE(x)- x \bigg|_{x=-0.45}
	< 1\,.
	\]
On the other hand, for all $x\le -0.45$, we have
	\beq\label{e:dEE.negative.bound}
	\cE''(x)
	< \bigg(\cE(x)-x\bigg)\cE'(x)
	= \bigg(\cE(x)-x\bigg)^2\cE(x)
	\le\f{(\cE(-0.45)+|x|)^2
		\varphi(x)}{\bPhi(-0.45)}
	\,.\eeq
We have by calculus that for any $k\ge0$,
	\beq\label{e:xk.varphi.sup.bound}
	\sup_{x\in\R} |x|^k\varphi(x)
	=
	\f1{\sqrt{2\pi}}
	\exp\bigg\{
	\sup_{x\ge0} \bigg(k\log x-\f{x^2}{2} \bigg)
	\bigg\}
	\le k\varphi(k^{1/2})\,.
	\eeq
Substituting \eqref{e:xk.varphi.sup.bound}
into \eqref{e:dEE.negative.bound} gives
$\cE''(x) \le 1$ for all $x\le -0.45$,
so part~\ref{l:EE.c} is proved. Next, for $x\ge1$, the upper bound of 
\eqref{e:gaussian.tail} gives $\cE(x) \le x+1$. On the other hand,
	{\setlength{\jot}{0pt}\begin{align*}
	\textup{for $x\le 0.3$,}\quad
	& \cE(x) \le \cE(0.3) < 1 \le |x|+1\,;\\
	\textup{for $0.3\le x\le 0.7$,}\quad
	& \cE(x) \le \cE(0.7) < 1.3 \le |x|+1\,;\\
	\textup{for $0.7\le x\le1$,}\quad
	& \cE(x) \le \cE(1) < 1.6 \le |x|+1\,.
	\end{align*}}%
This proves part~\ref{l:EE.a}. For part~\ref{l:EE.d} we calculate
	\[
	\cE^{(3)}(x)
	= -2\bigg(1-\cE'(x)\bigg)\cE'(x)
	+ \bigg( 2\cE(x)-x \bigg)\cE''(x)\,.\]
It follows from parts~\ref{l:EE.a}, \ref{l:EE.b}, and \ref{l:EE.c} that for all real $x$,
	\beq\label{e:dEEthird.sw}
	-\f12 \le 
	-2\bigg(1-\cE'(x)\bigg)\cE'(x)
	\le \cE^{(3)}(x)
	\le \bigg( 2\cE(x)-x \bigg)\cE''(x)\,.\eeq
For $x\le 0.42$, by similar considerations as for
\eqref{e:dEE.negative.bound}, and using \eqref{e:xk.varphi.sup.bound},
we find
	\beq\label{e:dEE.negative.bound.one}
	0 \le \bigg( 2\cE(x)-x \bigg)\cE''(x)
	\le \f{\varphi(x)(\cE(0.42)+|x|)^2}{\bPhi(0.42)}
	\bigg(2\cE(0.42)+|x|\bigg)
	\le 12\,.
	\eeq
On the other hand, for all $x\ge0$, we have from \eqref{e:gaussian.tail} that $0 < \cE(x)-x < 1/x$, so $0<2\cE(x)-x < x+2/x$ for all $x\ge0$. Combining with \eqref{e:dEE.prelimubd} gives
	\[
	0 < \cE''(x)
	\bigg(2\cE(x)-x\bigg)
	< \f1x\bigg(x+\f2x\bigg)
	= 1 + \f{2}{x^2}\,,
	\]
which is at most $13$ for all $x\ge 0.42$.
Combining with \eqref{e:dEE.negative.bound.one}
we see that the right-hand side of
\eqref{e:dEEthird.sw} is at most $13$ for all real $x$,
so part~\ref{l:EE.d} is proved.
\end{proof}\end{lem}

\begin{lem} 
It holds for all real $x$ that $|\cE^{(4)}(x)| \le 196$.

\begin{proof}
Taking Lemma~\ref{l:EE}\ref{l:EE.d}
and differentiating again gives
	\[
	\cE^{(4)}(x)
	=\cE''(x)\bigg(
	6\cE'(x)-3
	\bigg)
	+\cE^{(3)}(x)\bigg( 2\cE(x)-x\bigg)\,.\]
It follows from Lemma~\ref{l:EE}
and \eqref{e:dEEthird.sw} that for all real $x$ we have
	\beq\label{e:dEEfourth}
	-3 <-3 \cE''(x) <
	\cE^{(4)}(x) < 3\cE''(x)
	+ 
	\bigg( 2\cE(x)-x\bigg)
	\cE^{(3)}(x)
	< 3
	+ \bigg( 2\cE(x)-x\bigg)^2 \cE''(x)\,.
	\eeq
Next, similarly to \eqref{e:dEE.negative.bound} and \eqref{e:dEE.negative.bound.one}, we have for all $x\le1$ that
	\begin{align*}
	\bigg(2\cE(x)-x\bigg)^2\cE''(x)
	&\le \f{\varphi(x)}{\bPhi(1)}
	(\cE(1)+|x|)^2
	(2\cE(1)+|x|)^2
	\le193\,,
	\end{align*}
so the right-hand side of \eqref{e:dEEfourth} is at most $196$ for all $x\le1$. We next appeal to an improvement on \eqref{e:gaussian.tail} (see e.g.\ \cite{MR0167642}) which says that for all $x\ge0$,
	\beq\label{e:improved.gaussian.tail}
	x < x + \f1{x+2/x} < \cE(x)
	< x + \f1{x+2/(x+3/x)}
	< x+\f1x\,.
	\eeq
It follows from this that for all $x\ge0$ we have
	\[
	0 < u(x)
	\equiv x+\f1x-\cE(x) 
	< \f1x - \f1{x+2/x}
	= \f{2}{x(2+x^2)}
	\le
	\min\bigg\{ \f1x, \f{2}{x^3}\bigg\}\,.
	\]
Our previous calculation of $\cE'(x)$ can be rewritten in terms of $u(x)$ as
	\beq\label{e:dEE.tail.bound}
	0 < 1-\cE'(x)
	= \f{ x^3 u(x) - (1-x u(x))^2 }{x^2}
	\le x u(x) 
	<\min\bigg\{1, \f{2}{x^2}\bigg\}\,.
	\eeq
Next, it also follows from \eqref{e:improved.gaussian.tail} that for all $x\ge0$
we have
	\[0 < \tilde{u}(x) \equiv \cE(x)
	-\bigg(x + \f1{x+2/x}\bigg)
	< \f1{x+2/(x+3/x)} - \f1{x+2/x}
	= \f{6}{x(2+x^2)(5+x^2)}
	< \f{6}{x^5}\,.
	\]
Our previous calculation of $\cE''(x)$ can be rewritten as 
	\beq\label{e:dEEtwo.rewrite}
	\cE''(x)
	=\f{(x^3+3x+a(x))
	(b(x)+c(x))}{(x^2+2)^3}
	\eeq
where $a(x)\equiv(2+x^2) \tilde{u}(x)$,
$b(x)\equiv x^5 \tilde{u}(x)-4$, and
	\[c(x)\equiv
	2  \bigg(
	(3+2x^2)2x
	+ (2+x^2)^2 \tilde{u}(x)\bigg) \tilde{u}(x)\,.\]
For $x\ge1$, the bounds on $\tilde{u}(x)$ imply 
$0<a(x)<6/[x(5+x^2)]<1$, $-4< b(x) <2$, and
	\[0<c(x)
	< \f{24(1+2x^2)(6+6x+x^4)}{x^2(2+x^2)(5+x^2)^2}
	< \f{24(1+2x^2)}{x^2(5+x^2)}
	\le 12\,.\]
Substituting the bounds on $a(x),b(x),c(x)$ into \eqref{e:dEEtwo.rewrite} gives,
for all $x\ge1$,
	\beq\label{e:dEEtwo.tail.bound}
	0< \cE''(x)
	< \f{14 (x^3+3x+ 1) }{(x^2+2)^3}
	= \f{14  }{(x^2+2)^2}
	\bigg( x+ \f{x+ 1}{x^2+2}\bigg )
	\le \f{14 (x+2/3) }{(x^2+2)^2}
	\le \f{15}{(x^2+2)x}\,.
	\eeq
Substituting \eqref{e:dEE.tail.bound}~and~\eqref{e:dEEtwo.tail.bound} into
\eqref{e:dEEthird.sw} gives, for all $x\ge1$,
	\beq\label{e:dEEthree.tail.bound}
	-\min\bigg\{ \f12, \f{4}{x^2}\bigg\}
	\le 
	\cE^{(3)}(x)
	\le \bigg(2\cE(x)-x\bigg)\cE''(x)
	\le \f{15(x+2/x)}{(x^2+2)x}
	\le \f{15}{x^2}\,.
	\eeq
Substituting 
\eqref{e:dEEtwo.tail.bound}
and \eqref{e:dEEthree.tail.bound}
into \eqref{e:dEEfourth} gives, for all $x\ge1$,
	\[
	-\f{45}{(x^2+2)x}
	\le 
	\cE^{(4)}(x)
	\le \f{15}{x} \bigg( \f{3}{(x^2+2)}
	+ \f{(x^2+2)}{x^2}\bigg)
	\le \f{60}x \le 60\,,
	\]
concluding the proof.
\end{proof}
\end{lem}

\noindent Throughout the following we write
	\[\E_\xi f(\nu)
	\equiv \int f(\nu)\varphi_\xi(\nu)\,d\nu\,,
	\]
so for instance $\E_\xi\nu=\cE(\xi)$.

\begin{lem}\label{l:E.xi} For all $\xi\in\R$ we have:
\begin{enumerate}[a.]
\item\label{l:E.xi.one} $\E_\xi|\nu| \le 1+|\xi|$;
\item\label{l:E.xi.two} $\E_\xi(\nu^2)=\xi\cE(\xi)+1 \le \xi^2+|\xi|+1$;
\item\label{l:E.xi.three} $\E_\xi|\nu|^3\le (2+\xi^2)(1+|\xi|)$.
\end{enumerate}

\begin{proof} The claim for $\E_\xi(\nu^2)$ follows by direct calculation together with the bound from Lemma~\ref{l:EE}\ref{l:EE.a}. If $\xi\ge0$ then $\E_\xi|\nu|=\E_\xi\nu=\cE(\xi)$, so again the claim follows by Lemma~\ref{l:EE}\ref{l:EE.a}. If $\xi\le0$ then
	\[
	\E_\xi|\nu|
	= \f1{\bPhi(\xi)}
	\bigg(2\int_0^{|\xi|} \nu\varphi(\nu)\,d\nu
	+\int_{|\xi|}^\infty \nu\varphi(\nu)\,d\nu\bigg)
	= \f{2\varphi(0)-\varphi(\xi)}{\bPhi(\xi)}\,,
	\]
which we show is less than one for all $\xi\le0$:
	{\setlength{\jot}{0pt}\begin{align*}
	\textup{for $\xi\le -1$,}\quad
	&\E_\xi|\nu|
	\le 2\varphi(0)/\bPhi(\xi)
	\le 2\varphi(0)/\bPhi(-1)<1\,,\\
	\textup{for $-1\le\xi\le-0.32$,}\quad
	&\E_\xi|\nu|
	\le 2\varphi(0)/\bPhi(-0.32) -\cE(-1) <1\,,\\
	\textup{for $-0.32\le\xi\le0$,}\quad
	&\E_\xi|\nu|
	\le 2\varphi(0)/\bPhi(0)-\cE(-.32)<1\,.
	\end{align*}}%
The claim for $\E_\xi|\nu|$ follows. Similarly, if $\xi\ge0$ then $\E_\xi|\nu|^3=(2+\xi^2)\cE(\xi)$, and the claimed bound follows by combining with Lemma~\ref{l:EE}\ref{l:EE.a}. If $\xi\le0$ then
	\[\E_\xi|\nu|^3
	=\f1{\bPhi(\xi)}
	\bigg(2\int_0^{|\xi|} \nu^3\varphi(\nu)\,d\nu
	+\int_{|\xi|}^\infty \nu^3 \varphi(\nu)\,d\nu\bigg)
	=\f{4\varphi(0) - (\xi^2+2)\varphi(\xi)}{\bPhi(\xi)}
	\]
which we show is less than two for all $\xi\le0$:
	{\setlength{\jot}{0pt}\begin{align*}
	\textup{for $\xi\le -0.9$,}\quad
	&\E_\xi|\nu|^3
	\le 4\varphi(0)/\bPhi(\xi)
	\le 4\varphi(0)/\bPhi(-0.9) < 2\,,\\
	\textup{for $-0.9\le\xi\le-0.3$,}\quad
	&\E_\xi|\nu|^3
	\le 4\varphi(0)/\bPhi(-0.3)
		-((-0.3)^2+2)\cE(-0.9)<2\,,\\
	\textup{for $-0.3\le\xi\le0$,}\quad
	&\E_\xi|\nu|^3
	\le 4\varphi(0)/\bPhi(0)
		-2\cE(-0.3) < 2\,.
	\end{align*}}%
The claim for $\E_\xi|\nu|^3$ follows. \end{proof}
\end{lem}

\vfill {\raggedright \bibliographystyle{alphaabbr}\bibliography{refs}}

\end{document}